%% file: KL_HilbertSpace_December_2018.tex
\RequirePackage{fix-cm}
\documentclass[smallextended]{svjour3}       
\smartqed  
\usepackage{graphicx}

\usepackage{amsfonts,amsmath, amssymb, amstext, color}
\usepackage{graphicx} 

\usepackage{mathrsfs}
\usepackage{comment}

\input style.tex

\newcommand{\pr}{{\rm pr}}

\newcommand{\post}{{\rm post}}

\newcommand{\KL}{{\rm KL}}

\newcommand{\Rrm}{{\rm R}}
\newcommand{\Brm}{{\rm B}}
\newcommand{\Hrm}{{\rm H}}

\newtheorem{assumption}{Assumption}

\title{Regularized divergences between covariance operators and Gaussian measures on Hilbert spaces}
\titlerunning{Regularized divergences between Gaussian measures on Hilbert spaces}   
\author{H\`a Quang Minh}
\institute{
	\at
	RIKEN Center for Advanced Intelligence Project, 1-4-1 Nihonbashi, Chuo-ku,
 Tokyo 103-0027, JAPAN
 \\
\email{minh.haquang@riken.jp}
}

\begin{document}
\maketitle

\begin{abstract}
	This work presents an infinite-dimensional generalization of the correspondence between
	the Kullback-Leibler and R\'enyi divergences between Gaussian measures on Euclidean  space 
	and the Alpha Log-Determinant divergences between symmetric, positive definite matrices.
	Specifically, we present the regularized  Kullback-Leibler  and R\'enyi divergences between covariance operators and Gaussian measures on an infinite-dimensional Hilbert space, which are defined using the infinite-dimensional Alpha Log-Determinant divergences between positive definite trace class operators.  We show that, as the regularization parameter approaches zero, the regularized Kullback-Leibler and R\'enyi divergences between two equivalent Gaussian measures on a Hilbert space converge to the corresponding true divergences. The explicit formulas for the divergences involved are  presented in the most general Gaussian setting.
\end{abstract}





\section{Introduction}

This work is concerned with the correspondence between divergences between covariance operators
and the corresponding Gaussian measures on an infinite-dimensional Hilbert space. Specifically, we study the correspondence between
the infinite-dimensional Alpha Log-Determinant (Log-Det) divergences between covariance operators on a Hilbert space $\H$ and  the Kullback-Leibler and R\'enyi divergences, together with related quantities, between Gaussian measures on $\H$.

In the finite-dimensional setting, let $\Sym^{++}(n)$ denote the set of symmetric, positive definite (SPD) matrices. Then a divergence on $\Sym^{++}(n)$ correspond to a divergence on the set
of zero-mean Gaussian measures on $\R^n$ with strictly positive covariance matrices.
In particular, the Alpha Log-Det divergences \cite{Chebbi:2012Means}
on $\Sym^{++}(n)$ correspond to the Kullback-Leibler  and R\'enyi divergences between zero-mean Gaussian measures
on $\R^n$.

The infinite-dimensional generalization of the finite-dimensional setting requires substantially more
mathematical machinery. It is not straightforward, for instance, to define Log-Determinant divergences 
between covariance operators on an infinite-dimensional Hilbert space $\H$, which are trace class operators, thus have vanishing eigenvalues and therefore unbounded inverses and principal logarithms. In \cite{Minh:LogDet2016},
the author generalized the Alpha Log-Det divergences on $\Sym^{++}(n)$
to the set of positive definite trace class operators on $\H$
of the form $A+\gamma I > 0$, where $A$ is trace class, $\gamma \in \R, \gamma > 0$,
and $I$ is the identity operator.
This was subsequently generalized to the infinite-dimensional Alpha-Beta Log-Det divergences between positive definite trace class operators
\cite{Minh:LogDet2016-AB} and on the more general set of positive definite Hilbert-Schmidt operators 
\cite{Minh:GSI2017}.
Other distance functions on the set of  positive definite Hilbert-Schmidt operators 
include the affine-invariant Riemannian distance \cite{Larotonda:2007}\cite{Minh:GSI2015} and the Log-Hilbert-Schmidt distance \cite{MinhSB:NIPS2014}.

For a fixed $\gamma > 0$, each of the above divergence/distance functions automatically
becomes a divergence/distance function between covariance operators on $\H$.
In particular, for covariance operators on reproducing kernel Hilbert spaces (RKHS),
they all admit closed form expressions that can readily be employed in practical applications, 
see e.g.
\cite{MinhSB:NIPS2014,Minh:CVPR2016,Minh:Covariance2017}.
In computer vision and pattern recognition, other papers 
employing
this approach
include 
in \cite{ProbDistance:PAMI2006} and \cite{Covariance:CVPR2014}, in which 
Bregman divergences between RKHS covariance operators 
are applied to problems in object recognition and texture classification, among others.


It is not clear, however, how all of the above functions relate to the divergence/distance functions
between Gaussian measures on the Hilbert space $\H$, such as the Kullback-Leibler or R\'enyi divergences, as is the case in the finite-dimensional setting.
The aim of this work is to establish these correspondences
in the case of the infinite-dimensional Alpha Log-Det divergences.


{\bf Contributions}. The following are the main contributions of the current work.

\begin{enumerate}
	\item We study regularized versions of the Kullback-Leibler and R\'enyi divergences
	between covariance operators and Gaussian measures on Hilbert spaces, using
	the infinite-dimensional Alpha Log-Det divergences. We show that
	for two {\it equivalent} Gaussian measures on $\H$,
	the regularized Kullback-Leibler and R\'enyi divergences 
	converge to the corresponding true Kullback-Leibler and R\'enyi divergences, respectively,
	as the regularization parameter $\gamma \approach 0$.
	
	\item As part of the proof, we derive the explicit formulas for the Radon-Nikodym derivative and 
	the true Kullback-Leibler and R\'enyi divergences between two equivalent Gaussian measures
	$\Ncal(m,C)$, $\Ncal(m_0,C_0)$ on $\H$, under the most general setting. These formulas generalize
	those available in the current literature, which assume either $C_0 = C$ or
	$m_0 = m = 0$. We illustrate this with the computation of the Kullback-Leibler divergence
	between the posterior and prior probability measures, under the Gaussian setting, in a Bayesian inverse problem on Hilbert spaces.
\end{enumerate}

\textbf{Organization}. 
The paper is structured as follows.
In Section \ref{section:definitions}, we present the definitions of the regularized 
divergences between covariance operators and Gaussian measures on $\H$, using the
Alpha Log-Det divergences. Section \ref{section:main-results} summarizes
the main results on the convergence of the regularized divergences to the true divergences.
The proofs for the convergence are given in Sections 
\ref{section:limit-KL} and \ref{section:limiting-Renyi}.
In Section \ref{section:exact-divergences}, we present
the explicit formulas for the Radon-Nikodym derivative and the true Kullback-Leibler and 
R\'enyi divergences between two equivalent Gaussian measures on $\H$.

{\bf Notation}.
Throughout the paper, we assume that $\H$ is a real separable Hilbert space, with $\dim(\H) = \infty$, unless explicitly stated otherwise.
Let $\L(\H)$ be the Banach space of bounded linear operators on $\H$, with operator norm $||\;||$.
Let $\Sym(\H) \subset \L(\H)$ denote the subspace of bounded, self-adjoint operators on $\H$. Let $\Sym^{+}(\H) \subset \Sym(\H)$ denote the set of self-adjoint, {\it positive} operators on $\H$, that is $A \in \Sym^{+}(\H) \equivalent \la x, Ax\ra \geq 0$ $\forall x \in \H$. Let $\Sym^{++}(\H) \subset \Sym^{+}(\H)$ denote the set of self-adjoint, {\it strictly positive} operators on $\H$, that is $\A \in \Sym^{++}(\H)
\equivalent \la x, Ax\ra > 0$ $\forall x \in \H, x\neq 0 $, or equivalently, $\ker(A) = \{0\}$.

\section{Main definitions}
\label{section:definitions}

We first present the definitions of the key concepts involved in the paper, namely the infinite-dimensional Alpha Log-Determinant divergences and the corresponding regularized divergences between Gaussian measures on Hilbert spaces.  Many of these concepts 
were first introduced in \cite{Minh:LogDet2016}.

\subsection{Infinite-dimensional Alpha Log-Det divergences between positive definite trace-class operators}
In \cite{Minh:LogDet2016}, we introduced the following infinite-dimensional divergences between positive definite trace class operators on a Hilbert space $\H$, which generalize the 
Alpha Log-Determinant divergences between SPD matrices
\cite{Chebbi:2012Means}.

\begin{definition}
	[\textbf{Alpha Log-Determinant divergences between positive definite trace class operators}]
	\label{def:logdet}
	Assume that {$\dim(\H) = \infty$}. For {$-1 < \alpha < 1$},
	the Log-Det {$\alpha$}-divergence 
	$d^{\alpha}_{\logdet}[(A+\gamma I), (B+ \mu I)]$ between  $(A+\gamma I) > 0, (B+\mu I) > 0$,
	$A, B \in \Tr(\H)$, $\gamma,\mu \in \R$, 
	is defined to be
	{
		\begin{align}\label{equation:det1}
		&d^{\alpha}_{\logdet}
		[(A+\gamma I), (B+ \mu I)]
		\nonumber
		\\
		& = \frac{4}{1-\alpha^2}  
		\log\left[\frac{\detX\left(\frac{1-\alpha}{2}(A+\gamma I) + \frac{1+\alpha}{2}(B+\mu I)\right)}{\detX(A+\gamma I)^{\beta}\detX(B + \mu I)^{1-\beta}}\left(\frac{\gamma}{\mu}\right)^{\beta - \frac{1-\alpha}{2}}\right],
		\end{align}
	}
	where {$\beta = \frac{(1-\alpha)\gamma}{(1-\alpha) \gamma + (1+\alpha)\mu}$}.
The limiting cases $\alpha \approach \pm 1$ are defined by
\begin{align}
&
d^{1}_{\logdet}[(A+\gamma I), (B+\mu I)] = \left(\frac{\gamma}{\mu}-1\right)\log\frac{\gamma}{\mu} 
\nonumber
\\
&\;\;\;\;\;+ \trX[(B+\mu I)^{-1}(A+\gamma I) - I] - \frac{\gamma}{\mu}\log\detX[(B+\mu I)^{-1}(A+\gamma I)].
\label{equation:alpha+1}
\\
&
d^{-1}_{\logdet}[(A+\gamma I), (B+\mu I)] = 
\left(\frac{\mu}{\gamma}-1\right)\log\frac{\mu}{\gamma} 
\nonumber
\\
&+ \trX\left[(A+\gamma I)^{-1}(B+\mu I) - I\right] - \frac{\mu}{\gamma}\log\detX[(A+\gamma I)^{-1}(B+\mu I)].
\label{equation:alpha-1}
\end{align}
\end{definition}

In Definition \ref{def:logdet}, $\detX$ denotes the extended Fredholm determinant
defined via $\detX(A+\gamma I) = \gamma \det[(A/\gamma) + I])$,
for $A \in \Tr(\H), \gamma \in \R, \gamma \neq 0$, with $\det$ being the Fredholm determinant.
Likewise, $\trX$ denotes the extended trace, defined by $\trX(A+\gamma I) = \trace(A) + \gamma$
(see \cite{Minh:LogDet2016} for the motivations leading to these concepts).
 
In the case $\gamma = \mu$, $d^{\alpha}_{\logdet}[(A+\gamma I), (B+ \gamma I)]$ assumes a much simpler form, which directly generalizes the finite-dimensional formulas in \cite{Chebbi:2012Means}, as follows.
\begin{align}
\label{equation:det2}
d^{\alpha}_{\logdet}
[(A+\gamma I), (B+ \gamma I)]
& = \frac{4}{1-\alpha^2}  
\log\left[\frac{\detX\left(\frac{1-\alpha}{2}(A+\gamma I) + \frac{1+\alpha}{2}(B+\gamma I)\right)}{\detX(A+\gamma I)^{\frac{1-\alpha}{2}}\detX(B + \gamma  I)^{\frac{1+\alpha}{2}}}\right],
\\
d^{1}_{\logdet}[(A+\gamma I), (B+\gamma I)] &= 
\trX[(B+\gamma I)^{-1}(A+\gamma I) - I]
\nonumber 
\\
&- \log\detX[(B+\gamma I)^{-1}(A+\gamma I)].
\\
d^{-1}_{\logdet}[(A+\gamma I), (B+\gamma I)] &= 
\trX\left[(A+\gamma I)^{-1}(B+\gamma I) - I\right]
\nonumber
\\
& - \log\detX[(A+\gamma I)^{-1}(B+\gamma I)].
\end{align}
The finite-dimensional formulas are obtained by letting $A,B \in \Sym^{++}(n)$ 
and $\gamma = 0$.


From the above formulation, the following result is immediate.

\begin{theorem}
	[\textbf{Regularized divergences between covariance operators and zero-mean Gaussian measures on Hilbert spaces}]
\label{theorem:divergence-Gaussian-mean-0}
Let $-1 \leq \alpha \leq 1$ be fixed.
For each fixed $\gamma \in \R$, $\gamma > 0$, the following is a divergence on the set
$\Sym^{+}(\H) \cap \Tr(\H)$ of self-adjoint, positive trace class operators on $\H$
\begin{align}
D^{\gamma}_{\alpha}(A,B) = d^{\alpha}_{\logdet}[(A+\gamma I), (B+ \gamma I)], \;\;\; A,B \in \Sym^{+}(\H) \cap \Tr(\H).
\end{align}
Consequently, the following is a divergence on the set of Gaussian measures
on $\H$ with mean zero and  covariance operators $C_1,C_2 \in \Sym^{+}(\H) \cap \Tr(\H)$
\begin{align}
D^{\gamma}_{\alpha}[\Ncal(0, C_1), \Ncal(0,C_2)] = d^{\alpha}_{\logdet}[(C_1+\gamma I), (C_2+ \gamma I)].
\end{align}

\end{theorem}

\subsection{Regularized divergences between general Gaussian measures on Hilbert spaces}
We next consider divergences between Gaussian measures on Hilbert spaces without the zero-mean condition.
Motivated by the explicit formulas for the divergences between Gaussian densities in $\R^n$,  
in \cite{Minh:LogDet2016} we introduced the following
{\it regularized divergences} between Gaussian measures on Hilbert spaces, using the infinite-dimensional Log-Det divergences above.

\begin{definition}
[\textbf{Regularized Kullback-Leibler divergences between Gaussian measures on Hilbert spaces}]
\label{definition:regularized-KL}
Let $\Ncal(m_1,C_1)$ and $\Ncal(m_2,C_2)$ be 
two Gaussian measures on $\H$, with corresponding mean vectors $m_1,m_2 \in \H$  
and
covariance operators $C_1, C_2 \in \Sym^{+}(\H) \cap \Tr(\H)$. For any 
fixed $\gamma \in \R$, 
$\gamma > 0$, 
the regularized Kullback-Leibler divergence, denoted by $D^{\gamma}_{\KL}(\Ncal(m_1, C_1)|| \Ncal(m_2, C_2))$, 
is defined to be
\begin{align}
\label{equation:regularized-KL}
D^{\gamma}_{\KL}(\Ncal(m_1, C_1)|| \Ncal(m_2, C_2)) 
&=\frac{1}{2}\la m_1 - m_2, (C_2 + \gamma I)^{-1}(m_1 - m_2)\ra 
\nonumber
\\
&+ \frac{1}{2}d^1_{\logdet}[(C_1 + \gamma I), (C_2 + \gamma I)].
\end{align}
\end{definition}

\begin{definition}
[\textbf{Regularized R\'enyi divergences between Gaussian measures on Hilbert spaces}]
\label{definition:regularized-Renyi}
For two Gaussian measures $\Ncal(m_1,C_1)$ and $\Ncal(m_2,C_2)$ on $\H$,
the regularized R\'enyi divergence of order $r$, $0 < r < 1$, for a 
fixed $\gamma \in \R$, 
$\gamma > 0$,
denoted by $D^{\gamma}_{\Rrm,r}(\Ncal(m_1,C_1) ||\Ncal(m_2,C_2))$,
is defined to be 
\begin{align}
\label{equation:regularized-Renyi-infinite}
&D^{\gamma}_{\Rrm,r}(\Ncal(m_1,C_1) ||\Ncal(m_2,C_2)) 
\nonumber
\\
&\;\; = \frac{1}{2}\la m_1 - m_2, [(1-r)(C_1 + \gamma I) + r(C_2 + \gamma I)]^{-1}(m_1 - m_2)\ra
\nonumber
\\
&\;\; + \frac{1}{2}d^{2r-1}_{\logdet}[(C_1 + \gamma I), (C_2 + \gamma I)].
\end{align}
\end{definition}
{\bf Remark}. Our definition of the regularized R\'enyi divergence differs from 
that in \cite{Minh:LogDet2016} by a factor of $\frac{1}{r}$. It is motivated
from the finite-dimensional definition $d_{R,r}(P_1,P_2) = -\frac{1}{r(1-r)}\log\int_{\R^n}P_1^r(x)
P_2^{1-r}(x)dx$, see e.g. \cite{Pardo:2005}, of the R\'enyi divergence between two probability densities $P_1,P_2$ on $\R^n$
. This differs from the original definition by R\'enyi \cite{Renyi:1961}, namely $d_{R,r}(P_1,P_2) = -\frac{1}{(1-r)}\log\int_{\R^n}P_1^r(x)
P_2^{1-r}(x)dx$ by the factor $\frac{1}{r}$.
The advantage of the current formulation is that one can see immediately that
\begin{align}
\lim_{r \approach 1}D^{\gamma}_{\Rrm,r}(\Ncal(m_1,C_1) ||\Ncal(m_2,C_2)) = D_{\KL}^{\gamma}(\Ncal(m_1,C_1) ||\Ncal(m_2,C_2)),
\\
\lim_{r \approach 0}D^{\gamma}_{\Rrm,r}(\Ncal(m_1,C_1) ||\Ncal(m_2,C_2)) = D_{\KL}^{\gamma}(\Ncal(m_2,C_2) ||\Ncal(m_1,C_1)).
\end{align}
\begin{definition}
[\textbf{Regularized Bhattacharyya and Hellinger distances between Gaussian measures on Hilbert spaces}]
For two Gaussian measures $\Ncal(m_1,C_1)$ and $\Ncal(m_2,C_2)$ on $\H$,
the regularized Bhattacharyya 
distance $D_B^{\gamma}(\Ncal(m_1, C_1),\Ncal(m_2,C_2))$, for a fixed $\gamma \in \R, \gamma > 0$, is defined to be
\begin{align}
&D_B^{\gamma}(\Ncal(m_1, C_1),\Ncal(m_2,C_2))
\\
&= \frac{1}{8}\la (m_1 - m_2), \left(\frac{(C_1 +\gamma I) + (C_2+\gamma I)}{2}\right)^{-1}(m_1 - m_2)\ra 
\nonumber
\\
&+ \frac{1}{8}d^{0}_{\logdet}[(C_1 + \gamma I), (C_2 + \gamma I)] 
= \frac{1}{4}D^{\gamma}_{R,1/2}(\Ncal(m_1, C_1),\Ncal(m_2,C_2)).
\nonumber
\end{align}
The regularized Hellinger distance $D_H^{\gamma}(\Ncal(m_1, C_1),\Ncal(m_2,C_2))$
is defined via the regularized Bhattacharyya $D_B^{\gamma}(\Ncal(m_1, C_1),\Ncal(m_2,C_2))$
distance 
by 
\begin{align}
D_H^{\gamma} = \sqrt{2[1-\exp(-D_B^{\gamma})]}.
\end{align}
\end{definition}

{\bf Properties of the regularized divergences}.
\begin{enumerate}
	\item The regularized divergences between any pair of covariance operators, not necessarily strictly positive, are always well-defined and finite  for any $\gamma > 0$. Likewise, the regularized divergences
	between the corresponding Gaussian measures, not necessarily non-degenerate or equivalent (see below), are always well-defined and finite for any $\gamma > 0$.
	
	\item The regularized divergences between Gaussian measures are defined explicitly in terms of 
	their mean vectors and covariance operators,
	not via the evaluation of the 
	Radon-Nikodym derivatives and the corresponding integrals.
	
	\item In the RKHS setting, when the mean vectors and covariance operators are RKHS vectors and covariance operators, respectively, all of these divergences admit closed form formulas that can be efficiently computed
	\cite{Minh:LogDet2016}.
\end{enumerate}

\section{Main theorems}
\label{section:main-results}
The regularized divergences stated above are well-defined for any pairs of Gaussian measures on 
a Hilbert space $\H$. It is not clear from the definition, however,
whether they possess a probabilistic interpretation.
We now show that they are, in fact, closely related to the corresponding
true divergences when the Gaussian measures under consideration are equivalent.
Specifically, the following results state that, as $\gamma \approach 0^{+}$,
the regularized Kullback-Leibler and regularized R\'enyi divergences between two equivalent, non-degenerate Gaussian measures $\Ncal(m_0,C_0)$ and $\Ncal(m,C)$ converge to
the true Kullback-Leibler and R\'enyi divergences, respectively, between  $\Ncal(m_0,C_0)$ and $\Ncal(m,C)$.

\begin{theorem}
[\textbf{Limiting behavior of the regularized Kullback-Leibler divergence}]
\label{theorem:limit-KL-gamma-0}
Let $\mu = \Ncal(m_0, C_0)$ and $\nu = \Ncal(m,C)$ be two non-degenerate, equivalent Gaussian measures on $\H$, that is with $C_0, C \in \Sym^{++}(\H)$. Assume that $\mu$ and $\nu$ are equivalent, that is $m-m_0 \in \Im(C_0^{1/2})$ and there exists $S \in \Sym(\H) \cap\HS(\H)$ such that $C= C_0^{1/2}(I-S)C_0^{1/2}$. Then
\begin{align}
\label{equation:limit-KL-gamma-0-equivalent}
\lim_{\gamma \approach 0^{+}}D_{\KL}^{\gamma}(\nu || \mu) &= \frac{1}{2}||C_0^{-1/2}(m-m_0)||^2 - \frac{1}{2}\log\dettwo(I-S)
\\
& =D_{\KL}(\nu ||\mu),
\label{equation:limit-KL-gamma-0-equivalent-2}
\end{align}
where $D_{\KL}(\nu ||\mu)$ 
denotes the Kullback-Leibler divergence between $\nu$ and $\mu$.
\end{theorem}

In Theorem \ref{theorem:limit-KL-gamma-0}, $\dettwo$ denotes the Hilbert-Carleman determinant
(see e.g. \cite{Simon:1977}).
For a Hilbert-Schmidt operator $A$, the Hilbert-Carleman determinant of $I+A$ is defined by
$\dettwo(I+A) = \det[(I+A)\exp(-A)]$.
In particular, for $A \in \Tr(\H)$,
we have
$\dettwo(I+A) = \det(I+A) \exp(-\trace(A))$,
and
$\log\dettwo(I+A) = \log\det(I+A) - \trace(A)$.
The function $\dettwo(I+A)$ is continuous in the Hilbert-Schmidt norm, so that
$\lim_{k\approach \infty}||A_k - A||_{\HS} = 0 \imply \lim_{k \approach \infty}\dettwo(I+A_k) = \dettwo(I+A)$. 

Theorem \ref{theorem:limit-KL-gamma-0} can also be equivalently stated as
\begin{align}
\lim_{\gamma \approach 0^{+}}D_{\KL}^{\gamma}(\nu || \mu) &= \frac{1}{2}||m-m_0||^2_{C_0} - \frac{1}{2}\log\dettwo(I-S) =D_{\KL}(\nu ||\mu),
\end{align}
where $||\;||_{C_0}$ is the norm corresponding to the inner product
\begin{align}
\la x, y\ra_{C_0} = \la C_0^{-1/2}x, C_0^{-1/2}y\ra, \;\;\; x , y \in \Im(C_0^{1/2})
\end{align}
of the Cameron-Martin space $(\Im(C_0^{1/2}), \la \;, \;\ra_{C_0})$ associated with
$\Ncal(m_0,C_0)$.
\begin{theorem}
	[\textbf{Limiting behavior of the regularized R\'enyi divergences}]
\label{theorem:limit-Renyi-infinite}
Assume the hypothesis of Theorem \ref{theorem:limit-KL-gamma-0}.
Let $D_{\Rrm,r}(\nu ||\mu)$ denote the R\'enyi divergence of order $r$ between $\nu$ and $\mu$
, $0 < r <1$.
Then
\begin{align}
\label{equation:limit-Renyi-infinite}
\lim_{\gamma \approach 0^{+}}D_{\Rrm,r}^{\gamma}(\nu ||\mu) 
&= \frac{1}{2}||(I-(1-r)S)^{-1/2}C_0^{-1/2}(m-m_0)||^2
\\
&+ \frac{1}{2r(1-r)}\log\det[(I-(1-r)S)(I-S)^{r-1}]
\nonumber
\\
&= D_{\Rrm,r}(\nu ||\mu).
\label{equation:limit-Renyi-infinite-2}
\end{align}
\end{theorem}

\begin{corollary}
	[\textbf{Limiting behavior of the regularized Bhattacharyya and Hellinger distances}]
	\label{corollary:limit-Hellinger-infinite}
	Assume the hypothesis of Theorem \ref{theorem:limit-KL-gamma-0}.
	Let $D_B(\nu||\mu)$ denote the true Bhattacharyya distance between $\nu$ and $\mu$. Then
	\begin{align}
	\lim_{\gamma \approach 0^{+}}D_B^{\gamma}(\nu ||\mu) 
	&= \frac{1}{8}||(I-\frac{1}{2}S)^{-1/2}C_0^{-1/2}(m-m_0)||^2
	\\
	&+ \frac{1}{2}\log\det[(I-\frac{1}{2}S)(I-S)^{-1/2}]
	\nonumber
	\\
	& = D_B(\nu ||\mu).
	\end{align}
	Similarly, let $D_H(\nu ||\mu)$ denote the true Hellinger distance between $\nu$ and $\mu$. Then
	\begin{align}
	\lim_{\gamma \approach 0^{+}}D_H^{\gamma}(\nu || \mu) &=
	\sqrt{2}\left[1 -  \frac{\exp\left(-\frac{1}{8}||(I-\frac{1}{2}S)^{-1/2}C_0^{-1/2}(m - m_0)||^2\right)}{\sqrt{\det[(I-\frac{1}{2}S)(I-S)^{-1/2}]}} \right]^{1/2}.
	\\
	& =  D_H(\nu ||\mu). 
	\end{align}
	\end{corollary}


{\bf Computational consequences}. The focus of the current work
is on the statistical interpretation of the infinite-dimensional Alpha Log-Det 
divergences and the corresponding regularized divergences between Gaussian measures on Hilbert spaces. The results just stated also suggest numerical algorithms
for approximating the Kullback-Leibler and R\'enyi divergences between probability measures on infinite-dimensional Hilbert spaces.
This is an important topic, see e.g. \cite{Pinski:2015KL},\cite{Pinski:2015KLalgorithms}, which will be explored in a companion future work.

\subsection{Example: KL divergences in Bayesian inverse problems on Hilbert spaces}
In this section, we apply the concept of regularized KL divergences above to the setting
of linear Bayesian inverse problems. As a specific example, 
consider the following setting from \cite{Stuart:Inverse2010} (Theorem 6.20 and Example 6.23).
Let $u$ be a Gaussian random variable on the Hilbert space $\H$, distributed according
to the Gaussian measure $\mu_0 = \Ncal(m_0,C_0)$, with $\ker(C_0) = \{0\}, m_0 \in \Im(C_0^{1/2})$.
Let $A: \H \mapto \R^n$ be a bounded linear operator. Assume that the following random variable $y \in \R^n$ is Gaussian
\begin{align}
y = Au + \eta, \;\;\; \;\; \;\; \; \eta \sim \Ncal(0, \Gamma), \Gamma \in \Sym^{++}(n),
\end{align}
where 
$\eta$ is independent of $u$.
Then the random variable $y|u$ is Gaussian, with density propositional to
$\exp(-\frac{1}{2}(Au - y)^T\Gamma^{-1}(Au - y))$.
The Gaussian measure corresponding to $u|y$ is
$\mu^{y} = \Ncal(m,C)$, where $m$ and $C$ are given by, respectively (\cite{Stuart:Inverse2010}),
\begin{align}
\label{equation:m-posterior}
&m = m_0 + C_0A^{*}(\Gamma + AC_0A^{*})^{-1}(y- Am_0),
\\
&C = C_0 - C_0A^{*}(\Gamma + AC_0A^{*})^{-1}AC_0.
\label{equation:C-posterior}
\end{align}
In the Bayesian setting, $\mu_0$ is the prior probability measure on $u$
and $\mu^y$ is the posterior probability measure of $u$ given the data $y$.
In \cite{Alexanderian:2016}, the authors computed the KL-divergence $D_{\KL}(\Ncal(m, C)||\Ncal(m_0, C_0))$ directly for $\Gamma = I$. We now present the general formula for $\Gamma \in \Sym^{++}(n)$,
which is a straightforward consequence of the general expression for the KL-divergence given in Theorem \ref{theorem:limit-KL-gamma-0}.

\begin{theorem}
	\label{theorem:limit-KL-gamma-0-inverse}
	Assume that $m$ and $C$ are given by Eqs.~(\ref{equation:m-posterior}) and (\ref{equation:C-posterior}), respectively.
	Then the KL divergence between the
	posterior measure $\Ncal(m, C)$ and the prior measure $\Ncal(m_0, C_0)$ is given by
	\begin{align}
	& D_{\KL}(\Ncal(m, C)||\Ncal(m_0, C_0)) = \lim_{\gamma \approach 0^{+}}D_{\KL}^{\gamma}(\Ncal(m, C)||\Ncal(m_0, C_0)) 
	\\
	& = \frac{1}{2}\left[\log\det(\Gamma+AC_0A^{*}) -\log\det(\Gamma) -\trace(ACA^{*}\Gamma^{-1}) - \la m-m_0, A^{*}\Gamma^{-1}(Am-y)\ra\right].
	\nonumber
	\end{align}
\end{theorem}

{\bf Special case}. For $\Gamma = I$, we obtain
\begin{align}
\label{equation:KL-inverse-Gamma-I}
& D_{\KL}(\Ncal(m, C)||\Ncal(m_0, C_0)) =\lim_{\gamma \approach 0^{+}}D_{\KL}^{\gamma}(\Ncal(m, C)||\Ncal(m_0, C_0)) 
\nonumber 
\\
& = \frac{1}{2}\left[\log\det(I+AC_0A^{*}) -\trace(ACA^{*}) - \la m-m_0, A^{*}(Am-y)\ra\right].
\end{align}
This is precisely Eq.(19) in Proposition 3 in \cite{Alexanderian:2016}.

{\bf Remark}. As noted in \cite{Alexanderian:2016},
the last term in Eq.(\ref{equation:KL-inverse-Gamma-I}) is precisely $\frac{1}{2}||C_0^{-1/2}(m- m_0)||^2$.
As we can see from Theorem \ref{theorem:limit-KL-gamma-0}, this term is part of the general formula
for KL divergences and is not a specific feature
of the Bayesian inverse problem.

\section{Limiting behavior of the regularized Kullback-Leibler divergences}
\label{section:limit-KL}
In this section, we prove Equation (\ref{equation:limit-KL-gamma-0-equivalent}) in Theorem
\ref{theorem:limit-KL-gamma-0}, which we restate below.
\begin{theorem}
	Assume the hypothesis of Theorem \ref{theorem:limit-KL-gamma-0}. Then
\label{theorem:limit-KL-gamma-0-1}
\begin{align}
\label{equation:limit-KL-gamma-0-equivalent-1}
\lim_{\gamma \approach 0^{+}}D_{\KL}^{\gamma}(\nu || \mu) &= \frac{1}{2}||C_0^{-1/2}(m-m_0)||^2 - \frac{1}{2}\log\dettwo(I-S).
\end{align}
\end{theorem}
The first term on the right hand side of (\ref{equation:limit-KL-gamma-0-equivalent-1}) follows from the following result.

\begin{proposition}
\label{proposition:limit-mean-cov-inner-gamma-0}
Assume that $\ker(C_0) = \{0\}$. Then
\begin{align}
&\lim_{\gamma \approach 0^{+}} \la m- m_0, (C_0 +\gamma I)^{-1}(m-m_0)\ra
\nonumber
\\
&=\left\{
\begin{matrix}
||C_0^{-1/2}(m-m_0)||^2 & \;\;\; \text{when}\;\; m-m_0 \in \Im(C_0^{1/2}),
\\
\infty & \;\;\; \text{when}\;\; m-m_0 \notin  \Im(C_0^{1/2}).
\end{matrix}
\right.
\end{align}
\end{proposition}

We first prove the following more general 
technical result.

\begin{lemma} 
	\label{lemma:limit-inner-gamma-0}
	Let
	$A$ be a self-adjoint, positive, compact operator on $\H$.
	Then
	\begin{align}
	\lim_{\gamma \approach 0^{+}}\la x, A^{1/2}(A+\gamma I)^{-1}A^{1/2}x\ra = ||x||^2 \;\;\; \forall x \in \H.
	\end{align}
	Assume further that $\ker(A) = \{0\}$, then for any $x \in \H$,
	\begin{align}
	\lim_{\gamma \approach 0^{+}}\la x, (A+\gamma I)^{-1}x\ra =
	\left\{
	\begin{matrix}
	||A^{-1/2}x||^2 & \;\;\; \text{when}\;\; x \in \Im(A^{1/2}),
	\\
	\infty & \;\;\; \text{when}\;\; x \notin  \Im(A^{1/2}).
	\end{matrix}
	\right.
	\end{align}
\end{lemma}
\begin{proof} 
	Let $\{\lambda_k\}_{k=1}^{\infty}$ be the eigenvalues of $A$, with corresponding orthonormal eigenvectors $\{e_k\}_{k=1}^{\infty}$, then 
	we have 
	the spectral decomposition 
	$A = \sum_{k=1}^{\infty}\lambda_k e_k \otimes e_k \imply A^{1/2}(A+\gamma I)^{-1}A^{1/2} = \sum_{k=1}^{\infty}\frac{\lambda_k}{\lambda_k + \gamma}e_k \otimes e_k$.
	For each $x \in \H$, write $x = \sum_{k=1}^{\infty}x_k e_k$, where $x_k = \la x, e_k\ra$. Then
	$\la x, A^{1/2}(A+\gamma I)^{-1}A^{1/2}x\ra = \sum_{k=1}^{\infty}\frac{\lambda_k}{\lambda_k +\gamma}x_k^2$.
	By Lebesgue's Monotone Convergence Theorem, we then have
	$\lim_{\gamma \approach 0^{+}}\la x, A^{1/2}(A+\gamma I)^{-1}A^{1/2}x\ra = \lim_{\gamma \approach 0^{+}}\sum_{k=1}^{\infty}\frac{\lambda_k}{\lambda_k +\gamma}x_k^2
	= \sum_{k=1}^{\infty} \lim_{\gamma \approach 0^{+}}\frac{\lambda_k}{\lambda_k +\gamma}x_k^2 = 
	\sum_{k=1}^{\infty}x_k^2 = ||x||^2.
	$
	This proves
	the first identity. If $\ker(A) = \{0\}$, then we have $\lambda_k > 0$ $\forall k \in \N$ and
	\begin{align*}
	\Im(A^{1/2}) = \left\{ x = \sum_{k=1}^{\infty}x_ke_k \in \H \; : \; \sum_{k=1}^{\infty}\frac{x_k^2}{\lambda_k} < \infty\right\}.
	\end{align*}
	Thus for any $ x\in \H$, we have
	\begin{align*}
	\lim_{\gamma \approach 0^{+}}\la x, (A+\gamma I)^{-1}x\ra &= \lim_{\gamma \approach 0^{+}}\sum_{k=1}^{\infty}\frac{1}{\lambda_k +\gamma}x_k^2 =
	\sum_{k=1}^{\infty}\lim_{\gamma \approach 0^{+}}\frac{1}{\lambda_k +\gamma}x_k^2 = 
	\sum_{k=1}^{\infty}\frac{x_k^2}{\lambda_k}
	\\
	&=
	\left\{
	\begin{matrix}
	||A^{-1/2}x||^2 \;\;\; & \text{when} \;\; x \in \Im(A^{1/2})
	\\
	\infty \;\;\; & \text{when}\;\; x \notin \Im(A^{1/2}).
	\end{matrix}
	\right.
	\end{align*}
\qed \end{proof}

\begin{proof}
	[\textbf{
of Proposition \ref{proposition:limit-mean-cov-inner-gamma-0}}]
	This follows from Lemma \ref{lemma:limit-inner-gamma-0} by letting $x = m-m_0$ and $A = C_0$.
	\qed \end{proof}

The second term on the right hand side of (\ref{equation:limit-KL-gamma-0-equivalent-1}) follows from the following result.

\begin{assumption}
\label{assumption:main-1}
Let $C \in \Tr(\H),C_0 \in \Tr(\H)$ be self-adjoint, positive. Assume that there exists
$S \in \Sym(\H) \cap \HS(\H)$ such that 
$I-S$ is strictly positive and that
\begin{align}
C = C_0^{1/2}(I-S)C_0^{1/2}.
\end{align}
%
\end{assumption}
\begin{theorem} 
\label{theorem:general-limit-detX-gamma-0}
Let $C_0,C,S$ be three bounded linear operators on $\H$ satisfying the hypothesis of Assumption  \ref{assumption:main-1}. Then 
\begin{align}
\lim_{\gamma \approach 0^{+}}d^1_{\logdet}[(C + \gamma I), (C_0 + \gamma I)]
= -\log\dettwo(I-S).
\end{align}
The right hand side is nonnegative, with zero equality if and only if $S=0$, that is if and only if $C= C_0$.
If, in addition, $S$ is assumed to be trace class, then
\begin{align}
\lim_{\gamma \approach 0^{+}}d^1_{\logdet}[(C + \gamma I), (C_0 + \gamma I)]
= -\log\det(I-S) - \trace(S).
\end{align}
\end{theorem}

The limit in Theorem \ref{theorem:general-limit-detX-gamma-0} follows from the continuity of the Hilbert-Carleman determinant 
$\dettwo$
in the Hilbert-Schmidt norm $||\;||_{\HS}$. Its proof consists of two steps, which constitute the following two results.

\begin{proposition}
\label{proposition:d1logdet-HS}
Let $C_0,C$ be two self-adjoint, positive, trace class operators. Assume that there exists a self-adjoint, Hilbert-Schmidt operator $S$ such that $C = C_0^{1/2}(I-S)C_0^{1/2}$.
Then for any $\gamma > 0$, $\gamma \in \R$,
\begin{align}
&d^1_{\logdet}[(C+\gamma I), (C_0 +\gamma I)] 
\nonumber
\\
&= -\log\dettwo[I - (C_0 +\gamma I)^{-1/2}C_0^{1/2}SC_0^{1/2}(C_0 + \gamma I)^{-1/2}].
\end{align}
\end{proposition}

\begin{proposition}
\label{proposition:norm-convergence-HS}
Let $A$ be a compact, self-adjoint, positive operator on $\H$. Let $B \in \HS(\H)$. Then
\begin{align}
\lim_{\gamma \approach 0^{+}}||(A+\gamma I)^{-1/2}A^{1/2}BA^{1/2}(A+\gamma I)^{-1/2} - B||_{\HS} = 0.
\end{align}
\end{proposition}

\begin{lemma} 
	\label{lemma:logdet-trace-S}
	Let $S \in \Sym(\H) \cap \HS(\H)$ such that $I-S$ is strictly positive. Then
	\begin{align}
	\log\dettwo(I-S) \leq 0,
	\end{align}
	with equality if and only if $S = 0$.
\end{lemma}
\begin{proof}
	Consider the function $f(x) = \log(1-x) + x$ for $x < 1$. We have $f{'}(x) = -\frac{x}{1-x}$, with
	$f'(x) > 0$ for $x < 0$ and $f'(x) < 0$ for $0 < x < 1$. Thus $f$ has a unique global maximum $f_{\max} = f(0) = 0$. Hence $f(x) \leq 0$, with 
	equality if and only if $x = 0$.
	
	Let $\{\lambda_k\}_{k=1}^{\infty}$ denote the eigenvalues of $S$, then since $I-S$ is strictly positive, we have $\lambda_k < 1$ $\forall k \in \N$. Then
	$\log\dettwo(I-S) = \sum_{k=1}^{\infty}[\log(1-\lambda_k) +\lambda_k] \leq 0$,
	with equality if and only if $\lambda_k = 0$ $\forall k\in \N$, that is if and only if $S = 0$. 
\qed \end{proof}

\begin{proof}
{\textbf{(of Theorem \ref{theorem:general-limit-detX-gamma-0})}}
By Proposition \ref{proposition:d1logdet-HS}, we have for any $\gamma > 0$,
\begin{align*}
&d^1_{\logdet}[(C+\gamma I), (C_0 +\gamma I)] 
\nonumber
\\
&= -\log\dettwo[I - (C_0 +\gamma I)^{-1/2}C_0^{1/2}SC_0^{1/2}(C_0 + \gamma I)^{-1/2}].
\end{align*}
By Proposition \ref{proposition:norm-convergence-HS}, we have
\begin{align}
\lim_{\gamma \approach 0^{+}}||(C_0 +\gamma I)^{-1/2}C_0^{1/2}SC_0^{1/2}(C_0+\gamma I)^{-1/2} - S||_{\HS} = 0.
\end{align}
By Theorem 6.5 in \cite{Simon:1977}, which states the continuity of the Hilbert-Carleman determinant in the Hilbert-Schmidt norm topology, we then obtain
\begin{align*}
& \lim_{\gamma \approach 0^{+}}\dettwo[I - (C_0 +\gamma I)^{-1/2}C_0^{1/2}SC_0^{1/2}(C_0 + \gamma I)^{-1/2}]
= \dettwo(I-S).
\end{align*}
It then follows that
\begin{align*}
\lim_{\gamma \approach 0^{+}}d^1_{\logdet}[(C+\gamma I), (C_0 +\gamma I)]  =-\log\dettwo(I-S).
\end{align*}
By Lemma \ref{lemma:logdet-trace-S}, the right hand side is always nonnegative, with zero equality if and only if $S=0$.
From the expression $C = C_0^{1/2}(I-S)C_0^{1/2}$, this happens if and only if $C = C_0$.
If $S$ is trace class, then $\dettwo(I-S) = \det(I-S)\exp(\trace(S))$ and we have
\begin{align*}
\lim_{\gamma \approach 0^{+}}d^1_{\logdet}[(C+\gamma I), (C_0 +\gamma I)]  =-\log\det(I-S) - \trace(S).
\end{align*}
\qed \end{proof}

\begin{proof}
{\textbf{(of Proposition \ref{proposition:d1logdet-HS})}}
By the product property of the extended Fredholm determinant (Proposition 4 in \cite{Minh:LogDet2016}) and the commutativity of the extended trace operation (Lemma 4 in \cite{Minh:LogDet2016}), we have
\begin{align*}
&\detX[(C_0 + \gamma I)^{-1}(C+\gamma I)] = \detX[(C_0 + \gamma I)^{-1/2}(C+\gamma I)(C_0 + \gamma I)^{-1/2}],
\\
&\trace_X[(C_0 + \gamma I)^{-1}(C+\gamma I) - I] = \trace_X[(C_0 + \gamma I)^{-1/2}(C+\gamma I)(C_0 + \gamma I)^{-1/2}- I].
\end{align*}
For $C = C_0^{1/2}(I-S)C_0^{1/2} = C_0 - C_0^{1/2}SC_0^{1/2}$, we have for any $\gamma > 0$,
$C + \gamma I = C_0 + \gamma I  - C_0^{1/2}SC_0^{1/2}$.
Thus it follows that
\begin{align*}
&(C_0 +\gamma I)^{-1/2}(C+\gamma I)(C_0 + \gamma I)^{-1/2} 
&= I - (C_0 +\gamma I)^{-1/2}C_0^{1/2}SC_0^{1/2}(C_0 + \gamma I)^{-1/2}.
\end{align*}
By definition of $d^1_{\logdet}$, we have
\begin{align*}
&d^1_{\logdet}[(C+\gamma I), (C_0+\gamma I)]
\\
&= \trace_X[(C_0+\gamma I)^{-1}(C+\gamma I)-I] - \log\detX[(C_0+\gamma I)^{-1}(C+\gamma I)]
\\
& = \trace_X[(C_0 + \gamma I)^{-1/2}(C+\gamma I)(C_0 + \gamma I)^{-1/2}- I]
\\
&-\log\detX[(C_0 + \gamma I)^{-1/2}(C+\gamma I)(C_0 + \gamma I)^{-1/2}]
\\
& =  - \trace[(C_0 +\gamma I)^{-1/2}C_0^{1/2}SC_0^{1/2}(C_0 + \gamma I)^{-1/2}]
\\
& - \log\det[I - (C_0 +\gamma I)^{-1/2}C_0^{1/2}SC_0^{1/2}(C_0 + \gamma I)^{-1/2}]
\\
& = -\log\dettwo[I - (C_0 +\gamma I)^{-1/2}C_0^{1/2}SC_0^{1/2}(C_0 + \gamma I)^{-1/2}].
\end{align*}
\qed \end{proof}

\textbf{Proof of Proposition \ref{proposition:norm-convergence-HS}}.
We recall that a Banach space $\Bcal$ is  said to have the Radon-Riesz Property if $||x_n|| \approach ||x||$ and
$x_n \approach x$ weakly imply that $||x_n - x|| \approach 0$ for all $\{x_n\}_{n \in \N}$ and $x$ in $\Bcal$.
In particular, 
a Hilbert space $\H$
possesses the Radon-Riesz Property. 
We
now
utilize this property for the Hilbert space $\HS(\H)$, under the Hilbert-Schmidt inner product. We first prove the following.

\begin{lemma} 
\label{lemma:limit-inner-gamma-1}
Let
$A$ be a self-adjoint, positive, compact operator on $\H$.
Then
\begin{align}
\lim_{\gamma \approach 0^{+}}\la (A+\gamma I)^{-1/2}A^{1/2}x, y\ra = \la x, y\ra, \;\;\;\forall x,y \in \H,
\end{align}
that is 
$(A+\gamma I)^{-1/2}A^{1/2}$ converges to 
$I$ in the weak operator topology as $\gamma \approach 0^{+}$.
\end{lemma}
\begin{proof} 
Let $\{\lambda_k\}_{k=1}^{\infty}$ be the eigenvalues of $A$, with corresponding orthonormal eigenvectors $\{e_k\}_{k=1}^{\infty}$.
For any 
$x, y \in \H$, write $x = \sum_{k=1}^{\infty}x_k e_k$, $y = \sum_{k=1}^{\infty}y_ke_k$, where $x_k = \la x, e_k\ra$, $y_k = \la y, e_k\ra$. Then
$\la (A+\gamma I)^{-1/2}A^{1/2}x, y\ra = \sum_{k=1}^{\infty}\frac{\lambda_k^{1/2}}{(\lambda_k +\gamma)^{1/2}}x_ky_k$.
For each
$k \in \N$,
$\lim_{\gamma \approach 0^{+}} \frac{\lambda_k^{1/2}}{(\lambda_k +\gamma)^{1/2}}x_ky_k = x_ky_k$.
Furthermore,
\begin{align*}
&\sum_{k=1}^{\infty}\left|\frac{\lambda_k^{1/2}}{(\lambda_k +\gamma)^{1/2}}x_ky_k \right|\leq \sum_{k=1}^{\infty}|x_ky_k| \leq \frac{1}{2}\sum_{k=1}^{\infty}[|x_k|^2 + |y_k|^2] = \frac{1}{2}[||x||^2 + ||y||^2]
< \infty.
\end{align*}
Thus by Lebesgue's Dominated Convergence Theorem, 
$\lim_{\gamma \approach 0^{+}}\la (A+\gamma I)^{-1/2}A^{1/2}x,y\ra = \lim_{\gamma \approach 0^{+}}\sum_{k=1}^{\infty}\frac{\lambda_k^{1/2}}{(\lambda_k +\gamma)^{1/2}}x_ky_k 
 = \sum_{k=1}^{\infty}\lim_{\gamma \approach 0^{+}}\frac{\lambda_k^{1/2}}{(\lambda_k +\gamma)^{1/2}}x_ky_k 
= \sum_{k=1}^{\infty}x_ky_k = \la x, y\ra.
$
\qed \end{proof}

\begin{remark}
Lemma \ref{lemma:limit-inner-gamma-1}
states that 
$(A+\gamma I)^{-1/2}A^{1/2}$ converges weakly to the identity operator $I$ as $\gamma \approach 0^{+}$.
When $\dim(\H) = \infty$, this convergence does not hold in the operator norm topology.
For any $\gamma > 0$, the operator $A(A+\gamma I)^{-1}$ has eigenvalues $\{\frac{\lambda_k}{\lambda_k +\gamma}\}_{k=1}^{\infty}$, with
$\lim_{\gamma \approach 0^{+}}\frac{\lambda_k}{\lambda_k +\gamma} = 1$.
However, $\lim_{\gamma \approach 0^{+}}||I - A(A+\gamma I)^{-1}||  = \lim_{\gamma \approach 0^{+}}\gamma ||(A+\gamma I)^{-1}||\neq 0$ if $\dim(\H) = \infty$. In fact, we have
\begin{align*}
||\gamma (A+\gamma I)^{-1}e_k|| = \frac{\gamma}{\lambda_k + \gamma} \imply \sup_{k\in \N}||\gamma (A+\gamma I)^{-1}e_k|| = 1
\end{align*}
for any $\gamma > 0$, since $\lim_{k \approach \infty}\lambda_k = 0$. Thus $||\gamma (A+\gamma I)^{-1}|| = 1$ 
$\forall \gamma > 0$.
\end{remark}

\begin{lemma}
\label{lemma:weak-convergence-HS}
Let $A$ be a compact, self-adjoint, positive operator on $\H$. Let $B \in \HS(\H)$. Then
for any operator $C \in \HS(\H)$,
\begin{align}
\lim_{\gamma \approach 0^{+}}\la (A+\gamma I)^{-1/2}A^{1/2}BA^{1/2}(A+\gamma I)^{-1/2}, C\ra_{\HS} = \la B, C\ra_{\HS},
\end{align}
i.e.
$(A+\gamma I)^{-1/2}A^{1/2}BA^{1/2}(A+\gamma I)^{-1/2}$ converges weakly to $B$ in $\HS(\H)$
as $\gamma \approach 0^{+}$.
\end{lemma}
\begin{proof}
Let $\{\lambda_k\}_{k=1}^{\infty}$ be the eigenvalues of $A$, with corresponding orthonormal eigenvectors $\{e_k\}_{k=1}^{\infty}$.
For any operator $C \in \HS(\H)$, we have
\begin{align*}
&\la (A+\gamma I)^{-1/2}A^{1/2}BA^{1/2}(A+\gamma I)^{-1/2}, C\ra_{\HS}
\\
&
= \trace[C^{*} (A+\gamma I)^{-1/2}A^{1/2}BA^{1/2}(A+\gamma I)^{-1/2}]
\\
& = \sum_{k=1}^{\infty}\la e_k, C^{*} (A+\gamma I)^{-1/2}A^{1/2}BA^{1/2}(A+\gamma I)^{-1/2}e_k\ra
\\
& = \sum_{k=1}^{\infty}\frac{\lambda_k^{1/2}}{(\lambda_k + \gamma)^{1/2}}\la (A+\gamma I)^{-1/2}A^{1/2}Ce_k, Be_k\ra.
\end{align*}
By Lemma \ref{lemma:limit-inner-gamma-1}, we have for each fixed $k \in \N$,
\begin{align*}
\lim_{\gamma \approach 0^{+}}\frac{\lambda_k^{1/2}}{(\lambda_k + \gamma)^{1/2}}\la (A+\gamma I)^{-1/2}A^{1/2}Ce_k, Be_k\ra = \la Ce_k, Be_k\ra.
\end{align*}
Furthermore,
\begin{align*}
&\left|\frac{\lambda_k^{1/2}}{(\lambda_k + \gamma)^{1/2}}\la (A+\gamma I)^{-1/2}A^{1/2}Ce_k, Be_k\ra \right|
\\
& \leq ||(A+\gamma I)^{-1/2}A^{1/2}Ce_k||\;||Be_k|| \leq ||Ce_k||\;||Be_k||, \;\;\;\text{with}
\\
&\sum_{k=1}^{\infty}||Ce_k||\;||Be_k|| \leq \frac{1}{2}\sum_{k=1}^{\infty}[||Ce_k||^2 + ||Be_k||^2] = \frac{1}{2}[||C||^2_{\HS} + ||B||^2_{\HS}] < \infty.
\end{align*}
Thus by Lebesgue's Dominated Convergence Theorem, we then have
\begin{align*}
\lim_{\gamma \approach 0^{+}}\la (A+\gamma I)^{-1/2}A^{1/2}BA^{1/2}(A+\gamma I)^{-1/2}, C\ra_{\HS} = \sum_{k=1}^{\infty}
\la Ce_k,Be_k\ra  = \la C,B\ra_{\HS}.
\end{align*}
\qed \end{proof}

\begin{lemma}
\label{lemma:convergence-of-norm-HS}
Let $A$ be a compact, self-adjoint, positive operator on $\H$. Let $B \in \HS(\H)$. Then
\begin{align}
\lim_{\gamma \approach 0^{+}} ||(A+\gamma I)^{-1/2}A^{1/2}BA^{1/2}(A+\gamma I)^{-1/2}||_{\HS} = ||B||_{\HS}.
\end{align}
\end{lemma}
\begin{proof}
Let $\{\lambda_k\}_{k=1}^{\infty}$ be the eigenvalues of $A$, with corresponding orthonormal eigenvectors $\{e_k\}_{k=1}^{\infty}$. We have for any $\gamma > 0$,
\begin{align*}
&(A+\gamma I)^{-1/2}A^{1/2}BA^{1/2}(A+\gamma I)^{-1/2}e_k = \frac{\lambda_k^{1/2}}{(\lambda_k + \gamma)^{1/2}}(A+\gamma I)^{-1/2}A^{1/2}Be_k.
\end{align*}
It follows that
\begin{align*}
&||(A+\gamma I)^{-1/2}A^{1/2}BA^{1/2}(A+\gamma I)^{-1/2}||^2_{\HS}
\\
&
= \sum_{k=1}^{\infty}||(A+\gamma I)^{-1/2}A^{1/2}BA^{1/2}(A+\gamma I)^{-1/2}e_k||^2 
\\
&= \sum_{k=1}^{\infty}\frac{\lambda_k}{\lambda_k +\gamma}||(A+\gamma I)^{-1/2}A^{1/2}Be_k||^2
= \sum_{k=1}^{\infty}\frac{\lambda_k}{\lambda_k +\gamma}\la Be_k, A^{1/2}(A+\gamma I)^{-1}A^{1/2}Be_k\ra.
\end{align*}
By Lemma \ref{lemma:limit-inner-gamma-0}, we have
\begin{align*}
\lim_{\gamma \approach 0^{+}}\frac{\lambda_k}{\lambda_k +\gamma}\la Be_k, A^{1/2}(A+\gamma I)^{-1}A^{1/2}Be_k\ra
= ||Be_k||^2.
\end{align*}
Furthermore,
\begin{align*}
&\left|\frac{\lambda_k}{\lambda_k +\gamma}\la Be_k, A^{1/2}(A+\gamma I)^{-1}A^{1/2}Be_k\ra \right|
\leq ||Be_k||\;||A^{1/2}(A+\gamma I)^{-1}A^{1/2}Be_k||
\\
&\leq ||Be_k||^2, \;\text{with}\; \sum_{k=1}^{\infty}||Be_k||^2  = ||B||^2_{\HS} < \infty.
\end{align*}
Thus by Lebesgue's Dominated Convergence Theorem, we have
\begin{align*}
&\lim_{\gamma \approach 0^{+}} ||(A+\gamma I)^{-1/2}A^{1/2}BA^{1/2}(A+\gamma I)^{-1/2}||^2_{\HS} = \sum_{k=1}^{\infty}||Be_k||^2 = ||B||^2_{\HS}.
\end{align*}
\qed \end{proof}

\begin{lemma} 
	\label{lemma:convergence-in-norm-gamma-3}
	Let $A\in \Tr(\H)$
	 be
	self-adjoint, 
	positive.
	Let $B \in \L(\H)$.
	Then
	\begin{align}
	\lim_{\gamma \approach 0^{+}}||(A+\gamma I)^{-1/2}ABA(A+\gamma I)^{-1/2} - A^{1/2}BA^{1/2}||_{\HS} = 0.
	\end{align}
	If $B \in \HS(\H)$, then 
	\begin{align}
	\lim_{\gamma \approach 0^{+}}||(A+\gamma I)^{-1/2}ABA(A+\gamma I)^{-1/2} - A^{1/2}BA^{1/2}||_{\trace} = 0.
	\end{align}
\end{lemma}
\begin{proof} 
	Since $A$ and $(A+\gamma I)$ commute, we have
	\begin{align*}
	&||A^{1/2} - (A+\gamma I)^{1/2}|| = ||[A - (A+\gamma )][A^{1/2} + (A+\gamma I)^{1/2}]^{-1}||
	\\
	&= \gamma ||[A^{1/2} + (A+\gamma I)^{1/2}]^{-1}|| \leq \sqrt{\gamma}. 
	\end{align*}
	Since $A$ is trace class, self-adjoint, positive,
	$A^{1/2} \in \HS(\H)$,
	so that for $B \in \L(\H)$, $A^{1/2}B\in \HS(\H)$, $BA^{1/2} \in \HS(\H)$.
	We then have
	\begin{align*}
	&||(A+\gamma I)^{-1/2}ABA(A+\gamma I)^{-1/2} - A^{1/2}BA^{1/2}||_{\HS}
	\\
	&=||(A+\gamma I)^{-1/2}A^{1/2}[A^{1/2}BA^{1/2} - (A+\gamma I)^{1/2}B(A+\gamma I)^{1/2}]A^{1/2}(A+\gamma I)^{-1/2}||_{\HS}
	\\
	& \leq ||(A+\gamma I)^{-1/2}A^{1/2}[A^{1/2}BA^{1/2} - A^{1/2}B(A+\gamma I)^{1/2}]A^{1/2}(A+\gamma I)^{-1/2}||_{\HS}
	\\
	&+ ||(A+\gamma I)^{-1/2}A^{1/2}[A^{1/2}B(A+\gamma I)^{1/2} - (A+\gamma I)^{1/2}B(A+\gamma I)^{1/2}]A^{1/2}(A+\gamma I)^{-1/2}||_{\HS}
	\\ 
	& \leq ||(A+\gamma I)^{-1/2}A^{1/2}||^2||A^{1/2}BA^{1/2} - A^{1/2}B(A+\gamma I)^{1/2}||_{\HS}
	\\
	&+ ||(A+\gamma I)^{-1/2}A^{1/2}[A^{1/2}B(A+\gamma I)^{1/2} - (A+\gamma I)^{1/2}B(A+\gamma I)^{1/2}]A^{1/2}(A+\gamma I)^{-1/2}||_{\HS}
	\\
	& \leq ||A^{1/2}B||_{\HS}||A^{1/2}-(A+\gamma I)^{1/2}|| 
	\\
	&+ ||(A+\gamma I)^{-1/2}A^{1/2}||\;||A^{1/2}BA^{1/2}-(A+\gamma I)^{1/2}BA^{1/2}||_{\HS}
	\\
	& \leq ||A^{1/2}B||_{\HS}||A^{1/2}-(A+\gamma I)^{1/2}|| + ||A^{1/2}-(A+\gamma I)^{1/2}||\;||BA^{1/2}||_{\HS}
	\\
	& = ||A^{1/2}-(A+\gamma I)^{1/2}||[||A^{1/2}B||_{\HS} + ||BA^{1/2}||_{\HS}]
	\\
	& \leq \sqrt{\gamma} [||A^{1/2}B||_{\HS} + ||BA^{1/2}||_{\HS}]\approach 0 \;\text{as}\; \gamma \approach 0^{+}.
	\end{align*}
	If $B \in \HS(\H)$, then we have $A^{1/2}B \in \Tr(\H), BA^{1/2} \in \Tr(\H)$ and
	\begin{align*}
	&||(A+\gamma I)^{-1/2}ABA(A+\gamma I)^{-1/2} - A^{1/2}BA^{1/2}||_{\trace}
	\\
	&\leq \sqrt{\gamma}[||A^{1/2}B||_{\tr} + ||BA^{1/2}||_{\trace}]
	\approach 0 \;\text{as}\; \gamma \approach 0^{+}.
	\end{align*}
\qed \end{proof}

\begin{proof}
{\textbf{(of Proposition \ref{proposition:norm-convergence-HS})}}
By Lemma \ref{lemma:weak-convergence-HS}, we have for any 
$C \in \HS(\H)$,
\begin{align*}
\lim_{\gamma \approach 0^{+}}\la (A+\gamma I)^{-1/2}A^{1/2}BA^{1/2}(A+\gamma I)^{-1/2}, C\ra_{\HS} = \la B, C\ra_{\HS},
\end{align*}
that is the operator $(A+\gamma I)^{-1/2}A^{1/2}BA^{1/2}(A+\gamma I)^{-1/2}$ converges weakly to $B$ on $\HS(\H)$
as $\gamma \approach 0^{+}$.
By Lemma \ref{lemma:convergence-of-norm-HS}, 
\begin{align*}
\lim_{\gamma \approach 0^{+}} ||(A+\gamma I)^{-1/2}A^{1/2}BA^{1/2}(A+\gamma I)^{-1/2}||_{\HS} = ||B||_{\HS}.
\end{align*}
Thus Radon-Riesz Property can be invoked to give
\begin{align*}
\lim_{\gamma \approach 0^{+}}||(A+\gamma I)^{-1/2}A^{1/2}BA^{1/2}(A+\gamma I)^{-1/2} - B||_{\HS} = 0.
\end{align*}
\qed \end{proof}

\begin{proof}
[\textbf{
	of Theorem \ref{theorem:limit-KL-gamma-0-1}}]
	This follows from Proposition \ref{proposition:limit-mean-cov-inner-gamma-0}
	and Theorem \ref{theorem:general-limit-detX-gamma-0}.
\qed \end{proof}

\section{Limiting behavior of the 
	regularized R\'enyi divergence}
\label{section:limiting-Renyi}
In this section, we prove Equation(\ref{equation:limit-Renyi-infinite}) in 
Theorem \ref{theorem:limit-Renyi-infinite}, 
which we restate below.
\begin{theorem}
	\label{theorem:limit-Renyi-infinite-1}
	Assume the hypothesis of Theorem \ref{theorem:limit-Renyi-infinite}. Then
\begin{align}
\label{equation:limit-Renyi-infinite-1}
\lim_{\gamma \approach 0^{+}}D_{\Rrm,r}^{\gamma}(\nu ||\mu) 
&= \frac{1}{2}||(I-(1-r)S)^{-1/2}C_0^{-1/2}(m-m_0)||^2
\nonumber
\\
&+ \frac{1}{2r(1-r)}\log\det[(I-(1-r)S)(I-S)^{r-1}]
\end{align}
\end{theorem}

We need the following technical results.

\begin{lemma}
	[\cite{Minh:LogDet2016-AB}]
	\label{lemma:power-inequality}
	Let $0 < r \leq 1$ be fixed. Let $\{A_n\}_{n \in \N} \in \Sym(\H) \cap \HS(\H)$, $A \in \Sym(\H) \cap \HS(\H)$ be such that 
	$I+A > 0$, $I+A_n > 0$ $\forall n \in \N$.
	Assume that $\lim_{n \approach \infty}||A_n - A||_{\HS} = 0$. Then
	\begin{align}
	\lim_{n \approach \infty}||(I+A_n)^r - (I+A)^r||_{\HS} = 0,
	\\
	\lim_{n \approach \infty}||(I+A_n)^{-1} - (I+A)^{-1}||_{\HS} = 0,
	\\
	\lim_{n \approach \infty}||(I+A_n)^{-r} - (I+A)^{-r}||_{\HS} = 0,
	\end{align}
\end{lemma}

\begin{proposition}
	\label{proposition:limit-inner-gamma-0-sum}
	Let $0 < r < 1$ be fixed.
	For $m, m_0 \in \H$ and two self-adjoint, compact, positive operators $C, C_0$ on $\H$, 
	{\small
	\begin{align}
	&\lim_{\gamma \approach 0^{+}}\la m - m_0, [(1-r)(C+\gamma I) + r(C_0 + \gamma I)]^{-1}(m-m_0)\ra
	\\
	&= \left\{\begin{matrix}
	||[(1-r)C + rC_0]^{-1/2}(m-m_0)||^2 &\;\; \text{if} \; m-m_0 \in \Im[(1-r)C + rC_0]^{1/2},
	\\
	\infty &\;\;\text{otherwise}.
	\end{matrix}
	\right.
	\nonumber
	\end{align}
}
	In particular, $||[(1-r)C + rC_0]^{-1/2}(m-m_0)||^2 < \infty$ for $m - m_0 \in \Im(C_0^{1/2})$.
\end{proposition}
\begin{proof}
	By Lemma \ref{lemma:limit-inner-gamma-0},
	\begin{align*}
	&\lim_{\gamma \approach 0^{+}}\la m - m_0, [(1-r)(C+\gamma I) + r(C_0 + \gamma I)]^{-1}(m-m_0)\ra
	\\
	&= \lim_{\gamma \approach 0^{+}}\la m- m_0, [(1-r)C + r C_0 + \gamma I]^{-1}(m-m_0)\ra
	\\
	& = \left\{\begin{matrix}
	||[(1-r)C + rC_0]^{-1/2}(m-m_0)||^2 &\;\; \text{if} \; m-m_0 \in \Im[(1-r)C + rC_0]^{1/2},
	\\
	\infty &\;\;\text{otherwise}.
	\end{matrix}
	\right.
	\end{align*}
	By Theorem 2.2 in \cite{Fillmore:1971Operator}, for any two bounded operators $A,B$ on $\H$,
	\begin{align}
	\Im(A) + \Im(B) = \Im[(AA^{*} + BB^{*})^{1/2}].
	\end{align}
	In particular, for any two self-adjoint, positive bounded operators $A,B$ on $\H$,
	\begin{align}
	\Im(A^{1/2}) + \Im(B^{1/2}) = \Im[(A+B)^{1/2}].
	\end{align}
	Since $0 \in \Im(A^{1/2})$, $0 \in \Im(B^{1/2})$, this implies that $\Im(A^{1/2}) \subset \Im[(A+B)^{1/2}]$, $\Im(B^{1/2}) \subset \Im[(A+B)^{1/2}]$, and we have
	\begin{align}
	||(A+B)^{-1/2}A^{1/2}x|| < \infty,\;\;\; ||(A+B)^{-1/2}B^{1/2}x|| < \infty \;\;\;\forall x \in \H.
	\end{align}
	Thus if $m-m_0 \in \Im(C_0^{1/2})$, then $m-m_0 \in \Im[(1-r)C + rC_0]^{1/2}$ for $0 < r < 1$ and
	\begin{align*}
	&\lim_{\gamma \approach 0^{+}}\la m - m_0, [(1-r)(C+\gamma I) + r(C_0 + \gamma I)]^{-1}(m-m_0)\ra
	\\
	&= ||[(1-r)C + rC_0]^{-1/2}(m-m_0)||^2 < \infty.
	\end{align*}
\qed \end{proof}

\begin{proof}
	[\textbf{
		of Theorem \ref{theorem:limit-Renyi-infinite-1}}]
	By definition of the regularized Renyi divergence,
	Eq.(\ref{equation:regularized-Renyi-infinite}), 
	\begin{align*}
	D^{\gamma}_{\Rrm,r}(\nu ||\mu) 
	&= \frac{1}{2}\la m - m_0, [(1-r)(C + \gamma I) + r(C_0 + \gamma I)]^{-1}(m - m_0)\ra
	\nonumber
	\\
	&+ \frac{1}{2}d^{2r-1}_{\logdet}[(C + \gamma I), (C_0 + \gamma I)].
	\end{align*}
	For the first term, we have
	\begin{align*}
	&(1-r)(C+\gamma I) + r(C_0 + \gamma I) = 
	(1-r)(C_0^{1/2}(I-S)C_0^{1/2}+\gamma I) + r(C_0 +\gamma I)
	\\
	&= C_0^{1/2}(I - (1-r)S)C_0^{1/2} + \gamma I.
	\end{align*}
	Thus by Proposition \ref{proposition:limit-inner-gamma-0-sum}, we have for $m - m_0 \in \Im(C_0^{1/2})$,
	\begin{align*}
	&\lim_{\gamma \approach 0^{+}}\la m - m_0, [(1-r)(C + \gamma I) + r(C_0 + \gamma I)]^{-1}(m - m_0)\ra
	\\
	& 
	= ||[(1-r)C + rC_0]^{-1/2}(m-m_0)||^2
	= ||[C_0^{1/2}(I - (1-r)S)C_0^{1/2}]^{-1/2}(m-m_0)||^2.
	\end{align*}
	Let $\{\beta_k\}_{k \in \N}$ be the eigenvalues of $C_0^{1/2}(I - (1-r)S)C_0^{1/2}$,
	with corresponding orthonormal eigenvectors $\{\varphi_k\}_{k\in \N}$.
	Since $\ker(C_0) = \{0\}$, we have $\beta_k > 0$ $\forall k \in \N$.
	Then $\{\frac{(I-(1-r)S)^{1/2}C_0^{1/2}\varphi_k}{\sqrt{\beta_k}}\}_{k \in \N}$
	are the orthonormal eigenvectors
	of $(I-(1-r)S)^{1/2}C_0(I-(1-r)S)^{1/2}$, with the same eigenvalues. Thus
	\begin{align*}
&	||[C_0^{1/2}(I - (1-r)S)C_0^{1/2}]^{-1/2}(m-m_0)||^2 = \sum_{k=1}^{\infty}\frac{\la m- m_0, \varphi_k\ra^2}{\beta_k}
\\
& = \sum_{k=1}^{\infty}\left\la (I-(1-r)S)^{-1/2}C_0^{-1/2}(m-m_0), \frac{(I-(1-r)S)^{1/2}C_0^{1/2}\varphi}{\sqrt{\beta_k}}\right\ra^2
\\
& = ||(I-(1-r)S)^{-1/2}C_0^{-1/2}(m-m_0)||^2.
\end{align*}
	
	For the second term, by Definition
	\ref{def:logdet},
	\begin{align*}
	&d^{2r-1}_{\logdet}[(C+\gamma I), (C_0 + \gamma I)] 
	\\
	&= \frac{1}{r(1-r)}\log\left[\frac{\detX((1-r) (C+\gamma I) + r(C_0 + \gamma I))}{\detX(C+\gamma I)^{1-r}\detX(C_0 + \gamma I)^{r}}\right]
	\\
	& = \frac{1}{r(1-r)}\log\left[\frac{\detX[(1-r)(C_0 + \gamma I)^{-1/2}(C+\gamma I)(C_0 + \gamma I)^{-1/2} +r I]}
	{\detX[(C_0 + \gamma I)^{-1/2}(C+\gamma I)(C_0 + \gamma I)^{-1/2}]^{1-r}}\right].
	\end{align*}
	For $C = C_0^{1/2}(I-S)C_0^{1/2} = C_0 - C_0^{1/2}SC_0^{1/2}$, we have
	\begin{align*}
	&(C_0 + \gamma I)^{-1/2}(C+\gamma I)(C_0 + \gamma I)^{-1/2} 
	\\
	&= I - (C_0 +\gamma I)^{-1/2}C_0^{1/2}SC_0^{1/2}(C_0 +\gamma I)^{-1/2}.
	\end{align*}
	Thus the extended Fredholm determinant of $(C_0 + \gamma I)^{-1/2}(C+\gamma I)(C_0 + \gamma I)^{-1/2}$ is 
	the Fredholm determinant 
	of $I - (C_0 +\gamma I)^{-1/2}C_0^{1/2}SC_0^{1/2}(C_0 +\gamma I)^{-1/2}$ and consequently
	\begin{align*}
	&d^{2r-1}_{\logdet}[(C+\gamma I), (C_0 + \gamma I)] 
	\\
	&= \frac{1}{r(1-r)}\log\left[\frac{\det[I -(1-r)(C_0 +\gamma I)^{-1/2}C_0^{1/2}SC_0^{1/2}(C_0 +\gamma I)^{-1/2}]}{\det[I - (C_0 +\gamma I)^{-1/2}C_0^{1/2}SC_0^{1/2}(C_0 +\gamma I)^{-1/2}]^{1-r}}\right]
	\\
	& = \frac{1}{r(1-r)}\log\det\left([I-(1-r)A_{S,\gamma}](I - A_{S,\gamma})^{r-1}\right),
	\end{align*}
	where $A_{S,\gamma} = (C_0 +\gamma I)^{-1/2}C_0^{1/2}SC_0^{1/2}(C_0+\gamma I)^{-1/2}$.
	
	By Proposition \ref{proposition:norm-convergence-HS}, we have
	$\lim_{\gamma \approach 0^{+}}||A_{S,\gamma} - S||_{\HS} = 0$.
	By Lemma \ref{lemma:power-inequality}, 
	\begin{align*}
	\lim_{\gamma \approach 0^{+}}||(I - A_{S,\gamma})^{r-1} - (I-S)^{r-1}||_{\HS} = 0, \;\;\; 0 < r < 1.
	\end{align*}
	We then exploit the property that $||A_1A_2||_{\trace} \leq ||A_1||_{\HS}||A_2||_{\HS}$
	for any two Hilbert-Schmidt operators $A_1,A_2$ (see e.g. \cite{ReedSimon:Functional}). This gives us
	\begin{align*}
	\lim_{\gamma \approach 0^{+}}||(I-(1-r)A_{S,\gamma})(I - A_{S,\gamma})^{r-1} - (I-(1-r)S)(I-S)^{r-1}||_{\trace} = 0.
	\end{align*}
	By the continuity of the Fredholm determinant with respect to the trace norm (see e.g. Theorem 3.5 in \cite{Simon:1977}), we then obtain
	\begin{align*}
	\lim_{\gamma \approach 0^{+}}\log\det\left([I-(1-r)A_{S,\gamma}](I - A_{S,\gamma})^{r-1}\right) = \log\det[(I-(1-r)S)(I-S)^{r-1}].
	\end{align*}
\qed \end{proof}

\section{The Radon-Nikodym derivatives and divergences between Gaussian measures on Hilbert spaces}
\label{section:exact-divergences}
For completeness, 
we now derive the explicit formulas for 
the exact Kullback-Leibler and R\'enyi divergences between two {
{\it equivalent} Gaussian measures, that is
  Eq.~(\ref{equation:limit-KL-gamma-0-equivalent-2}) 
	in Theorem \ref{theorem:limit-KL-gamma-0}
	and 
	Eq.~(\ref{equation:limit-Renyi-infinite-2})
	 in Theorem \ref{theorem:limit-Renyi-infinite}.
%

Throughout the following, we
utilize the 
{\it white noise mapping}, see e.g. \cite{DaPrato:2006,DaPrato:PDEHilbert}.
Let $m \in \H$ and $Q$ be a self-adjoint, positive trace class operator on $\H$.
Assume that $\ker(Q) = \{0\}$, then the Gaussian measure $\mu = \Ncal(m,Q)$ is said to be non-degenerate.
Let $\{\lambda_k\}_{k=1}^{\infty}$ be the eigenvalues of $Q$, with corresponding orthonormal eigenvectors $\{e_k\}_{k=1}^{\infty}$, then
$\lambda_ k > 0$ $\forall k \in \N$, with $\lim_{k \approach \infty}\lambda_k = 0$.
The inverse operator $Q^{-1}:\Im(Q) \mapto \H$ is {\it unbounded}, since
$Q^{-1}e_k = \frac{1}{\lambda_k}e_k$
with $||Q^{-1}e_k|| = \frac{1}{\lambda_k} \approach \infty
$ 
	as $k \approach \infty$.
For $r \geq 0$, define the following subspace
\begin{align}
Q^r(\H) = \Im(Q^{r}) = \left\{\sum_{k=1}^{\infty}\lambda_k^r a_k e_k \; : \; \sum_{k=1}^{\infty}a_k^2 < \infty\right\} \subset \H.
\end{align}
For $r = \frac{1}{2}$, the space $Q^{1/2}(\H) = \Im(Q^{1/2})$ is called the {\it Cameron-Martin space}
associated with the Gaussian measure $\Ncal(m,Q)$. It is a Hilbert space with
inner product
\begin{align}
\la x, y\ra_{Q} = \la Q^{-1/2}x, Q^{-1/2}y\ra, \;\;\; x,y \in \Im(Q^{1/2}).
\end{align}

In the following, for $\mu = \Ncal(m, Q)$, we  define
\begin{align}
\L^2(\H, \mu) = \L^2(\H, \Bsc(\H),\mu) = \L^2(\H, \Bsc(\H), \Ncal(m,Q)).
\end{align}




{\bf White noise mapping}.
Consider the following mapping
\begin{align}
&W:Q^{1/2}(\H) \subset \H \mapto \L^2(\H,\mu), \;\; z  \in Q^{1/2}(\H) \mapto W_z \in \L^2(\H, \mu),
\\
&W_z(x) = \la x -m, Q^{-1/2}z\ra,  \;\;\; z \in Q^{1/2}(\H), x \in \H.
\end{align}
For any pair $z_1, z_2 \in Q^{1/2}(\H)$, we have by definition of the covariance operator
\begin{align}
\la W_{z_1}, W_{z_2}\ra_{\L^2(\H,\mu)}
&= \int_{\H}\la x -m, Q^{-1/2}z_1\ra\la x-m, Q^{-1/2}z_2\ra\Ncal(m, Q)(dx)
\nonumber
\\
& = \la Q(Q^{-1/2}z_1), Q^{-1/2}z_2\ra 
= \la z_1, z_2\ra_{\H}.
\end{align}
Thus the map $W:Q^{1/2}(\H) \mapto \L^2(\H, \mu)$ is an isometry, that is
\begin{align}
||W_z||_{\L^2(\H,\mu)} = ||z||_{\H}, \;\;\; z \in Q^{1/2}(\H).
\end{align}
Since $\ker(Q) = \{0\}$, the subspace $Q^{1/2}(\H)$ is dense in $\H$ and the map $W$ can be uniquely extended to all of $\H$, as follows.
For any $z \in \H$, let $\{z_n\}_{n\in \N}$ be a sequence in $Q^{1/2}(\H)$ with $\lim_{n \approach \infty}||z_n -z||_{\H} = 0$.
Then $\{z_n\}_{n \in \N}$ is a Cauchy sequence in $\H$, so that by isometry, $\{W_z\}_{n\in \N}$ is also
a Cauchy sequence in $\L^2(\H, \mu)$, thus converging to a unique element in $\L^2(\H, \mu)$.
Thus for any $z \in \H$, we can define the map
\begin{align}
W: \H \mapto \L^2(\H, \mu),  \;\;\; z \in \H \mapto \L^2(\H, \mu)
\end{align}
by the following unique limit in $\L^2(\H, \mu)$
\begin{align}
W_z(x) = \lim_{n \approach \infty}W_{z_n}(x) = \lim_{n \approach \infty}\la x-m, Q^{-1/2}z_n\ra.
\end{align}
The map $W: \H \mapto \L^2(\H, \mu)$ is called the {\it white noise mapping}
associated with the measure $\mu = \Ncal(m,Q)$.
One sees immediately that $W$
maps any orthonormal sequence
$\{\phi_k\}_{k\in \N}$ in $\H$ to an orthonormal sequence $\{W_{\phi_k}\}_{k=1}^{\infty}$ in $\L^2(\H,\mu)$, since
\begin{align*}
\la W_{\phi_j}, W_{\phi_k} \ra_{\L^2(\H, \mu)} = \la \phi_j, \phi_k\ra = \delta_{jk}.
\end{align*}
Furthermore, the random variables $\{\W_{\phi_k}\}_{k=1}^N$ are {\it independent}
	(\cite{DaPrato:2006}, Proposition 1.28).
	
{\bf White noise mapping via finite-rank orthogonal projections}.
$W_z$ can be 
expressed  explicitly in terms of the finite-rank orthogonal projections
$P_N = \sum_{k=1}^N e_k \otimes e_k$
onto the $N$-dimensional subspaces of $\H$
spanned by $\{e_k\}_{k=1}^N$, $N \in \N$, where $\{e_k\}_{k \in \N}$ are the orthonormal eigenvectors of $Q$.
For any $z \in \H$, we have
\begin{align}
P_Nz = \sum_{k=1}^N\la z,e_k \ra e_k \imply Q^{-1/2}P_Nz = \sum_{k=1}^N\frac{1}{\sqrt{\lambda_k}}\la z, e_k\ra e_k.
\end{align}
Thus $Q^{-1/2}P_Nz$ is always well-defined $\forall z \in \H$. 
Furthermore, for all $x,y \in \H$, 
\begin{align}
\la Q^{-1/2}P_Nx, y\ra = \sum_{j=1}^N\frac{1}{\sqrt{\lambda_j}}\la x, e_j\ra \la y, e_j\ra =\la x, Q^{-1/2}P_Ny\ra.
\end{align}
In other words, the operator $Q^{-1/2}P_N:\H \mapto \H$ is bounded and self-adjoint $\forall N \in \N$.
Since the 
sequence $\{P_Nz\}_{N\in \N}$ converges to $z$ in $\H$, we
have, in the $\L^2(\H,\mu)$ sense,
\begin{align}
W_z(x) = \lim_{N \approach \infty}W_{P_Nz}(x) = \lim_{N \approach \infty}\la x-m, Q^{-1/2}P_Nz\ra.
\end{align}
{\bf The Radon-Nikodym derivatives between Gaussian measures}.
Given their importance, these objects
have been studied
extensively,
e.g. \cite{Capon:Radon1964,Shepp:1966Radon,Henrich:Gaussian1972,DaPrato:2006,DaPrato:PDEHilbert,Bogachev:Gaussian}.
However, the explicit formulas available in the literature generally consider two separate cases, namely two Gaussian measures both with {\it mean zero} or {\it with the same covariance operator}.
We now present an explicit formula for the general case.

In the following, let $Q,R$ be two self-adjoint, positive trace class operators on $\H$ such that $\ker(Q) = \ker(R) = \{0\}$. Let $m_1, m_2 \in \H$. 
A fundamental result in the theory of Gaussian measures is the Feldman-Hajek Theorem \cite{Feldman:Gaussian1958}, \cite{Hajek:Gaussian1958}, which states that 
two Gaussian measures $\mu = \Ncal(m_1,Q)$ and 
$\nu = \Ncal(m_2, R)$
are either mutually singular or mutually equivalent.
The necessary and sufficient conditions for
the equivalence of the two Gaussian measures $\nu$ and $\mu$ are given by the following.

\begin{theorem}
	[\cite{Bogachev:Gaussian}, Corollary 6.4.11, \cite{DaPrato:PDEHilbert}, Theorems  1.3.9 and 1.3.10]
	\label{theorem:Gaussian-equivalent}
	Let $\H$ be a separable Hilbert space. Consider two Gaussian measures $\mu = \Ncal(m_1, Q)$ and
	$\nu = \Ncal(m_2, R)$ on $\H$. Then $\mu$ and $\nu$ are equivalent if and only if the following
	 hold
	\begin{enumerate}
		\item $m_2 - m_1 \in \Im(Q^{1/2})$.
		\item There exists  $S \in  \Sym(\H) \cap \HS(\H)$, without the eigenvalue $1$, such that
		\begin{align}
		\label{equation:RQ-equivalent}
		R = Q^{1/2}(I-S)Q^{1/2}.
		\end{align}
		
	\end{enumerate}
\end{theorem}
For any $A \in \L(\H)$, we have $\Im(A) = \Im((AA^{*})^{1/2})$
\cite{Fillmore:1971Operator}, thus Eq.(\ref{equation:RQ-equivalent}) implies
\begin{align}
\Im(R^{1/2}) = \Im((Q^{1/2}(I-S)Q^{1/2})^{1/2}) = \Im(Q^{1/2}(I-S)^{1/2}) = \Im(Q^{1/2}).
\end{align}
We assume from now on that $\mu$ and $\nu$ are equivalent.
In Corollary 6.4.11 in \cite{Bogachev:Gaussian}, an explicit formula for the Radon-Nikodym derivative $\frac{d\nu}{d\mu}$
is given when $m_1 = m_2 = 0$.  In Proposition 1.3.11 in \cite{DaPrato:PDEHilbert},
an explicit formula is given when $m_1 = m_2 =0 $ and $S$ is trace class.
In the following, we present an explicit formula for the general case.

Let $\{\alpha_k\}_{k \in \N}$ be the eigenvalues of $S$, with corresponding orthonormal eigenvectors $\{\phi_k\}_{k \in \N}$, which form an orthonormal basis in $\H$.
The following result expresses the Radon-Nikodym derivative $\frac{d\nu}{d\mu}$ in terms of the $\alpha_k$'s and
$\phi_k$'s.

\begin{theorem}
\label{theorem:radon-nikodym-infinite}
Let $\mu = \Ncal(m_1, Q)$, $\nu = \Ncal(m_2,R)$, with $m_2 - m_1 \in \Im(Q^{1/2})$, $R = Q^{1/2}(I-S)Q^{1/2}$.
The Radon-Nikodym derivative $\frac{d\nu}{d\mu}$
 is given by
\begin{align}
\frac{d\nu}{d\mu}(x) = \exp\left[-\frac{1}{2}\sum_{k=1}^{\infty}\Phi_k(x)\right]\exp\left[-\frac{1}{2}||(I-S)^{-1/2}Q^{-1/2}(m_2 - m_1)||^2\right],
\end{align}
where for each $k \in \N$
\begin{align}
\label{equation:Phik}
\Phi_k = \frac{\alpha_k}{1-\alpha_k}W^2_{\phi_k} - \frac{2}{1-\alpha_k}\la Q^{-1/2}(m_2-m_1), \phi_k\ra W_{\phi_k}+ \log(1-\alpha_k).
\end{align}
The series $\sum_{k=1}^{\infty}\Phi_k$ converges in $\L^1(\H,\mu)$ and $\L^2(\H,\mu)$ and the function $s(x) = \exp\left[-\frac{1}{2}\sum_{k=1}^{\infty}\Phi_k(x)\right] \in \L^1(\H, \mu)$.
\end{theorem}

{\bf Special case}. For $m_1 = m_2 = 0$, Theorem \ref{theorem:radon-nikodym-infinite} gives
\begin{align}
\frac{d\nu}{d\mu}(x) = \exp\left\{-\frac{1}{2}\sum_{k=1}^{\infty}\left[\frac{\alpha_k}{1-\alpha_k}W^2_{\phi_k}(x) + \log(1-\alpha_k)\right]\right\}.
\end{align}
This is essentially Eq.~(6.4.13) in Corollary 6.4.11 in \cite{Bogachev:Gaussian}.

\begin{corollary}
\label{corollary:radon-nikodym-traceclass}
Assume the hypothesis of Theorem \ref{theorem:radon-nikodym-infinite}. Assume 
further that $S$ is trace class. The Radon-Nikodym derivative of $\nu$ with respect to $\mu$ is given by
\begin{align}
\frac{d\nu}{d\mu}(x) &= [\det(I-S)]^{-1/2}
\\
& \times \exp\left\{-\frac{1}{2}\la Q^{-1/2}(x-m_1), S(I-S)^{-1}Q^{-1/2}(x-m_1)\ra \right\}
\nonumber
\\
&\times \exp(\la Q^{-1/2}(x-m_1), (I-S)^{-1}Q^{-1/2}(m_2 - m_1)\ra)
\nonumber
\\
&\times \exp\left[-\frac{1}{2}||(I-S)^{-1/2}Q^{-1/2}(m_2 - m_1)||^2\right].
\nonumber
\end{align}
In the above expression,
\begin{align}
&\la Q^{-1/2}(x-m_1), S(I-S)^{-1}Q^{-1/2}(x-m_1)\ra 
\nonumber
\\
& \doteq \lim_{N \approach \infty} \la Q^{-1/2}P_N(x-m_1), S(I-S)^{-1}Q^{-1/2}P_N(x-m_1)\ra 
\\
&
\la Q^{-1/2}(x-m_1), (I-S)^{-1}Q^{-1/2}(m_2 - m_1)\ra 
\nonumber
\\
&\doteq \lim_{N \approach \infty} \la Q^{-1/2}P_N(x-m_1), (I-S)^{-1}Q^{-1/2}(m_2 - m_1)\ra,
\end{align}
with the limits being in the $\L^1(\H,\mu)$ and $\L^2(\H, \mu)$ sense, respectively.
\end{corollary}

{\bf Special case}. For $m_1 = m_2 = 0$ and $S$  trace class, Corollary 
\ref{corollary:radon-nikodym-traceclass} gives
\begin{align}
\frac{d\nu}{d\mu}(x) = [\det(I-S)]^{-1/2}\exp\left\{-\frac{1}{2}\la Q^{-1/2}x, S(I-S)^{-1}Q^{-1/2}x\ra \right\}.
\end{align}
This is precisely Proposition 1.3.11 in \cite{DaPrato:PDEHilbert}.

{\bf Special case}. If $Q = R$, then obviously $S = 0$ and Corollary 
\ref{corollary:radon-nikodym-traceclass} gives
\begin{align}
\frac{d\nu}{d\mu}(x) &=
\exp(\la Q^{-1/2}(x-m_1), Q^{-1/2}(m_2 - m_1)\ra)
\nonumber
\\
&\times \exp\left[-\frac{1}{2}||Q^{-1/2}(m_2 - m_1)||^2\right].
\\
& = \exp\left[\la (x-m_1), (m_2 - m_1)\ra_{Q}-\frac{1}{2}||(m_2 - m_1)||^2_{Q}\right].
\end{align}
This is precisely Theorem 6.14 in \cite{Stuart:Inverse2010}.

{\bf Special case: Radon-Nikodym derivative between Gaussian densities on $\R^n$}.
Let $P_1 \sim \Ncal(\mu_1, \Sigma_1)$, $P_2 \sim \Ncal(\mu_2, \Sigma_2)$, with $\mu_1, \mu_2 \in \R^n$, 
$\Sigma_1, \Sigma_2
\in \Sym^{++}(n)$. Let $S\in \Sym(n)$ be such that $\Sigma_2 = \Sigma_1^{1/2}(I-S)\Sigma_1^{1/2}$, then one can verify directly that
\begin{align}
\frac{dP_2}{dP_1}(x) &= [\det(I-S)]^{-1/2}\exp(-\Phi(x)), \;\;\; x \in \R^n, \;\;\; \text{where}
\\
\Phi(x) 
& = \frac{1}{2}\la \Sigma_1^{-1/2}(x-\mu_1), S(I-S)^{-1}\Sigma_1^{-1/2}(x-\mu_1)\ra
\\
& \; \;\; + \la \Sigma_1^{-1/2}(x-\mu_1), (I-S)^{-1}\Sigma_1^{-1/2}(\mu_1 - \mu_2) \ra
\nonumber
\\
& \; \;\;  + \frac{1}{2}\la \Sigma_1^{-1/2}(\mu_2 - \mu_1), (I-S)^{-1}\Sigma_1^{-1/2}(\mu_2 - \mu_1)\ra.
\nonumber
\end{align}

To prove Theorem \ref{theorem:radon-nikodym-infinite}, we first prove the following.

\begin{proposition}
	\label{proposition:positiveIS}
	Assume that $\ker(Q) = \ker(R) = \{0\}$ and that $R = Q^{1/2}(I-S)Q^{1/2}$, 
	where $S \in \Sym(\H) \cap \HS(\H)$. Then
	the operator
	$(I-S)$ is necessarily strictly positive, that is $\la x, (I-S)x\ra > 0$ $\forall 0 \neq x \in \H$.
\end{proposition}
\begin{proof}
	For any $x \in \H$, we have
	\begin{align*}
	\la x, Rx\ra = \la x, Q^{1/2}(I-S)Q^{1/2}x\ra = \la Q^{1/2}x, (I-S)Q^{1/2}x\ra \geq 0,
	\end{align*}
	with equality if and only if $x = 0$, since $\ker(R) = \{0\}$.
	Thus we have
	\begin{align*}
	\la y, (I-S)y \ra \geq 0 \;\;\forall y \in \Im(Q^{1/2}),
	\end{align*}
	with equality if and only if $y = 0$.  Since $\ker(Q) = \{0\}$, $\Im(Q^{1/2})$ is dense in $\H$ and 
	$\forall y \in \H$, 
	$\exists$
	a sequence $\{y_n\}_{n\in\N}$ in $\Im(Q^{1/2})$ such that 
	$\lim_{n \approach \infty}||y_n - y|| = 0$. One has
	\begin{align*}
	&|\la y_n, (I-S)y_n\ra - \la y, (I-S)y\ra| \leq |\la y_n - y, (I-S)y_n\ra| + |\la y, (I-S)(y_n -y)\ra|
	\\
	&\leq ||y_n - y||\;||I-S||[||y_n|| + ||y||] \approach 0 \;\text{as}\; n \approach \infty.
	\end{align*}
	It follows that
	$\la y, (I-S)y\ra = \lim_{n \approach \infty}\la y_n, (I-S)y_n\ra \geq 0$.
	Hence the operator $I-S$ is self-adjoint, positive on $\H$. 
	
	Let us show
	that $I-S$ is strictly positive. Assume that $\exists y \neq 0\in \H$ such that $
	\la y, (I-S)y\ra = 0$, then $y \notin \Im(Q^{1/2})$ and there exists a sequence $\{y_n\}_{n \in \N}$ in
	$\Im(Q^{1/2})$ such that $\lim_{n \approach \infty}||y_n - y|| = 0$ and
	$\lim_{n \approach \infty}\la y_n, (I-S)y_n\ra = 0$.
	Equivalently, there exists a sequence $\{x_n\}_{n\in \N}$ in $\H$ such that $y_n = Q^{1/2}x_n$ and
	\begin{align*}
	\lim_{n \approach \infty}\la Q^{1/2}x_n, (I-S)Q^{1/2}x_n \ra = \lim_{n \approach \infty} \la x_n, Rx_n\ra  = \lim_{n\approach \infty}||R^{1/2}x_n||^2= 0.
	\end{align*}
	This implies that for any $z \in \H$, we have
	\begin{align*}
	\lim_{n \approach \infty}\la x_n, R^{1/2}z\ra = \lim_{n\approach \infty} \la R^{1/2}x_n, z\ra = 0.
	\end{align*}
	Since $\ker(R) = \{0\}$, $\Im(R^{1/2})$ is dense in $\H$ and thus
	$\lim_{n \approach \infty}\la x_n, z\ra = 0 \;\;\forall z \in \H$.
	Thus the sequence $\{x_n\}_{n \in \N}$ converges weakly to zero in $\H$. Then for any $z\in \H$,
	\begin{align*}
	\lim_{n \approach \infty}\la y_n, z\ra = \lim_{n \approach \infty}\la Q^{1/2}x_n , z\ra = \lim_{n \approach \infty}\la x_n , Q^{1/2}z\ra = 0.
	\end{align*}
	Thus the sequence $\{y_n\}_{n \in \N}$ also converges weakly to zero in $\H$. Since we already assume that $y_n$ converges strongly, and hence weakly, to $y \in \H$,
	by the uniqueness of the weak limit, we must have $y = 0$, contradicting our prior assumption that $y \neq 0$.
\qed \end{proof}



In the following, 
we make use of the Vitali Convergence Theorem
(see e.g. \cite{Folland:Real,Rudin:RealComplex}).
Let 
$(\X, \F, \mu)$ be a positive measurable space. A sequence of functions $\{f_n\}_{n\in \N} \in \L^1(\X,\mu)$ 
is said to be {\it uniformly integrable} if $\forall \epsilon > 0$ $\exists \delta > 0$ such that
\begin{align}
\sup_{n \in \N}\int_{E}|f_n|d\mu < \epsilon \;\;\; \text{whenever} \;\;\; \mu(E) < \delta, E \in \Fcal.
\end{align}

\begin{theorem}
[\textbf{Vitali Convergence Theorem}]
Assume that $(\X, \F, \mu)$ is a positive measurable space with $\mu(\X) < \infty$.
Let $\{f_n\}_{n\in \N}$ be a sequence of functions that are uniformly integrable on $\X$, with
$f_n \approach f$ a.e. and $|f| < \infty$ a.e.. Then $f \in \L^1(\X, \mu)$ and $||f_n - f||_{\L^1(\X, \mu)} \approach 0$.
\end{theorem}

\begin{proposition} 
	\label{proposition:exponential-quadratic-whitenoise}
	Let $g \in \H$. Let $c_1\in \R$, $c_2 \in \R$ be such that $c_1||g||^2 < 1$. Then
	\begin{align}
	&\int_{\H}\exp\left[\frac{1}{2}c_1W_{g}^2(x) + c_2W_{g}(x) \right]\Ncal(m,Q)(dx)
	\\
	&=\frac{1}{(1-c_1||g||^2)^{1/2}}\exp\left(\frac{c_2^2||g||^2}{2(1-c_1||g||^2)}\right).
	\nonumber
	\end{align}
\end{proposition}
{\bf Special case}. For $c_1 = 0$, Proposition \ref{proposition:exponential-quadratic-whitenoise} gives
\begin{align}
\int_{\H}\exp[c_2W_{g}(x)]\Ncal(m,Q)(dx) = \exp\left(\frac{c_2^2}{2}||g||^2\right).
\end{align}
With $c_2 = 1$, the above formula gives Proposition 1.2.7 in \cite{DaPrato:PDEHilbert}.

The proof of Proposition \ref{proposition:exponential-quadratic-whitenoise} 
requires the following results. The first one,
Lemma \ref{lemma:inverse-rank-one}, can be directly verified.

\begin{lemma}
	\label{lemma:inverse-rank-one}
	Let $u \in \H$  and $c \in \R$ be such that $c||u||^2 < 1$. Then
	the operator $I - c(u \otimes u)$ is invertible and 
	\begin{align}
	[I - c(u \otimes u)]^{-1} = I + \frac{c(u \otimes u)}{1-c||u||^2}.
	\end{align}
	In particular,
	$[I - c (u \otimes u)]^{-1}u = \frac{1}{1-c||u||^2}u$.
\end{lemma}

The second is the following result from \cite{DaPrato:PDEHilbert}.

\begin{theorem}
	[\cite{DaPrato:PDEHilbert}, Proposition 1.2.8]
	\label{theorem:exponential-quadratic-Gaussian}
	Assume that $M$ is a self-adjoint operator on $\H$ such that $\la Q^{1/2}MQ^{1/2} x, x \ra < ||x||^2$ $\forall 
	x \in \H, x \neq 0$. Let $b \in \H$. Then 
	\begin{align}
	&\int_{\H}\exp\left(\frac{1}{2}\la M y, y\ra + \la b,y\ra\right)\Ncal(0,Q)(dy)
	\\
	& = [\det(I - Q^{1/2}MQ^{1/2})]^{-1/2}\exp\left(\frac{1}{2}||(I-Q^{1/2}MQ^{1/2})^{-1/2}Q^{1/2}b||^2\right).
	\nonumber
	\end{align}
\end{theorem}

\begin{proof}
	{\textbf{(of Proposition \ref{proposition:exponential-quadratic-whitenoise}})}
	It suffices to prove for
	$m = 0$.
	We apply
	Theorem \ref{theorem:exponential-quadratic-Gaussian} as follows.
	Let $\{P_N\}_{N \in \N}$, $P_N = \sum_{j=1}^N e_j \otimes e_j$ be the sequence of orthogonal projections in $\H$ corresponding to the eigenvectors $\{e_j\}_{j\in \N}$ of $Q$.
	Consider the limit
	\begin{align*}
	W_{g}(x) = \lim_{N \approach \infty}W_{P_Ng}(x) = \lim_{N \approach \infty}\la Q^{-1/2}P_Ng, x\ra \;\;\;\Ncal(0,Q) \;\;\text{a.e.}.
	\end{align*}
	Let $N \in \N$ be fixed. We have
	\begin{align*}
	W_{P_Ng}^2(x) = \la Q^{-1/2}P_Ng, x\ra^2 = \la [(Q^{-1/2}P_Ng) \otimes (Q^{-1/2}P_Ng)]x, x\ra.
	\end{align*}
	Let $M = c_1[(Q^{-1/2}P_Ng) \otimes (Q^{-1/2}P_Ng)]$, $b = c_2(Q^{-1/2}P_Ng)$. Then for any $x \in \H$,
	\begin{align*}
	Q^{1/2}MQ^{1/2}x = c_1Q^{1/2}(Q^{-1/2}P_Ng)\la (Q^{-1/2}P_Ng), Q^{1/2}x\ra = c_1P_Ng\la P_Ng, x\ra, 
	\end{align*}
	which implies that
	$Q^{1/2}MQ^{1/2} = c_1(P_Ng) \otimes (P_Ng)$,
	which is a rank-one operator with eigenvalue $c_1||P_Ng||^2$. If $c_1 < 0$, then obviously 
	$c_1||P_Ng||^2 < 1$. If $c_1 \geq 0$, then
	$c_1||P_Ng||^2 \leq c_1||g||^2 < 1$. Also,
	$Q^{1/2}b = c_2P_Ng$.
	By Lemma \ref{lemma:inverse-rank-one}, the operator $(I- Q^{1/2}MQ^{1/2})$ is invertible, with
	\begin{align*}
	(I- Q^{1/2}MQ^{1/2})^{-1}Q^{1/2}b = c_2 [I - c_1(P_Ng \otimes P_Ng)]^{-1}P_Ng = \frac{c_2}{1-c_1||P_Ng||^2}P_Ng.
	\end{align*}
	It follows that
	\begin{align*}
	&||(I-Q^{1/2}MQ^{1/2})^{-1/2}Q^{1/2}b||^2 = \la Q^{1/2}b, (I-Q^{1/2}MQ^{1/2})^{-1}Q^{1/2}b\ra
	\\
	& = \left\la c_2 P_Ng, \frac{1}{1-c_1||P_Ng||^2}c_2 P_Ng\right\ra = \frac{c_2^2||P_Ng||^2}{1-c_1||P_Ng||^2}.
	\end{align*}
	By the assumption that $c_1||g||^2 < 1$, there exists $p > 1$ such that
	$pc_1||g||^2 < 1$, so that $pc_1||P_Ng||^2 < 1$ $\forall N \in \N$.
	Hence by Theorem \ref{theorem:exponential-quadratic-Gaussian}, we have
	\begin{align*}
	&\int_{\H}\exp\left[\frac{1}{2}pc_1W_{P_Ng}^2(x) + pc_2W_{P_Ng}(x) \right]\Ncal(0,Q)(dx)
	\\
	& = \int_{\H}\exp\left[\frac{1}{2}\la pMx,x\ra + \la pb,x\ra \right]\Ncal(0,Q)(dx)
	\\
	& = [\det(I - pQ^{1/2}MQ^{1/2})]^{-1/2}\exp\left(\frac{1}{2}||(I-pQ^{1/2}MQ^{1/2})^{-1/2}Q^{1/2}pb||^2\right)
	\\
	& = \frac{1}{(1-pc_1||P_Ng||^2)^{1/2}}\exp\left(\frac{p^2c_2^2||P_Ng||^2}{2(1-pc_1||P_Ng||^2)}\right).
	\end{align*}
	Taking limit as $N \approach \infty$ gives
	\begin{align*}
	&\lim_{N \approach \infty}\frac{1}{(1-pc_1||P_Ng||^2)^{1/2}}\exp\left(\frac{p^2c_2^2||P_Ng||^2}{2(1-pc_1||P_Ng||^2)}\right)
	\\
	&= \frac{1}{(1-pc_1||g||^2)^{1/2}}\exp\left(\frac{p^2c_2^2||g||^2}{2(1-pc_1||g||^2)}\right) < \infty.
	\end{align*}
	Hence it follows, by applying from H\"older's Inequality, that the sequence of functions
	$\left\{\exp\left[\frac{1}{2}c_1W_{P_Ng}^2(x) + c_2W_{P_Ng}(x) \right]\right\}_{N \in \N}$
	is uniformly integrable. Thus we can apply Vitali's Convergence Theorem to obtain
	%
	\begin{align*}
	&\int_{\H}\exp\left[\frac{1}{2}c_1W_{g}^2(x) + c_2W_{g}(x) \right]\Ncal(0,Q)(dx)
	\\
	& = \int_{\H} \lim_{N \approach \infty} \exp\left[\frac{1}{2}c_1W_{P_Ng}^2(x) + c_2W_{P_Ng}(x) \right]\Ncal(0,Q)(dx)
	\\
	&= \lim_{N \approach \infty}\int_{\H}\exp\left[\frac{1}{2}c_1W_{P_Ng}^2(x) + c_2W_{P_Ng}(x) \right]\Ncal(0,Q)(dx)
	\\
	& = \lim_{N \approach \infty}\frac{1}{(1-c_1||P_Ng||^2)^{1/2}}\exp\left(\frac{c_2^2||P_Ng||^2}{2(1-c_1||P_Ng||^2)}\right)
	\\
	& = \frac{1}{(1-c_1||g||^2)^{1/2}}\exp\left(\frac{c_2^2||g||^2}{2(1-c_1||g||^2)}\right) < \infty.
	\end{align*}
\qed \end{proof}

\begin{proposition}
\label{proposition:L1mu-dnu-dmu}
Assume the hypothesis of Theorem \ref{theorem:radon-nikodym-infinite}.
There exists $p > 1$ such that $I+(p-1)S > 0$. Define
 $s(x) = \exp\left[-\frac{1}{2}\sum_{k=1}^{\infty}\Phi_k(x)\right]$, where $\Phi_k$ is defined by Eq.~(\ref{equation:Phik}) in Theorem \ref{theorem:radon-nikodym-infinite}. Then $s \in \L^q(\H,\mu)$ for all $q$ satisfying 
 $0 < q < p$, 
 with
\begin{align}
||s||_{\L^q(\H, \mu)}^q &= \exp\left(\frac{q^2}{2}||[(I-S)(I+(q-1)S)]^{-1/2}Q^{-1/2}(m_2-m_1)||^2\right)
\nonumber
\\
& \times (\det[(I-S)^{q-1}(I+(q-1)S)])^{-1/2}.
\end{align} 
In particular, for $q=1$,
\begin{align}
||s||_{\L^1(\H,\mu)} =  \exp\left(\frac{1}{2}||(I-S)^{-1/2}Q^{-1/2}(m_2-m_1)||^2\right).
\end{align}
Furthermore, for $s_N(x) = \exp\left[-\frac{1}{2}\sum_{k=1}^{N}\Phi_k(x)\right]$, the sequence $\{s_N^q\}_{N \in \N}$ is uniformly integrable
on $(\Bsc(\H),\mu)$ for
$0 < q < p$.
\end{proposition}
\begin{proof}
For each fixed $k \in \N$, we recall that the function $\Phi_k$ is given by
\begin{align*}
\Phi_k = \frac{\alpha_k}{1-\alpha_k}W^2_{\phi_k} - \frac{2}{1-\alpha_k}\la Q^{-1/2}(m_2-m_1), \phi_k\ra W_{\phi_k}+ \log(1-\alpha_k).
\end{align*}
We first claim that there exists $p > 1$ such that $1+(p-1)\alpha_k > 0$ $\forall k \in \N$. Since $\lim_{k \approach \infty}\alpha_k = 0$, there exists $\mu > 0$ such that $\alpha_k \geq - \mu$ $\forall k \in \N$. Let $p$ be such that
$1 < p < \frac{1}{\mu}+1$, so that $(p-1)\mu < 1$. Then 
\begin{align*}
1 +(p-1)\alpha_k \geq 1-(p-1)\mu > 0 \;\;\forall k \in \N, \text{\; or equivalently \;} I+(p-1)S > 0.
\end{align*}
Similarly, we have
$I+ (q-1)S > 0$ for all $q$ satisfying $1 \leq q < p$.
Recall that since $I-S > 0$, we have $\alpha_k < 1$ $\forall k \in \N$.
For $q$ satisfying $0 < q < 1$, we have $1-(1-q)\alpha_k > 0$ when $\alpha_k < 0$ and
$1-(1-q)\alpha_k \geq 1-\alpha_k > 0$ for $0 \leq \alpha_k < 1$.
It follows that $I+(q-1)S > 0$ for all $q$ satisfying $0 < q < 1$. Hence
\begin{align*}
I +(q-1)S > 0 \;\;\;\text{for all $q$ satisfying $0 < q < p$}.
\end{align*}
For each
$k \in \N$,
by Proposition \ref{proposition:exponential-quadratic-whitenoise},
with $c_1 = - \frac{p\alpha_k}{1-\alpha_k},
c_2 = \frac{p}{1-\alpha_k}\la Q^{-1/2}(m_2-m_1), \phi_k\ra$,
\begin{align*}
&\int_{\H}\exp\left[-\frac{p}{2}\Phi_k(x)\right]\mu(dx) 
\\
&= \frac{1}{(1-\alpha_k)^{p/2}}\int_{\H} \exp\left[-\frac{1}{2}\frac{p\alpha_k}{1-\alpha_k}W^2_{\phi_k} + \frac{p}{1-\alpha_k}\la Q^{-1/2}(m_2-m_1), \phi_k\ra W_{\phi_k}\right]\mu(dx)
\\
& =\frac{1}{(1-\alpha_k)^{p/2}}\left[\sqrt{\frac{1-\alpha_k}{1+(p-1)\alpha_k}}\exp\left(\frac{p^2\la Q^{-1/2}(m_2-m_1), \phi_k\ra^2}{2(1-\alpha_k)(1+(p-1)\alpha_k)}\right)\right] 
\\
&= \frac{1}{(1-\alpha_k)^{(p-1)/2} (1+(p-1)\alpha_k)^{1/2}}\exp\left(\frac{p^2\la Q^{-1/2}(m_2-m_1), \phi_k\ra^2}{2(1-\alpha_k)(1+(p-1)\alpha_k)}\right)
\end{align*}
For each $N \in \N$, consider the nonnegative function $s_N(x) = \exp\left[-\frac{1}{2}\sum_{k=1}^{N}\Phi_k(x)\right]$. By the independence of
the functions $W_{\phi_k}$, we have
\begin{align*}
&\int_{\H}s_N^p(x)d\mu(x) = \prod_{k=1}^N\int_{\H}\exp\left[-\frac{p}{2}\Phi_k(x)\right]d\mu(x) 
\\
&= \prod_{k=1}^N
\frac{1}{(1-\alpha_k)^{(p-1)/2} (1+(p-1)\alpha_k)^{1/2}}\exp\left(\frac{p^2\la Q^{-1/2}(m_2-m_1), \phi_k\ra^2}{2(1-\alpha_k)(1+(p-1)\alpha_k)}\right)
\\
& = \exp\left(\frac{p^2}{2}\sum_{k=1}^{N}\frac{\la Q^{-1/2}(m_2-m_1), \phi_k\ra^2}{(1-\alpha_k)(1+(p-1)\alpha_k)}\right)
\\
&\times \exp\left(-\frac{1}{2}\sum_{k=1}^{N}[(p-1)\log(1-\alpha_k) + \log(1+(p-1)\alpha_k)]\right).
\end{align*}
Since $-1/(p-1) < \alpha_k < 1$ $\forall k \in \N$, by Lemma \ref{lemma:inequality-log-2} we have
\begin{align*}
-[(p-1)\log(1-\alpha_k) + \log(1+(p-1)\alpha_k)] \geq 0, \;\; \forall k \in \N.
\end{align*}
Since $\sum_{k=1}^{\infty}\alpha_k^2 < \infty$, 
$\exists$
$N_0 \in \N$ such that $|\alpha_k|  < 1/2$ $\forall k \geq N_0$. Then by Lemma \ref{lemma:inequality-log-2},
\begin{align*}
-[(p-1)\log(1-\alpha_k) + \log(1+(p-1)\alpha_k)] \leq p(p-1)\alpha_k^2 \;\;\forall k \geq N_0.
\end{align*}
Thus it follows that
\begin{align*}
0 \leq -\sum_{k=N_0}^{\infty}[(p-1)\log(1-\alpha_k) + \log(1+(p-1)\alpha_k)] 
\leq p(p-1)\sum_{k=N_0}^{\infty}\alpha_k^2 < \infty.
\end{align*}
It follows that the sequence
$\left\{\exp\left(-\frac{1}{2}\sum_{k=1}^{N}[(p-1)\log(1-\alpha_k) + \log(1+(p-1)\alpha_k)]\right)\right\}_{N \in \N}$ 
is increasing towards the limit $\exp\left(-\frac{1}{2}\sum_{k=1}^{\infty}[(p-1)\log(1-\alpha_k) + \log(1+(p-1)\alpha_k)]\right)$. Hence the sequence $\{\int_{\H}s_N^p(x)d\mu(x)\}_{N \in \N}$ is increasing towards the limit
\begin{align*}
\lim_{N \approach \infty}\int_{\H}s_N^p(x)d\mu(x)
&= \exp\left(\frac{p^2}{2}\sum_{k=1}^{\infty}\frac{\la Q^{-1/2}(m_2-m_1), \phi_k\ra^2}{(1-\alpha_k)(1+(p-1)\alpha_k)}\right)
\\
&\times \exp\left(-\frac{1}{2}\sum_{k=1}^{\infty}[(p-1)\log(1-\alpha_k) + \log(1+(p-1)\alpha_k)]\right)
\\
& = \exp\left(\frac{p^2}{2}||[(I-S)(I+(p-1)S)]^{-1/2}Q^{-1/2}(m_2-m_1)||^2\right)
\\
& \times (\det[(I-S)^{p-1}(I+(p-1)S)])^{-1/2} < \infty.
\end{align*}
By H\"older's Inequality, for any $0 < q < p$, for any set $A \in \Bsc(\H)$, we have
\begin{align*}
\int_{A}s_N^q(x)d\mu(x) &= \int_{\H}\1_A s_N^q(x)d\mu(x) \leq ||\1_A||_{\L^{\frac{p}{p-q}}(\H, \mu)}||s_N^q||_{\L^{\frac{p}{q}}(\H,\mu)}
\\
& = (\mu(A))^{\frac{p-q}{p}}\left(\int_{\H}s_N^p(x)d\mu(x)\right)^\frac{q}{p}.
\end{align*}
Combining
with the
limit for 
$\{\int_{\H}s_N^p(x)d\mu(x)\}_{N \in \N}$, 
this shows
that the sequence $\{s_N^q(x)\}$ is uniformly integrable on $(\Bsc(\H), \mu)$.
By Vitali's Convergence Theorem,
\begin{align*}
&\int_{\H}s^q(x)d\mu(x) = \int_{\H}\lim_{N \approach \infty}s_N^q(x)d\mu(x) = 
\lim_{N \approach \infty}\int_{\H}s_N^q(x)d\mu(x)
\\
& = \exp\left(\frac{q^2}{2}||[(I-S)(I+(q-1)S)]^{-1/2}Q^{-1/2}(m_2-m_1)||^2\right)
\\
& \times (\det[(I-S)^{q-1}(I+(q-1)S)])^{-1/2} < \infty.
\end{align*}
Thus it follows that $s(x) = \exp\left(-\frac{1}{2}\sum_{k=1}^{\infty}\Phi_k(x)\right) \in \L^q(\H, \mu)$.
In particular, for $q=1$, 
\begin{align*}
||s||_{\L^1(\H,\mu)} = \int_{\H}s(x)d\mu(x) = \exp\left(\frac{1}{2}||(I-S)^{-1/2}Q^{-1/2}(m_2-m_1)||^2\right) < \infty.
\end{align*}
\qed \end{proof}

\begin{lemma}
	\label{lemma:norm-W4}
	For any $a \in \H$, we have $W_a^2 \in \L^2(\H, \mu)$. For any $a, b \in \H$,
	\begin{align}
	\int_{\H}W^2_a(x)W^2_b(x)\Ncal(m, Q)(dx) &= ||a||^2 ||b||^2 + 2\la a, b\ra^2.
	\end{align}
	In particular, for $a= b$,
	$\int_{\H}W_a^4(x)\Ncal(m,Q)(dx)  = 3||a||^4$.
	For any two $a,b\in \H$, 
	{\small
	\begin{align}
	&\int_{\H}(W_a^2(x)-1)(W_b^2(x)-1)\Ncal(m,Q)(dx)  = ||a||^2||b||^2 + 2\la a, b\ra^2 - ||a||^2-||b||^2+1.
\\
	&\frac{1}{2}\int_{\H}(W_a^2(x)-1)(W_b^2(x)-1)\Ncal(m,Q)(dx)  =  \la a, b\ra^2, \;\; \text{for $||a|| = ||b||=1$}.
	\end{align}
}
	In particular, an orthonormal sequence $\{a_k\}_{k \in \N}$ in $\H$ gives rise to an 
	orthonormal sequence $\{\frac{1}{\sqrt{2}}(W_{a_k}^2-1)\}_{k\in \N}$ in $\L^2(\H,\mu)$
	(see also \cite{DaPrato:PDEHilbert}, Proposition 1.2.6).
\end{lemma}
\begin{proof}
	For $a, b \in Q^{1/2}(\H)$, by Lemma \ref{lemma:Gaussian-integral-1}, we have
	\begin{align*}
	&\int_{\H}W_{a}^2(x)W_b^2(x)\Ncal(m,Q)(dx)
	= \int_{\H}\la x-m, Q^{-1/2}a\ra^2\la x-m, Q^{-1/2}b\ra^2\Ncal(m,Q)(dx) 
	\\
	& = [\la Q^{-1/2}a, Q(Q^{-1/2}a)\ra \la Q^{-1/2}b, Q(Q^{-1/2}b)\ra + 2 \la Q^{-1/2}a, Q(Q^{-1/2}b)\ra^2)]
	\\
	& = ||a||^2||b||^2 + 2\la a, b\ra^2.
	\end{align*}
	Let $a \in \H$.
	Since
	$Q^{1/2}(\H)$ is dense in $\H$, let $\{a_n\}_{n \in \N}$ be a Cauchy sequence in $\H$ with $a_n \in Q^{1/2}(\H)$ and $\lim_{ n\approach \infty}||a_n -a||= 0$.
	Then $W_{a_n} \approach W_a$ in $\L^2(\H, \mu)$. The previous identity gives
	\begin{align*}
	||W_{a_n}^2 - W_{a_m}^2||^2_{\L^2(\H,\mu)} = 3||a_n||^4+3||a_m||^4 - 2||a_n||^2||a_m||^2 - 4\la a_n,a_m\ra^2
	\end{align*}
	The hypothesis $\lim_{n,m \approach \infty}||a_n - a_m|| = 0$ and the above identity show that
	$\lim_{n,m  \approach \infty} ||W_{a_n}^2 - W_{a_m}^2||_{\L^2(\H,\mu)} = 0$.
	Thus $\{W_{a_n}^2\}_{n \in \N}$ is a Cauchy sequence in $\L^2(\H,\mu)$ and hence converges
	to a unique element in $\L^2(\H,\mu)$, which must be $W_a^2$. Thus $W_a^2 \in \L^2(\H,\mu)$.
	
	Let $b \in \H$ with the corresponding Cauchy sequence $\{b_n\}_{n \in \N}$, $b_n \in Q^{1/2}(\H)$.
	Then 
	\begin{align*}
	&\int_{\H}W^2_a(x)W^2_b(x)\Ncal(m, Q)(dx) = \la W_a^2, W_b^2\ra_{\L^2(\H, \mu)} =
	\lim_{n \approach \infty}\la W_{a_n}^2, W_{b_n}^2\ra_{\L^2(\H, \mu)} 
	\\
	&= \lim_{n \approach \infty}||a_n||^2||b_n||^2 + 2 \la a_n, b_n\ra^2
	= ||a||^2||b||^2 + 2\la a, b\ra^2.
	\end{align*}
	This give us the first and second identities.
	The third identity follows from the first by invoking the isometry
	$||W_a||^2_{\L^2(\H,\mu)} = 
	||a||^2 \;\;\; \forall a \in \H$.
\qed \end{proof}

\begin{lemma}
	\label{lemma:fN-convergence-L2mu}
	Consider the functions
	\begin{align}
	f_N = \sum_{k=1}^N\left[\frac{\alpha_k}{1-\alpha_k}W^2_{\phi_k} + \log(1-\alpha_k)\right],
	f =  \sum_{k=1}^{\infty}\left[\frac{\alpha_k}{1-\alpha_k}W^2_{\phi_k} + \log(1-\alpha_k)\right].
	\end{align}
	Then 
	$\lim_{N \approach \infty} ||f_N - f||_{\L^2(\H,\mu)} = 0,
	\;\;\;
	\lim_{N \approach \infty} ||f_N - f||_{\L^1(\H,\mu)} = 0$.
\end{lemma}
\begin{proof}
	By Lemma \ref{lemma:norm-W4},
	the functions $\{\frac{1}{\sqrt{2}}(W_{\phi_k}^2-1)\}_{k \in \N}$ are orthonormal in 
	$\L^2(\H,\mu)$.
	We rewrite $f_N$ as
	\begin{align*}
	f_N = \sum_{k=1}^{N}\left[\frac{\sqrt{2}\alpha_k}{1-\alpha_k}\frac{1}{\sqrt{2}}(W^2_{\phi_k}-1) + \frac{\alpha_k}{1-\alpha_k} +  \log(1-\alpha_k)\right].
	\end{align*}
	Consider the functions
	\begin{align*}
	h_N = \sum_{k=1}^{N}\left[\frac{\sqrt{2}\alpha_k}{1-\alpha_k}\frac{1}{\sqrt{2}}(W^2_{\phi_k}-1)\right],\;\;
	h = \sum_{k=1}^{\infty}\left[\frac{\sqrt{2}\alpha_k}{1-\alpha_k}\frac{1}{\sqrt{2}}(W^2_{\phi_k}-1)\right].
	\end{align*}
	Since $\sum_{k=1}^{\infty}\alpha_k^2 < \infty$, there exists $N_0 \in \N$ such that $|\alpha_k| < 1/2$ $\forall k > N_0$.
	By Lemma \ref{lemma:norm-W4},
	we have for all $N \geq N_0$,
	\begin{align*}
	||h_N - h||^2_{\L^2(\H,\mu)} = 2\sum_{k=N+1}^{\infty}\frac{\alpha_k^2}{(1-\alpha_k)^2} < 8\sum_{k=N+1}^{\infty}\alpha_k^2
	\approach 0 \;\;\text{as}\;\; N \approach \infty.
	\end{align*}
	Consider next the series
	\begin{align*}
	\sum_{k=1}^{\infty}\left[\frac{\alpha_k}{1-\alpha_k} +  \log(1-\alpha_k)\right] = \sum_{k=1}^{\infty}\frac{\alpha_k +(1-\alpha_k)\log(1-\alpha_k)}{1-\alpha_k}.
	\end{align*}
	By Lemma \ref{lemma:inequality-log}, we have , since $\alpha_k < 1$ $\forall k \in \N$,
	\begin{align*}
	0 \leq \alpha_k +(1-\alpha_k)\log(1-\alpha_k) \leq \alpha_k^2.
	\end{align*}
	It thus follows that for al $N \geq N_0$,
	\begin{align*}
	0 \leq \sum_{k=N+1}^{\infty}\left[\frac{\alpha_k}{1-\alpha_k} +  \log(1-\alpha_k)\right] \leq \sum_{k=N+1}^{\infty}\frac{\alpha_k^2}{1-\alpha_k} < 2\sum_{k=N+1}^{\infty}\alpha_k^2 \approach 0 
	\end{align*}
	as $N \approach \infty$. Thus the series $\sum_{k=1}^{\infty}\left[\frac{\alpha_k}{1-\alpha_k} +  \log(1-\alpha_k)\right]$ converges to a finite positive value. Together with $\lim_{N\approach \infty}||h_N - h||_{\L^2(\H,\mu)} = 0$,
	this implies that $\lim_{N \approach \infty}||f_N - f||_{\L^2(\H,\mu)} = 0$.
	Since $\mu$ is a probability measure, by H\"older's Inequality, we have
	$||f_N - f||_{\L^1(\H, \mu)} \leq \sqrt{\mu(\H)}||f_N - f||_{\L^2(\H, \mu)} = ||f_N - f||_{\L^2(\H, \mu)} \approach 0$
	as $N \approach \infty$. 
\qed \end{proof}

\begin{lemma}
	\label{lemma:gN-convergence-L2mu}
	Consider the functions
	\begin{align}
	g_N &= \sum_{k=1}^N\frac{1}{1-\alpha_k}\la Q^{-1/2}(m_2-m_1), \phi_k\ra W_{\phi_k}, \;\; N \in \N,
	\\
	g &= \sum_{k=1}^{\infty}\frac{1}{1-\alpha_k}\la Q^{-1/2}(m_2-m_1), \phi_k\ra W_{\phi_k}.
	\end{align}
	Then $g \in \L^2(\H, \mu)$, $g \in \L^1(\H, \mu)$, and
	\begin{align}
	\lim_{N \approach \infty}||g_N - g||_{\L^2(\H, \mu)} = 0,
	\;\;\;
	\lim_{N \approach \infty}||g_N - g||_{\L^1(\H, \mu)} = 0.
	\end{align}
\end{lemma}
\begin{proof}
	Since the functions $\{W_{\phi_k}\}_{k\in \N}$ are orthonormal in $\L^2(\H, \mu)$, we have
	\begin{align*}
	&||g||^2_{\L^2(\H,\mu)} = \sum_{k=1}^{\infty}\frac{1}{(1-\alpha_k)^2}|\la Q^{-1/2}(m_2-m_1), \phi_k\ra|^2 
	\\
	&= ||(I-S)^{-1}Q^{-1/2}(m_2 - m_1)||^2 < \infty.
	\end{align*}
	Thus $g \in \L^2(\H,\mu)$ and
	\begin{align*}
	||g_N - g||^2_{\L^2(\H, \mu)} = \sum_{k=N+1}^{\infty}\frac{1}{(1-\alpha_k)^2}|\la Q^{-1/2}(m_2-m_1), \phi_k\ra|^2 
	\approach 0 \;\text{as}\; N \approach \infty.
	\end{align*}
	Since $\mu$ is a probability measure, by H\"older's Inequality, we have
	$||g_N - g||_{\L^1(\H, \mu)} \leq \sqrt{\mu(\H)}||g_N - g||_{\L^2(\H, \mu)} = ||g_N - g||_{\L^2(\H, \mu)} \approach 0$
	as $N \approach \infty$.
\qed \end{proof}

The following is a direct generalization of Claim 1 in Proposition 1.2.8 in \cite{DaPrato:PDEHilbert}.
\begin{lemma}
	\label{lemma:bx-inner}
	Let $\{\phi_k\}_{k=1}^{\infty}$ be any orthonormal basis in $\H$.
	For any $b \in \H$,
	\begin{align}
	\la b, x-m\ra = \sum_{k=1}^{\infty}\la Q^{1/2}b, \phi_k\ra W_{\phi_k}(x) \;\;\;\Ncal(m,Q)\;\;\text{a.e.},
	\end{align}
	where the series converges in $\L^2(\H, \Ncal(m,Q))$.
\end{lemma}

\begin{proof}
[\textbf{
	of Theorem \ref{theorem:radon-nikodym-infinite}}]
By Lemmas \ref{lemma:fN-convergence-L2mu} and \ref{lemma:gN-convergence-L2mu}, 
the series $\sum_{k=1}^{\infty}\Phi_k$ converges in $\L^1(\H,\mu)$ and $\L^2(\H,\mu)$. 
By Proposition \ref{proposition:L1mu-dnu-dmu}, 
$s(x) = \exp\left[-\frac{1}{2}\sum_{k=1}^{\infty}\Phi_k(x)\right] \in \L^1(\H, \mu)$, with
$\int_{\H}s(x)d\mu(x) = \exp\left[\frac{1}{2}||(I-S)^{-1/2}Q^{-1/2}(m_2 - m_1)||^2\right]$.
Define 
\begin{align*}
\rho(x) = \exp\left[-\frac{1}{2}\sum_{k=1}^{\infty}\Phi_k(x)\right]\exp\left[-\frac{1}{2}||(I-S)^{-1/2}Q^{-1/2}(m_2 - m_1)||^2\right].
\end{align*}
Then $\rho$ is nonnegative and satisfies
$\rho \in \L^1(\H, \mu)$, with
$\int_{\H}\rho(x)d\mu(x) = 1$,
i.e.
$\rho\mu$ is a probability measure on $\Bsc(\H)$.
To show that the two measures $\rho \mu$ and $\nu$ coincide,
we show that the corresponding characteristic functions are identical, that is
\begin{align*}
\int_{\H}\exp(i \la h, x\ra)\rho(x)d\mu(x) =  \int_{\H}\exp(i \la h, x\ra)d\nu(x) \;\;\forall h \in \H.
\end{align*}
For the measure $\nu$, the characteristic function is given by
\begin{align*}
\int_{\H}\exp(i\la h, x\ra)\nu(dx) &= \int_{\H}\exp(i\la h, x\ra)\Ncal(m_2,R)(dx) 
\\
&= \exp\left(i \la m_2, h\ra -\frac{1}{2}\la Rh, h\ra\right), \;\; h \in \H.
\end{align*}
To compute the characteristic function for $\rho\mu$, we first note that by Lemma \ref{lemma:bx-inner},
\begin{align*}
\la h, x\ra = \la h, m_1\ra + \sum_{k=1}^{\infty}\la Q^{1/2}h, \phi_k\ra W_{\phi_k}(x) \;\;\;\Ncal(m_1,Q)\;\;\text{a.e.} \; \forall h \in \H.
\end{align*}
Let $b_k = \frac{Q^{-1/2}(m_2-m_1)}{(1-\alpha_k)}$.
The characteristic function for $\rho\mu$ is given by
{
\begin{align}
\label{equation:character-inter-1}
&\int_{\H}\exp(i \la h, x\ra)\rho(x)d\mu(x) 
\\
& = \exp\left[-\frac{1}{2}||(I-S)^{-1/2}Q^{-1/2}(m_2 - m_1)||^2\right]\int_{\H}\exp(i \la h, x\ra)s(x)d\mu(x)
\nonumber
\\
&= \exp(i\la h, m_1\ra)\exp\left[-\frac{1}{2}||(I-S)^{-1/2}Q^{-1/2}(m_2 - m_1)||^2\right]
\nonumber \;\; \times \;\;
\\
&\int_{\H}\exp\left\{-\frac{1}{2}\sum_{k=1}^{\infty}\left[\frac{\alpha_k}{1-\alpha_k}W^2_{\phi_k}(x) -2\left\la iQ^{1/2}h +b_k, \phi_k\right\ra W_{\phi_k}(x) + \log(1-\alpha_k)\right]\right\}d\mu(x).
\nonumber
\end{align}
}
For each $k \in \N$, we have by Proposition \ref{proposition:exponential-quadratic-whitenoise}, using the fact that $||\phi_k||=1$,
{
\begin{align*}
&\int_{\H}\exp\left(-\frac{1}{2}\left[\frac{\alpha_k}{1-\alpha_k}W^2_{\phi_k}(x) -2\left\la i Q^{1/2}h + b_k, \phi_k\right\ra W_{\phi_k}(x) + \log(1-\alpha_k)\right]\right) d\mu(x)
\\
& = \frac{1}{(1-\alpha_k)^{1/2}}\int_{\H}\exp\left[-\frac{1}{2}\frac{\alpha_k}{1-\alpha_k}W^2_{\phi_k}(x) + \left\la i Q^{1/2}h +b_k, \phi_k\right\ra W_{\phi_k}(x) \right] \Ncal(m_1,Q)(dx)
\\
& = \frac{1}{(1-\alpha_k)^{1/2}}\left\{(1-\alpha_k)^{1/2}\exp\left[\frac{1}{2}(1-\alpha_k)\left\la i Q^{1/2}h +b_k, \phi_k\right\ra^2\right]\right\}
\\
& = \exp\left[\frac{1}{2}(1-\alpha_k)\left\la i Q^{1/2}h +b_k, \phi_k\right\ra^2\right]
\\
& = \exp\left[-\frac{1}{2}(1-\alpha_k)\la Q^{1/2}h,\phi_k\ra^2 +i\la Q^{1/2}h,\phi_k\ra\la Q^{-1/2}(m_2-m_1), \phi_k\ra\right]
\\
&\times \exp\left[ \frac{\la Q^{-1/2}(m_2-m_1), \phi_k\ra^2}{2(1-\alpha_k)}\right].
\end{align*}
}
For each $N \in \N$,  for the function $s_N(x) = \exp\left[-\frac{1}{2}\sum_{k=1}^N\Phi_k(x)\right]$, we have by the independence of the $W_{\phi_k}$'s that
{\small
\begin{align*}
& \int_{\H}\exp(i \la h, x\ra) s_N(x)d\mu(x)
\\
& = \int_{\H}\exp\left\{-\frac{1}{2}\sum_{k=1}^N\left[\frac{\alpha_k}{1-\alpha_k}W^2_{\phi_k}(x) -2\left\la i Q^{1/2}h + b_k, \phi_k\right\ra W_{\phi_k}(x) + \log(1-\alpha_k)\right]\right\} d\mu(x)
\\
& = \prod_{k=1}^N\int_{\H}\exp\left(-\frac{1}{2}\left[\frac{\alpha_k}{1-\alpha_k}W^2_{\phi_k}(x) -2\left\la i Q^{1/2}h +b_k, \phi_k\right\ra W_{\phi_k}(x) + \log(1-\alpha_k)\right]\right) d\mu(x)
\\
& = \prod_{k=1}^N\exp\left[-\frac{1}{2}(1-\alpha_k)\la Q^{1/2}h,\phi_k\ra^2 +i\la Q^{1/2}h,\phi_k\ra\la Q^{-1/2}(m_2-m_1), \phi_k\ra \right]
\\
& \times \prod_{k=1}^N\exp\left[\frac{\la Q^{-1/2}(m_2-m_1), \phi_k\ra^2}{2(1-\alpha_k)}\right]
\\
& = \exp\left[-\frac{1}{2}\sum_{k=1}^N(1-\alpha_k)\la Q^{1/2}h,\phi_k\ra^2 \right]\exp\left[i\sum_{k=1}^N\la Q^{1/2}h,\phi_k\ra\la Q^{-1/2}(m_2-m_1), \phi_k\ra\right]
\\
& \times \exp\left[\sum_{k=1}^N\frac{\la Q^{-1/2}(m_2-m_1), \phi_k\ra^2}{2(1-\alpha_k)}\right].
\end{align*}
}
By Proposition \ref{proposition:L1mu-dnu-dmu},  there exists $p > 1$ is such that $I+(p-1)S > 0$. Then for $s_N(x) = \exp\left[-\frac{1}{2}\sum_{k=1}^N\Phi_k(x)\right]$, the sequence $\{s_N^q\}_{N \in \N}$ is uniformly integrable on $(\Bsc(\H),\mu)$ for all $1 \leq q < p$. Thus the sequence $\{\exp(i q\la h,x\ra)s_N^q(x)\}_{N \in \N}$ is also uniformly integrable for $1 \leq q < p$. 
For
$q=1$, Vitali's Convergence Theorem gives
\begin{align*}
&\int_{\H}\exp(i \la h, x\ra)s(x)d\mu(x) = \int_{\H}\lim_{N \approach \infty}[\exp(i\la h,x\ra)s_N(x)]d\mu(x)
\\
& = \lim_{N \approach \infty}\int_{\H}\exp(i \la h, x\ra)s_N(x)d\mu(x)
\\
&= \exp\left[-\frac{1}{2}\sum_{k=1}^{\infty}(1-\alpha_k)\la Q^{1/2}h,\phi_k\ra^2 \right]\exp\left[i\sum_{k=1}^{\infty}\la Q^{1/2}h,\phi_k\ra\la Q^{-1/2}(m_2-m_1), \phi_k\ra\right]
\\
& \times \exp\left[\sum_{k=1}^{\infty}\frac{\la Q^{-1/2}(m_2-m_1), \phi_k\ra^2}{2(1-\alpha_k)}\right].
\end{align*} 
For the first exponent,  we have for any $h \in \H$,
\begin{align*}
&\sum_{k=1}^{\infty}(1-\alpha_k)\la Q^{1/2}h, \phi_k\ra^2 = \la Q^{1/2} h, [\sum_{k=1}^{\infty}(1-\alpha_k)\phi_k \otimes \phi_k]Q^{1/2}h\ra 
\\
&= \la Q^{1/2} h, (I-S)Q^{1/2}h\ra = \la h, Q^{1/2}(I-S)Q^{1/2}h\ra = \la h, Rh\ra.
\end{align*}
For the second exponent, since $\{\phi_k\}_{k \in \N}$ is an orthonormal basis for $\H$, we have
\begin{align*}
&\sum_{k=1}^{\infty}\la Q^{1/2}h,\phi_k\ra\la Q^{-1/2}(m_2-m_1), \phi_k\ra = \la Q^{1/2}h, Q^{-1/2}(m_2-m_1)\ra
= \la h, m_2 - m_1\ra.
\end{align*}
For the third exponent, 
\begin{align*}
&\sum_{k=1}^{\infty}\frac{\la Q^{-1/2}(m_2-m_1), \phi_k\ra^2}{2(1-\alpha_k)}
\\
&
= \frac{1}{2}\la Q^{-1/2}(m_2-m_1), \left[\sum_{k=1}^{\infty}\frac{1}{1-\alpha_k}\phi_k \otimes \phi_k\right] Q^{-1/2}(m_2-m_1)\ra
\\
& = \frac{1}{2} \la Q^{-1/2}(m_2-m_1), (I-S)^{-1}Q^{-1/2}(m_2-m_1)\ra
=  \frac{1}{2}||(I-S)^{-1/2}Q^{-1/2}(m_2-m_1)||^2.
\end{align*}
Thus, taking the limit as $N \approach \infty$, we obtain
\begin{align*}
&\int_{\H}\exp(i \la h, x\ra)s(x)d\mu(x) 
\\
& = \exp\left[-\frac{1}{2}\la Rh,h\ra +i\la h, m_2 - m_1\ra+  \frac{1}{2}||(I-S)^{-1/2}Q^{-1/2}(m_2-m_1)||^2\right].
\end{align*}
Combining this with Eq.~(\ref{equation:character-inter-1}), we obtain the desired equality, namely
\begin{align*}
\int_{\H}\exp(i \la h, x\ra)\rho(x)d\mu(x) = \exp\left(i\la h,m_2\ra -\frac{1}{2}\la Rh,h\ra\right).
\end{align*}
\qed \end{proof}

\begin{lemma} 
	\label{lemma:trace-class-L1mu}
	Assume that $S$ is trace class. Then $\sum_{k=1}^{\infty}\frac{\alpha_k}{1-\alpha_k}W^2_{\phi_k} \in \L^1(\H,\mu)$ and the following limit holds
	in the $\L^1(\H, \mu)$ sense,
	\begin{align*}
	\lim_{N \approach \infty}\sum_{k=1}^{\infty}\frac{\alpha_k}{1-\alpha_k}W^2_{P_N\phi_k}= \sum_{k=1}^{\infty}\frac{\alpha_k}{1-\alpha_k}W^2_{\phi_k}.
	\end{align*}
\end{lemma}
\begin{proof}We first note that, since $S$ is trace class, $S(I-S)^{-1}$ is also trace class and
	\begin{align*}
&||S(I-S)^{-1}||_{\tr} = \sum_{j=1}^{\infty}\la e_j, |S(I-S)^{-1}|e_j\ra 
= \sum_{j=1}^{\infty}\la e_j, \sum_{k=1}^{\infty}\left|\frac{\alpha_k}{1-\alpha_k}\right| (\phi_k \otimes \phi_k) e_j\ra
\\
& = \sum_{j=1}^{\infty}\sum_{k=1}^{\infty}\left|\frac{\alpha_k}{1-\alpha_k}\right| \la \phi_k, e_j\ra^2 
= \sum_{k=1}^{\infty}\left|\frac{\alpha_k}{1-\alpha_k}\right|< \infty
\imply \sum_{j=N+1}^{\infty}\sum_{k=1}^{\infty}\left|\frac{\alpha_k}{1-\alpha_k}\right| \la \phi_k, e_j\ra^2 \approach 0
	\end{align*}
	as $N \approach \infty$. Furthermore,
	\begin{align*}
	&\sum_{k=1}^{\infty}\left|\frac{\alpha_k}{1-\alpha_k}\right|\int_{\H}W^2_{\phi_k}(x)\mu(dx) = \sum_{k=1}^{\infty}\left|\frac{\alpha_k}{1-\alpha_k}\right|\;||W^2_{\phi_k}||_{\L^2(\H, \mu)}
	= \sum_{k=1}^{\infty}\left|\frac{\alpha_k}{1-\alpha_k}\right| < \infty,
	\end{align*}
	showing that $\sum_{k=1}^{\infty}\frac{\alpha_k}{1-\alpha_k}W^2_{\phi_k} \in \L^1(\H,\mu)$. By H\"older's Inequality, we have
	\begin{align*}
	&\int_{\H}|W^2_{P_N\phi_k}(x)-W^2_{\phi_k}(x)|\mu(dx) = \int_{\H}|W_{P_N\phi_k}(x)-W_{\phi_k}(x)|\;
	|W_{P_N\phi_k}(x)+W_{\phi_k}(x)|\mu(dx)
	\\
	& \leq |||W_{P_N\phi_k}-W_{\phi_k}||_{\L^2(\H,\mu)}\; [||W_{P_N\phi_k}||_{\L^2(\H,\mu)})+||W_{\phi_k}||_{\L^2(\H,\mu)}]
	\\
	& \leq 2 |||W_{P_N\phi_k}-W_{\phi_k}||_{\L^2(\H,\mu)} = 2||P_N\phi_k - \phi_k||,
	\end{align*}
	since $||W_{P_N\phi_k}||_{\L^2(\H,\mu)} = ||P_N\phi_k|| \leq ||\phi_k|| = ||W_{\phi_k}||_{\L^2(\H, \mu)} = 1$.
It follows that
{\small
	\begin{align*}
	&\int_{\H}\left| \sum_{k=1}^{\infty}\frac{\alpha_k}{1-\alpha_k}W^2_{P_N\phi_k}(x) - \sum_{k=1}^{\infty}\frac{\alpha_k}{1-\alpha_k}W^2_{\phi_k}(x)\right|\mu(dx)
	\\
	&\leq \int_{\H}\sum_{k=1}^{\infty}\left|\frac{\alpha_k}{1-\alpha_k}\right| \; 
	|W^2_{P_N\phi_k}(x)-W^2_{\phi_k}(x)|\mu(dx)\leq 2\sum_{k=1}^{\infty}\left|\frac{\alpha_k}{1-\alpha_k}\right|\;||P_N\phi_k - \phi_k||
	\\
	& \leq 2 \left(\sum_{k=1}^{\infty}\left|\frac{\alpha_k}{1-\alpha_k}\right|\right)^{1/2}\left(\sum_{k=1}^{\infty}\left|\frac{\alpha_k}{1-\alpha_k}\right|\;||P_N\phi_k - \phi_k||^2\right)^{1/2}
	\\
	& =  2 \left(\sum_{k=1}^{\infty}\left|\frac{\alpha_k}{1-\alpha_k}\right|\right)^{1/2}\left(\sum_{k=1}^{\infty}\left|\frac{\alpha_k}{1-\alpha_k}\right|\sum_{j=N+1}^{\infty}\la \phi_k, e_j\ra^2\right)^{1/2}
	\approach 0 \;\;\;
	\end{align*}
}
\text{as $N \approach \infty$}.
\qed \end{proof}

\begin{lemma}
	\label{lemma:traceclass-L1mu-2}
	Let $b \in \H$ be arbitrary. Then $\sum_{k=1}^{\infty}\frac{1}{1-\alpha_k} W_{\phi_k}\la b, \phi_k \ra
	\in \L^2(\H,\mu)$ and the following limit holds in the $\L^2(\H,\mu)$ sense
	\begin{align}
	\lim_{N \approach \infty}\sum_{k=1}^{\infty}\frac{1}{1-\alpha_k} W_{P_N\phi_k}\la b, \phi_k \ra
	= \sum_{k=1}^{\infty}\frac{1}{1-\alpha_k} W_{\phi_k}\la b, \phi_k \ra.
	\end{align}	
\end{lemma}
\begin{proof}
	Since the sequence $\{W_{\phi_k}\}_{k \in \N}$ is orthonormal in $\L^2(\H,\mu)$, we have
	\begin{align*}
	\left\|\sum_{k=1}^{\infty}\frac{1}{1-\alpha_k} W_{\phi_k}\la b, \phi_k \ra\right\|^2_{\L^2(\H,\mu)}
	= \sum_{k=1}^{\infty}\frac{(\la b, \phi_k\ra)^2}{(1-\alpha_k)^2}  = ||(I-S)^{-1}b||^2 < \infty.
	\end{align*}
	Thus $\sum_{k=1}^{\infty}\frac{1}{1-\alpha_k} W_{\phi_k}\la b, \phi_k \ra \in \L^2(\H,\mu)$. Furthermore,
%
{\small
	\begin{align*}
		&\left\|\sum_{k=1}^{\infty}\frac{1}{1-\alpha_k} W_{P_N\phi_k}\la b, \phi_k \ra
	- \sum_{k=1}^{\infty}\frac{1}{1-\alpha_k} W_{\phi_k}\la b, \phi_k \ra\right\|_{\L^2(\H,\mu)}^2
	\\
	& =\sum_{j,k=1}^{\infty}\frac{\la b, \phi_k\ra}{1-\alpha_k}\frac{\la b, \phi_j\ra}{1-\alpha_j}\la (W_{P_N\phi_k} - W_{\phi_k}),
	(W_{P_N\phi_j} - W_{\phi_j})\ra_{\L^2(\H,\mu)}
	\\
	& 
	= \sum_{j,k=1}^{\infty}\frac{\la b, \phi_k\ra}{1-\alpha_k}\frac{\la b, \phi_j\ra}{1-\alpha_j}
	\la P_N\phi_k - \phi_k, P_N\phi_j - \phi_j\ra = \left\|\sum_{k=1}^{\infty}\frac{\la b, \phi_k\ra}{1-\alpha_k}(P_N\phi_k - \phi_k)\right\|^2
	\\
	& = \left\|\sum_{k=1}^{\infty}\frac{\la b, \phi_k\ra}{1-\alpha_k}
	\sum_{j=N+1}^{\infty}\la \phi_k, e_j\ra e_j\right\|^2 = \sum_{j=N+1}^{\infty}
	\left(\sum_{k=1}^{\infty}\frac{\la b, \phi_k\ra \la \phi_k, e_j\ra}{1-\alpha_k}\right)^2
	\\
	& 
	= \sum_{j=N+1}^{\infty}\la (I-S)^{-1}b, e_j\ra^2 \approach  0 \;\;\text{as $N \approach \infty$}.
	\end{align*}
}
	This gives the desired convergence.
\qed \end{proof}

\begin{proof}
	\textbf{(of Corollary \ref{corollary:radon-nikodym-traceclass})}
	When $S$ is trace class, the Fredholm determinant $\det(I-S)$ is well-defined and for $I-S$ strictly positive, we have
	\begin{align*}
	\exp\left(-\frac{1}{2}\sum_{k=1}^{\infty}\log(1-\alpha_k)\right) = \exp\left(-\frac{1}{2}\log\det(I-S)\right) = 
	\det(I-S)^{-1/2}.
	\end{align*}
	From the spectral decomposition
	$S(I-S)^{-1} = \sum_{k=1}^{\infty}\frac{\alpha_k}{1-\alpha_k} \phi_k \otimes \phi_k$,
	we have $\forall N \in \N$,
	\begin{align*}
	& \la Q^{-1/2}P_N(x-m_1), S(I-S)^{-1}Q^{-1/2}P_N(x-m_1)\ra 
	\\
	&= \sum_{k=1}^{\infty}\frac{\alpha_k}{1-\alpha_k}\la Q^{-1/2}P_N(x-m_1), \phi_k\ra^2
	\\
	& = \sum_{k=1}^{\infty}\frac{\alpha_k}{1-\alpha_k}\la x- m_1, Q^{-1/2}P_N\phi_k\ra^2 
	= \sum_{k=1}^{\infty}\frac{\alpha_k}{1-\alpha_k}W^2_{P_N\phi_k}(x).
	\end{align*}
	By Lemma \ref{lemma:trace-class-L1mu}, 
	taking limit as $N \approach \infty$ gives, where the limit is in $\L^1(\H,\mu)$, 
	\begin{align*}
	&\sum_{k=1}^{\infty}\frac{\alpha_k}{1-\alpha_k}W^2_{\phi_k}(x)
	= \lim_{N \approach \infty} \sum_{k=1}^{\infty}\frac{\alpha_k}{1-\alpha_k}W^2_{P_N\phi_k}(x)
	\\
	& = \lim_{N \approach \infty}\la Q^{-1/2}P_N(x-m_1), S(I-S)^{-1}Q^{-1/2}P_N(x-m_1)\ra
	\\
	&\doteq\la Q^{-1/2}(x-m_1), S(I-S)^{-1}Q^{-1/2}(x-m_1)\ra.
	\end{align*}
	Similarly,
	\begin{align*}
	&\la Q^{-1/2}P_N(x-m_1), (I-S)^{-1}Q^{-1/2}(m_2 - m_1)\ra
	\\
	& = \sum_{k=1}^{\infty}\frac{1}{1-\alpha_k}\la Q^{-1/2}P_N(x-m_1), \phi_k\ra
	\la Q^{-1/2}(m_2 - m_1), \phi_k\ra
	\\
	&= \sum_{k=1}^{\infty}\frac{1}{1-\alpha_k}\la x-m_1, Q^{-1/2}P_N\phi_k\ra\ \la Q^{-1/2}(m_2-m_1), \phi_k\ra
	\\
	&=\sum_{k=1}^{\infty}\frac{1}{1-\alpha_k}W_{P_N\phi_k}(x)\la Q^{-1/2}(m_2-m_1), \phi_k\ra.
	\end{align*}
	By Lemma \ref{lemma:traceclass-L1mu-2}, taking limit as $N \approach \infty$, we have
	\begin{align*}
	&\sum_{k=1}^{\infty}\frac{1}{1-\alpha_k}W_{\phi_k}(x)\la Q^{-1/2}(m_2-m_1), \phi_k\ra
	\\
	& = \lim_{N \approach \infty}\la Q^{-1/2}P_N(x-m_1), (I-S)^{-1}Q^{-1/2}(m_2 - m_1)\ra
	 \\
	 &\doteq \la Q^{-1/2}(x-m_1), (I-S)^{-1}Q^{-1/2}(m_2 - m_1)\ra.
	\end{align*}
	Combining these, we obtain
	\begin{align*}
	\sum_{k=1}^{\infty}\Phi_k(x) &= \la Q^{-1/2}(x-m_1), S(I-S)^{-1}Q^{-1/2}(x-m_1)\ra 
	\\
	&-2\la Q^{-1/2}(x-m_1), (I-S)^{-1}Q^{-1/2}(m_2 - m_1)\ra + \log\det(I-S).
	\end{align*}
\qed \end{proof}

\subsection{Exact Kullback-Leibler divergences}
\label{section:KL}
We now derive the explicit expression for the exact Kullback-Leibler 
divergence between two equivalent Gaussian measures on $\H$.
In the following, let $\mu = \Ncal(m_1, Q)$ and $W:\H \mapto \L^2(\H,\mu)$ be the white noise mapping induced by $\mu$.
Let $\nu = \Ncal(m_2,R)$, with
$m_2 - m_1 \in \Im(Q^{1/2})$ and $R = Q^{1/2}(I-S)Q^{1/2}$ for some $S \in \Sym(\H) \cap \HS(\H)$.
Let $\{\alpha_k\}_{k \in \N}$ be the eigenvalues of $S$ with corresponding orthonormal eigenvectors $\{\phi_k\}_{k \in \N}$.

\begin{theorem} 
\label{theorem:KL-infinite}
Let $\mu = \Ncal(m_1, Q)$ and $\nu = \Ncal(m_2,R)$, with $m_2 - m_1 \in \Im(Q^{1/2})$ and $R = Q^{1/2}(I-S)Q^{1/2}$, where $S \in \Sym(\H) \cap \HS(\H)$. Then
\begin{align}
D_{\KL}(\nu ||\mu) = \frac{1}{2}||Q^{-1/2}(m_2 - m_1)||^2 - \frac{1}{2}\log\dettwo(I-S).
\end{align}
If, furthermore, $S$ is trace class, then 
\begin{align}
D_{\KL}(\nu || \mu) = \frac{1}{2}||Q^{-1/2}(m_2 - m_1)||^2-\frac{1}{2}\log\det(I-S) - \frac{1}{2}\trace(S).
\end{align}
\end{theorem}
For $m_1 = m_2 = 0$, we obtain the Kullback-Leibler divergence given in \cite{Michalek:1999},
which also derived the R\'enyi divergences between two zero-mean Gaussian measures with different covariance operators.


\begin{lemma}
\label{lemma:L2nu}
For any $z, z_1, z_2\in \H$, 
{\small
\begin{align}
\int_{\H}W_z(x)d\nu(x) &= \la Q^{-1/2}(m_2 - m_1), z\ra,
\\
\la W_{z_1}, W_{z_2}\ra_{\L^2(\H, \nu)} &= \la (I-S)z_1, z_2\ra 
+ \la Q^{-1/2}(m_2-m_1), z_1\ra\la Q^{-1/2}(m_2 - m_1), z_2\ra.
\end{align}
}
In particular, 
for the orthonormal eigenvectors $\{\phi_k\}_{k \in \N}$ of $S$, 
{\small
\begin{align}
\la W_{\phi_j}, W_{\phi_k}\ra_{\L^2(\H,\nu)} &= (1-\alpha_k) \delta_{jk} + \la Q^{-1/2}(m_2-m_1), \phi_k\ra\la Q^{-1/2}(m_2 - m_1), \phi_j\ra,
\\
||W_{\phi_k}||^2_{\L^2(\H,\nu)} &= (1-\alpha_k) + |\la  Q^{-1/2}(m_2-m_1), \phi_k\ra|^2.
\end{align}
}
\end{lemma}
\begin{proof}
For $z \in Q^{1/2}(\H)$, which is dense in  $\H$, we have
\begin{align*}
& \int_{\H}W_z(x)d\nu(x) = \int_{\H}\la x- m_1, Q^{-1/2}z\ra\Ncal(m_2,R)(dx) 
\\
& = \int_{\H}\la x-m_2 + m_2 - m_1, Q^{-1/2}z\ra\Ncal(m_2,R)(dx) = \la m_2 - m_1, Q^{-1/2}z\ra
\\
&= \la Q^{-1/2}(m_2 - m_1), z\ra.
\end{align*}
By a limiting argument, we then have
$\int_{\H}W_z(x)d\nu(x) = \la Q^{-1/2}(m_2 - m_1), z\ra$
 $\forall z \in \H$.

For any pair $(z_1, z_2) \in Q^{1/2}(\H)$, we have
{\small
\begin{align*}
&\la W_{z_1}, W_{z_2}\ra_{\L^2(\H, \nu)} = 
\int_{\H}\la x - m_1, Q^{-1/2}z_1\ra\la x-m_1, Q^{-1/2}z_2\ra \Ncal(m_2, R)(dx)
\\
& =\int_{\H}\la x - m_2 + m_2-m_1, Q^{-1/2}z_1\ra\la x-m_2 + m_2 - m_1, Q^{-1/2}z_2\ra \Ncal(m_2, R)(dx)
\\
& =\int_{\H}[\la x - m_2, Q^{-1/2}z_1\ra +\la m_2-m_1, Q^{-1/2}z_1\ra]
\\
\;\; &\times [\la x-m_2, Q^{-1/2}z_2\ra + \la m_2 - m_1, Q^{-1/2}z_2\ra]\Ncal(m_2, R)(dx)
\\
& = \int_{\H}\la x - m_2, Q^{-1/2}z_1\ra\la x-m_2, Q^{-1/2}z_2\ra\Ncal(m_2, R)(dx) 
\\
&+ \la m_2-m_1, Q^{-1/2}z_1\ra\int_{\H}\la x-m_2, Q^{-1/2}z_2\ra\Ncal(m_2, R)(dx) 
\\
&+ \la m_2 - m_1, Q^{-1/2}z_2\ra\int_{\H}\la x - m_2, Q^{-1/2}z_1\ra \Ncal(m_2, R)(dx)
\\
&+ \la m_2-m_1, Q^{-1/2}z_1\ra\la m_2 - m_1, Q^{-1/2}z_2\ra
\end{align*}
\begin{align*}
&=\la RQ^{-1/2}z_1, Q^{-1/2}z_2\ra +\la m_2-m_1, Q^{-1/2}z_1\ra\la m_2 - m_1, Q^{-1/2}z_2\ra
\\
&= \la Q^{-1/2}RQ^{-1/2}z_1,z_2\ra + \la Q^{-1/2}(m_2-m_1), z_1\ra\la Q^{-1/2}(m_2 - m_1), z_2\ra
\\
& = \la (I-S)z_1, z_2\ra +  \la Q^{-1/2}(m_2-m_1), z_1\ra\la Q^{-1/2}(m_2 - m_1), z_2\ra.
\end{align*}
} 
Since $Q^{1/2}(\H)$ is dense in $\H$, by a limiting argument, we have $\forall z_1, z_2 \in \H$,
\begin{align*}
\la W_{z_1}, W_{z_2}\ra_{\L^2(\H, \nu)} &= \la (I-S)z_1, z_2\ra 
\\
&+ \la Q^{-1/2}(m_2-m_1), z_1\ra\la Q^{-1/2}(m_2 - m_1), z_2\ra.
\end{align*}
For the orthonormal basis $\{\phi_k\}_{k \in \N}$, we have
$\la (I-S)\phi_j, \phi_k\ra = (1-\alpha_k) \delta_{jk}$, \text{so that}
\begin{align*}
&\la W_{\phi_j}, W_{\phi_k}\ra_{\L^2(\H,\nu)} = (1-\alpha_k) \delta_{jk} + \la Q^{-1/2}(m_2-m_1), \phi_k\ra\la Q^{-1/2}(m_2 - m_1), \phi_j\ra.
\end{align*}
\qed \end{proof}

\begin{proposition}
\label{proposition:gN-convergence-L2nu}
Consider the functions
\begin{align}
g_N &= \sum_{k=1}^N\frac{1}{1-\alpha_k}\la Q^{-1/2}(m_2-m_1), \phi_k\ra W_{\phi_k}, \;\; N \in \N,
\\
g &= \sum_{k=1}^{\infty}\frac{1}{1-\alpha_k}\la Q^{-1/2}(m_2-m_1), \phi_k\ra W_{\phi_k}.
\end{align}
Then  $g \in \L^1(\H,\nu)$,$g\in \L^2(\H,\nu)$, and
\begin{align}
&\lim_{N\approach \infty}||g_N - g||_{\L^2(\H, \nu)} = 0,
\;\;\;
\lim_{N \approach \infty}||g_N - g||_{\L^1(\H, \nu)} = 0.
\end{align}
\end{proposition}
\begin{proof}
Using the expression for $\la W_{\phi_j}, W_{\phi_k}\ra_{\L^2(\H,\nu)}$
from Lemma \ref{lemma:L2nu}, we have
{\small
\begin{align*}
&||g_N||^2_{\L^2(\H, \nu)} = \sum_{k,j=1}^N\frac{\la Q^{-1/2}(m_2-m_1), \phi_k\ra\la Q^{-1/2}(m_2-m_1), \phi_j\ra }{(1-\alpha_k)(1-\alpha_j)}\la W_{\phi_k}, W_{\phi_j}\ra_{\L^2(\H,\nu)}
\\
& = \sum_{k=1}^N\frac{\la Q^{-1/2}(m_2-m_1), \phi_k\ra^2}{1-\alpha_k} + \sum_{k,j=1}^N\frac{\la Q^{-1/2}(m_2-m_1), \phi_k\ra^2\la Q^{-1/2}(m_2-m_1), \phi_j\ra^2 }{(1-\alpha_k)(1-\alpha_j)}
\\
& = \sum_{k=1}^N\frac{\la Q^{-1/2}(m_2-m_1), \phi_k\ra^2}{1-\alpha_k} + \left(\sum_{k=1}^N\frac{\la Q^{-1/2}(m_2-m_1), \phi_k\ra^2}{1-\alpha_k}\right)^2
\\
& \leq \sum_{k=1}^{\infty}\frac{\la Q^{-1/2}(m_2-m_1), \phi_k\ra^2}{1-\alpha_k} + \left(\sum_{k=1}^{\infty}\frac{\la Q^{-1/2}(m_2-m_1), \phi_k\ra^2}{1-\alpha_k}\right)^2
\\
& = ||(I-S)^{-1/2}Q^{-1/2}(m_2 - m_1)||^2  + ||(I-S)^{-1/2}Q^{-1/2}(m_2 - m_1)||^4 
= ||g||^2_{\L^2(\H,\nu)} < \infty.
\end{align*}
}
Furthermore, the expression for $||g_N||^2_{L^2(\H, \nu)}$ shows that
\begin{align*}
||g_N - g||^2_{\L^2(\H, \nu)} &= \sum_{k=N+1}^{\infty}\frac{\la Q^{-1/2}(m_2-m_1), \phi_k\ra^2}{1-\alpha_k} 
\\
&+ \left(\sum_{k=N+1}^{\infty}\frac{\la Q^{-1/2}(m_2-m_1), \phi_k\ra^2}{1-\alpha_k}\right)^2 \approach 0 \;\;\text{as}\;\; N \approach \infty.
\end{align*}
By the H\"older Inequality, we obtain $||g||_{\L^1(\H,\nu)} \leq ||g||_{\L^2(\H,\nu)} < \infty$ and
$||g_N - g||_{\L^1(\H,\nu)} \leq ||g_N - g||_{\L^2(\H, \nu)} \approach 0 \;\;\text{as}\;\; N \approach \infty$.
\qed \end{proof}

\begin{lemma}
\label{lemma:Gaussian-integral-3}
For any pair $a_1, a_2 \in \H$,
\begin{align*}
&\int_{\H}\la x-m_1,a_1\ra^2\la x-m_1,a_2\ra^2\Ncal(m_2,R)(dx) = \la a_1,Ra_1\ra\la a_2, Ra_2\ra + 2\la a_1, Ra_2\ra^2
\nonumber
\\
&\;\;\; \; + \la m_2 - m_1, a_2\ra^2\la a_1,Ra_1\ra + 4\la m_2 - m_1,a_1\ra\la m_2 - m_1,a_2\ra \la a_1, Ra_2\ra 
\nonumber
\\
& \;\; \; \; + \la m_2-m_1,a_1\ra^2\la a_2,Ra_2\ra + \la m_2 - m_1,a_1\ra^2 \la m_2 -m_1,a_2\ra^2.
\end{align*}
In particular, for $a_1 = a_2 = a$,
\begin{align*}
\int_{\H}\la x-m_1, a\ra^4 \Ncal(m_2,R) &= 3\la a, Ra\ra^2 
+ 6\la m_2 - m_1, a\ra^2 \la a, Ra\ra 
+  \la m_2 - m_1, a\ra^4.
\end{align*}
\end{lemma}
\begin{proof}
We have, by symmetry, for any $a \in \H$,
$\int_{\H}\la x - m_2, a\ra\Ncal(m_2,R) = \int_{\H}\la x- m_2, a\ra^3\Ncal(m_2,R) = 0$.
Also, by Lemma \ref{lemma:Gauss-integral-2}, for any $a, b \in \H$,
\begin{align*}
\int_{\H}\la x- m_2, a\ra^2\la x- m_2, b\ra \Ncal(m_2,R)(dx) = 0.
\end{align*}
Thus
for any pair $a_1, a_2 \in \H$, by Lemma \ref{lemma:Gaussian-integral-1},
{\small
\begin{align*}
&\int_{\H}\la x-m_1, a_1\ra^2\la x- m_1, a_2\ra^2\Ncal(m_2,R)(dx)
\\
&= \int_{\H}(\la x-m_2, a_1\ra + \la m_2 - m_1,a_1\ra)^2
(\la x-m_2, a_2\ra + \la m_2 - m_1,a_2\ra)^2\Ncal(m_2,R)(dx)
\\
& = \int_{\H}\la x-m_2,a_1\ra^2\la x-m_2,a_2\ra^2\Ncal(m_2,R)(dx) + \la m_2 - m_1, a_2\ra^2\int_{\H}\la x - m_2, a_1\ra^2\Ncal(m_2,R)(dx)
\\
& + 4(\la m_2 - m_1,a_1\ra\la m_2 - m_1,a_2\ra\int_{\H}\la x-m_2,a_1\ra\la x-m_2, a_2\ra \Ncal(m_2,R)(dx)
\\
&+ \la m_2-m_1,a_1\ra^2\int_{\H}\la x-m_2, a_2\ra^2\Ncal(m_2,R)(dx) + \la m_2 - m_1,a_1\ra^2 \la m_2 -m_1,a_2\ra^2
\\
& = \la a_1,Ra_1\ra\la a_2, Ra_2\ra + 2\la a_1, Ra_2\ra^2
+ \la m_2 - m_1, a_2\ra^2\la a_1,Ra_1\ra 
\\
&+ 4(\la m_2 - m_1,a_1\ra\la m_2 - m_1,a_2\ra \la a_1, Ra_2\ra + \la m_2-m_1,a_1\ra^2\la a_2,Ra_2\ra
\\
&+ \la m_2 - m_1,a_1\ra^2 \la m_2 -m_1,a_2\ra^2.
\end{align*}
}
This completes the proof.
\qed \end{proof}

\begin{lemma}
\label{lemma:RQ-equivalent-1}
For any pair $a,b \in Q^{1/2}(\H)$,
\begin{align}
\la Q^{-1/2}a, RQ^{-1/2}b\ra = \la a, (I-S)b\ra.
\end{align}
\end{lemma}
\begin{proof}
By assumption, there exist $c,d \in \H$ such that $a = Q^{1/2}c$, $b = Q^{1/2}d$. Thus
\begin{align*}
\la Q^{-1/2}a, RQ^{-1/2}b\ra &= \la c, Rd\ra = \la c, Q^{1/2}(I-S)Q^{1/2}d\ra 
\\
&= \la Q^{1/2}c, (I-S)Q^{1/2}d\ra =
\la a, (I-S)b\ra.
\end{align*}
\qed \end{proof}

\begin{lemma}
\label{lemma:norm-W4-2}
For any $a,b \in \H$,
{\small
\begin{align*}
&\int_{\H}W_a^2(x)W_b^2(x)\Ncal(m_2,R)(dx)
\\
& = \la a, (I-S)a\ra\la b, (I-S)b\ra + 2 \la a, (I-S)b\ra^2 + \la Q^{-1/2}(m_2 - m_1), b\ra^2\la a, (I-S)a\ra
\nonumber
\\
& + 4\la Q^{-1/2}(m_2-m_1), a\ra\la Q^{-1/2}(m_2-m_1),b\ra \la a, (I-S)b\ra
\nonumber
\\
& + \la Q^{-1/2}(m_2 - m_1), a\ra^2\la b, (I-S)b\ra + \la Q^{-1/2}(m_2 - m_1),a\ra^2 \la Q^{-1/2}(m_2 -m_1),b\ra^2.
\nonumber
\end{align*}
}
In particular, for $a=b$,
\begin{align}
\int_{\H}W_{a}^4(x)\Ncal(m_2,R)(dx) &=  3\la a, (I-S)a\ra^2 + 6\la Q^{-1/2}(m_2 - m_1), a\ra^2 \la a, (I-S)a\ra
\nonumber
\\
&+ \la Q^{-1/2}(m_2 - m_1),a\ra^4.
\end{align}
For two orthonormal eigenvectors $\phi_k, \phi_j$ of $S$,
\begin{align}
&\int_{\H}W_{\phi_k}^2(x)W_{\phi_j}^2(x)\Ncal(m_2,R)(dx) = (1-\alpha_k)(1-\alpha_j) + 2(1-\alpha_k)^2\delta_{jk}
\\
& + (1-\alpha_k)\la Q^{-1/2}(m_2 - m_1), \phi_j\ra^2 + (1-\alpha_j)\la Q^{-1/2}(m_2 - m_1), \phi_k\ra^2 
\nonumber
\\
&+ 4(1-\alpha_k)\delta_{jk}\la Q^{-1/2}(m_2-m_1), \phi_k\ra\la Q^{-1/2}(m_2-m_1),\phi_j\ra
\nonumber
\\
& + \la Q^{-1/2}(m_2 - m_1),\phi_k\ra^2 \la Q^{-1/2}(m_2 -m_1), \phi_j\ra^2.
\nonumber
\\
&\int_{\H}W_{\phi_k}^4(x)\Ncal(m_2,R)(dx) = 3(1-\alpha_k)^2
\\
& + 6(1-\alpha_k)\la Q^{-1/2}(m_2 - m_1), \phi_k\ra^2 
 + \la Q^{-1/2}(m_2 - m_1),\phi_k\ra^4.
\nonumber
\end{align}
\end{lemma}

\begin{proof}
For $a,b \in Q^{1/2}(\H)$, by Lemmas \ref{lemma:Gaussian-integral-3} and \ref{lemma:RQ-equivalent-1}, we have
{\small 
\begin{align*}
&\int_{\H}W_a^2(x)W_b^2(x)\Ncal(m_2,R)(dx) 
= \int_{\H}\la x-m_1, Q^{-1/2}a\ra^2\la x-m_1, Q^{-1/2}b\ra^2\Ncal(m_2,R)(dx)
\\
&=
\la Q^{-1/2}a,RQ^{-1/2}a\ra\la Q^{-1/2}b, RQ^{-1/2}b\ra + 2\la Q^{-1/2}a, RQ^{-1/2}b\ra^2
\nonumber
\\
& + \la m_2 - m_1, Q^{-1/2}b\ra^2\la Q^{-1/2}a,RQ^{-1/2}a\ra 
\\
&+ 4(\la m_2 - m_1,Q^{-1/2}a\ra\la m_2 - m_1,Q^{-1/2}b\ra \la Q^{-1/2}a, RQ^{-1/2}b\ra 
\nonumber
\\
&+ \la m_2-m_1,Q^{-1/2}a\ra^2\la Q^{-1/2}b,RQ^{-1/2}b\ra + \la m_2 - m_1,Q^{-1/2}a\ra^2 \la m_2 -m_1,Q^{-1/2}b\ra^2.
\\
& = \la a, (I-S)a\ra\la b, (I-S)b\ra + 2 \la a, (I-S)b\ra^2 + \la Q^{-1/2}(m_2 - m_1), b\ra^2\la a, (I-S)a\ra
\\
& + 4\la Q^{-1/2}(m_2-m_1), a\ra\la Q^{-1/2}(m_2-m_1),b\ra \la a, (I-S)b\ra
\\
& + \la Q^{-1/2}(m_2 - m_1), a\ra^2\la b, (I-S)b\ra + \la Q^{-1/2}(m_2 - m_1),a\ra^2 \la Q^{-1/2}(m_2 -m_1),b\ra^2.
\end{align*}
}
The general case $a,b \in \H$ then follows by a limiting argument.
\qed \end{proof}

\begin{proposition}
\label{proposition:orthonormal-L2nu}
The following functions are orthonormal in $\L^2(\H,\nu)$
\begin{align}
\left\{1,\frac{W_{\phi_k}^2 - [1-\alpha_k + \la Q^{-1/2}(m_2 -m_1), \phi_k\ra^2]}{\sqrt{2(1-\alpha_k)^2 + 4(1-\alpha_k) \la Q^{-1/2}(m_2 -m_1), \phi_k\ra^2}}\right\}_{k=1}^{\infty}.
\end{align}
\end{proposition}
\begin{proof} We have by Lemma \ref{lemma:L2nu} that
\begin{align*}
[1-\alpha_k + \la Q^{-1/2}(m_2 -m_1), \phi_k\ra^2] = \int_{\H}W_{\phi_k}^2(x)\Ncal(m_2,R).
\end{align*}
Thus the constant function $1$ is orthogonal to $W_{\phi_k}^2 - [1-\alpha_k + \la Q^{-1/2}(m_2 -m_1), \phi_k\ra^2]$.
By Lemma \ref{lemma:norm-W4-2}, for $k \neq j \in \N$,
\begin{align*}
&\int_{\H}\left[W^2_{\phi_k}(x) - \int_{\H}W_{\phi_k}^2(x)\Ncal(m_2,R)\right]\left[W^2_{\phi_j}(x) - \int_{\H}W_{\phi_j}^2(x)\Ncal(m_2,R)\right]\Ncal(m_2,R)(dx)
\\
& = \int_{\H}\W^2_{\phi_k}(x)W^2_{\phi_j}(x)\Ncal(m_2,R)(dx) - \left[\int_{\H}W_{\phi_k}^2(x)\Ncal(m_2,R)\right]
\left[\int_{\H}W_{\phi_j}^2(x)\Ncal(m_2,R)\right]
\\
&= (1-\alpha_k)(1-\alpha_j) 
+ (1-\alpha_k)\la Q^{-1/2}(m_2 - m_1), \phi_j\ra^2 + (1-\alpha_j)\la Q^{-1/2}(m_2 - m_1), \phi_k\ra^2 
\\
& + \la Q^{-1/2}(m_2 - m_1),\phi_k\ra^2 \la Q^{-1/2}(m_2 -m_1), \phi_j\ra^2
\\
& - [1-\alpha_k + \la Q^{-1/2}(m_2 -m_1), \phi_k\ra^2][1-\alpha_j + \la Q^{-1/2}(m_2 -m_1), \phi_j\ra^2] = 0,
\end{align*}
thus
the sequence $\{W^2_{\phi_k}(x) - \int_{\H}W_{\phi_k}^2(x)\Ncal(m_2,R)\}_{k \in \N}$ is orthogonal. 
By Lemma \ref{lemma:norm-W4-2},
\begin{align*}
&\int_{\H}\left(W_{\phi_k}^2 - [1-\alpha_k + \la Q^{-1/2}(m_2 -m_1), \phi_k\ra^2]\right)^2\Ncal(m_2,R)(dx)
\\
& = \int_{\H}W_{\phi_k}^4(x)\Ncal(m_2,R)(dx) - [1-\alpha_k + \la Q^{-1/2}(m_2 -m_1), \phi_k\ra^2]^2
\\
& = 3(1-\alpha_k)^2 + 6(1-\alpha_k)\la Q^{-1/2}(m_2 - m_1), \phi_k\ra^2 
\\
&+ \la Q^{-1/2}(m_2 - m_1),\phi_k\ra^4 - [1-\alpha_k + \la Q^{-1/2}(m_2 -m_1), \phi_k\ra^2]^2
\\
& = 2(1-\alpha_k)^2 + 4(1-\alpha_k) \la Q^{-1/2}(m_2 -m_1), \phi_k\ra^2.
\end{align*}
This gives the normalization constant for each term in the sequence.
\qed \end{proof}

\begin{proposition}
\label{proposition:fN-convergence-L2nu}
Consider the functions
\begin{align}
f_N = \sum_{k=1}^N\left[\frac{\alpha_k}{1-\alpha_k}W^2_{\phi_k} + \log(1-\alpha_k)\right],
f =  \sum_{k=1}^{\infty}\left[\frac{\alpha_k}{1-\alpha_k}W^2_{\phi_k} + \log(1-\alpha_k)\right].
\end{align}
Then $f \in \L^1(\H,\nu)$, $f \in \L^2(\H,\nu)$, and 
\begin{align}
\lim_{N \approach \infty} ||f_N - f||_{\L^2(\H,\nu)} = 0,
\;\;\;
\lim_{N \approach \infty} ||f_N - f||_{\L^1(\H,\nu)} = 0.
\end{align}
\end{proposition}
\begin{proof}
Let $a_k = [(1-\alpha_k)^2 + 2(1-\alpha_k) \la Q^{-1/2}(m_2 -m_1), \phi_k\ra^2]$, $b _k = [1-\alpha_k + \la Q^{-1/2}(m_2 -m_1), \phi_k\ra^2]$, then
\begin{align*}
\frac{\alpha_k}{1-\alpha_k}W_{\phi_k}^2 + \log(1-\alpha_k) = \frac{\alpha_k\sqrt{2a_k}}{1-\alpha_k}\frac{1}{\sqrt{2a_k}}[W_{\phi_k}^2 - b_k] + \frac{\alpha_k b_k}{1- \alpha_k} + \log(1-\alpha_k).
\end{align*}
Consider the series of constants
\begin{align*}
&\sum_{k=1}^{\infty}\left[\frac{\alpha_k b_k}{1- \alpha_k} + \log(1-\alpha_k)\right] 
= \sum_{k=1}^{\infty}\left[\frac{\alpha_k\la Q^{-1/2}(m_2 -m_1), \phi_k\ra^2}{1-\alpha_k} + \alpha_k + \log(1-\alpha_k)\right]
\\
& = \sum_{k=1}^{\infty}\left[\frac{\alpha_k\la Q^{-1/2}(m_2 -m_1), \phi_k\ra^2}{1-\alpha_k}\right]
 + \sum_{k=1}^{\infty}[\alpha_k + \log(1-\alpha_k)]
\\
& = \la S(I-S)^{-1}Q^{-1/2}(m_2 - m_1), Q^{-1/2}(m_2 - m_1)\ra + \log\dettwo(I-S) < \infty.
\end{align*}
Consider the functions
\begin{align*}
h_N = \sum_{k=1}^N\frac{\alpha_k\sqrt{2a_k}}{1-\alpha_k}\frac{1}{\sqrt{2a_k}}[W_{\phi_k}^2 - b_k], \;\; N \in \N,
\;
h = \sum_{k=1}^{\infty}\frac{\alpha_k\sqrt{2a_k}}{1-\alpha_k}\frac{1}{\sqrt{2a_k}}[W_{\phi_k}^2 - b_k].
\end{align*}
By Proposition \ref{proposition:orthonormal-L2nu} and the definition of $a_k$ above, we have
\begin{align*}
&||h||^2_{L^2(\H,\nu)} = 2\sum_{k=1}^{\infty}\frac{\alpha_k^2a_k}{(1-\alpha_k)^2}
 = 2\sum_{k=1}^{\infty}\alpha_k^2 + 4\sum_{k=1}^{\infty} \frac{\alpha_k^2\la Q^{-1/2}(m_2 -m_1), \phi_k\ra^2}{1-\alpha_k}
\\
& = 2||S||^2_{\HS} + 4||S(I-S)^{-1/2}Q^{-1/2}(m_2 - m_1)||^2 < \infty.
\end{align*}
Thus $h \in \L^2(\H,\nu)$. Furthermore, 
\begin{align*}
||h_N - h||^2_{\L^2(\H,\nu)} = 2\sum_{k=N+1}^{\infty}\alpha_k^2 + 4\sum_{k=N+1}^{\infty} \frac{\alpha_k^2\la Q^{-1/2}(m_2 -m_1), \phi_k\ra^2}{1-\alpha_k}  \approach 0
\end{align*}
as $N \approach \infty$. Thus it follows that $f \in \L^2(\H,\nu)$ and
$\lim_{N \approach \infty}||f_N - f||_{\L^2(\H,\nu)} = 0$.
Since $\nu$ is a probability measure on $\H$, it also follows that $f \in \L^1(\H,\nu)$ and that
$\lim_{N \approach \infty}||f_N - f||_{\L^1(\H,\nu)} = 0$.
This completes the proof.
\qed \end{proof}

\begin{proof}
	{\textbf{(of Theorem \ref{theorem:KL-infinite})}}
	By Theorem \ref{theorem:radon-nikodym-infinite}, 
	\begin{align*}
	&\log\left\{\frac{d\nu}{d\mu}(x)\right\} 
	=-\frac{1}{2}||(I-S)^{-1/2}Q^{-1/2}(m_2 - m_1)||^2
	\\
	& -\frac{1}{2}\sum_{k=1}^{\infty}\left[
	\frac{\alpha_k}{1-\alpha_k}W^2_{\phi_k}(x) - \frac{2}{1-\alpha_k}\la Q^{-1/2}(m_2-m_1), \phi_k\ra W_{\phi_k}(x)+ \log(1-\alpha_k)\right].
	\end{align*}
	For each $k \in \N$, by Lemma \ref{lemma:L2nu}, we obtain
	\begin{align*}
	&\int_{\H}\left[
	\frac{\alpha_k}{1-\alpha_k}W^2_{\phi_k}(x) - \frac{2}{1-\alpha_k}\la Q^{-1/2}(m_2-m_1), \phi_k\ra W_{\phi_k}(x)+ \log(1-\alpha_k)\right]d\nu(x)
	\\
	& = \frac{\alpha_k}{1-\alpha_k}[(1-\alpha_k)+ |\la Q^{-1/2}(m_2 - m_1), \phi_k\ra|^2]
	-\frac{2}{1-\alpha_k}|\la Q^{-1/2}(m_2 - m_1), \phi_k\ra|^2
	\\
	&\;\; + \log(1-\alpha_k)
	\\
	& = \alpha_k + \log(1-\alpha_k) - \left(1+\frac{1}{1-\alpha_k}\right)|\la Q^{-1/2}(m_2 - m_1), \phi_k\ra|^2.
	\end{align*}
	For each $N \in \N$, consider the function
	$r_N = f_N -2 g_N$, $r = f- 2g$, where
	\begin{align*}
	f_N &= \sum_{k=1}^N\left[\frac{\alpha_k}{1-\alpha_k}W^2_{\phi_k} + \log(1-\alpha_k)\right],
	\;
	f = \sum_{k=1}^N\left[\frac{\alpha_k}{1-\alpha_k}W^2_{\phi_k} + \log(1-\alpha_k)\right]
	\\
	g_N &= \sum_{k=1}^{\infty}\frac{1}{1-\alpha_k}\la Q^{-1/2}(m_2-m_1), \phi_k\ra W_{\phi_k}, 
	\\
	g &= \sum_{k=1}^{\infty}\frac{1}{1-\alpha_k}\la Q^{-1/2}(m_2-m_1), \phi_k\ra W_{\phi_k}.
	\end{align*}
	By Propositions \ref{proposition:fN-convergence-L2nu} and \ref{proposition:gN-convergence-L2nu}, we have
	$f \in \L^1(\H,\nu)$, $g \in \L^1(\H,\nu)$, and
	\begin{align*}
	\lim_{N \approach \infty}||f_N-f||_{\L^1(\H,\nu)} = 0,\;\; \lim_{N \approach \infty}||g_N - g||_{\L^1(\H,\nu)} = 0.
	\end{align*}
	It follows that
	$r \in \L^1(\H,\nu)$ and that
	$\lim_{N \approach \infty}||r_N - r||_{\L^1(\H,\nu)} = 0$.
	Therefore
	\begin{align*}
	&\int_{\H}r(x)d\nu(x) = \lim_{N \approach \infty}\int_{\H}r_N(x)d\nu(x) 
	\\
	&= \lim_{N \approach \infty}\sum_{k=1}^N\left[\alpha_k + \log(1-\alpha_k) - \left(1+\frac{1}{1-\alpha_k}\right)|\la Q^{-1/2}(m_2 - m_1), \phi_k\ra|^2\right]
	\\
	& = \sum_{k=1}^{\infty}[\alpha_k + \log(1-\alpha_k)] 
	-\sum_{k=1}^{\infty}\left(1+\frac{1}{1-\alpha_k}\right)|\la Q^{-1/2}(m_2 - m_1), \phi_k\ra|^2
	\\
	&=  \log\dettwo(I-S) - ||Q^{-1/2}(m_2 - m_1)||^2 - ||(I-S)^{-1/2} Q^{-1/2}(m_2 - m_1)||^2.
	\end{align*}
	Combining the last expression with the expression for $\log\left\{\frac{d\nu}{d\mu}(x)\right\}$, we obtain
	\begin{align*}
	&D_{\KL}(\nu ||\mu) =  \int_{\H}\log\left\{\frac{d\nu}{d\mu}(x)\right\}d\nu(x) 
	\\
	&= -\frac{1}{2}||(I-S)^{-1/2}Q^{-1/2}(m_2 - m_1)||^2 - \frac{1}{2}\int_{\H}r(x)d\nu(x)
	\\
	&=\frac{1}{2}||Q^{-1/2}(m_2 - m_1)||^2 - \frac{1}{2}\log\dettwo(I-S).
	\end{align*}
\qed \end{proof}

\begin{proof}
	[\textbf{
of Theorem \ref{theorem:limit-KL-gamma-0-inverse}}]
	Consider the formula
	\begin{align}
	C = C_0 - C_0A^{*}(AC_0A^{*} + \Gamma)^{-1}AC_0 = C_0 - C_0^{1/2}SC_0^{1/2},
	\end{align}
	where
	$S$ is given by
	$S = C_0^{1/2}A^{*}(AC_0A^{*} + \Gamma)^{-1}AC_0^{1/2} \in \Tr(\H)$.
	By Theorem \ref{theorem:KL-infinite},
	\begin{align*}
	D_{\KL}(\Ncal(m,C), \Ncal(m_0,C_0)) = \frac{1}{2}||C_0^{-1/2}(m-m_0)||^2 - \frac{1}{2}\log\det(I-S) - \frac{1}{2}\trace(S).
	\end{align*}	
	For the first term, since $m = m_0 + C_0A^{*}(AC_0A^{*} + \Gamma)^{-1}(y - Am_0)$,
	\begin{align*}
	&||C_0^{-1/2}(m-m_0)||^2 
	= \la C_0^{-1/2}(m-m_0), C_0^{-1/2}(m-m_0)\ra 
	\\
	&= \la C_0^{1/2}A^{*}(AC_0A^{*} + \Gamma)^{-1}(y - Am_0), C_0^{-1/2}(m-m_0)\ra 
	\\
	& = \la A^{*}(AC_0A^{*} + \Gamma)^{-1}(y - Am_0), m - m_0\ra
	=  \la (AC_0A^{*} + \Gamma)^{-1}(y - Am_0), A(m-m_0)\ra
	\\
	& = \la [\Gamma^{-1} - \Gamma^{-1}AC_0A^{*}(AC_0A^{*} + \Gamma)^{-1}](y - Am_0), A(m - m_0)\ra
	\\
	& = \la \Gamma^{-1}(y-Am_0)-\Gamma^{-1}A(m-m_0)  , A(m - m_0)\ra = \la \Gamma^{-1}(y - Am), A(m-m_0)\ra
	\\
	& = - \la m-m_0, A^{*}\Gamma^{-1}(Am-y)\ra.
	\end{align*}
	For the second and third terms, 
	\begin{align*}
	\trace(S) = \trace[C_0^{1/2}A^{*}(AC_0A^{*} + \Gamma)^{-1}AC_0^{1/2}] = \trace[AC_0A^{*}(AC_0A^{*} + \Gamma)^{-1}].
	\end{align*}
	From the expression
	$C = C_0 - C_0A^{*}(AC_0A^{*} + \Gamma)^{-1}AC_0$,
	we obtain
	\begin{align*}
	ACA^{*} &= AC_0A^{*} - AC_0A^{*}(AC_0A^{*} + \Gamma)^{-1}AC_0A^{*} 
	\\
	&=AC_0A^{*}[I - (AC_0A^{*} + \Gamma)^{-1}AC_0A^{*}] = AC_0A^{*}(AC_0A^{*} + \Gamma)^{-1}\Gamma.
	\end{align*}
	Thus we have
	\begin{align*}
	\trace(S) = \trace[AC_0A^{*}(AC_0A^{*} + \Gamma)^{-1}] = \trace[ACA^{*}\Gamma^{-1}].
	\end{align*}
	For the term $\log\det(I-S)$, we have
	\begin{align*}
	&\det(I-S) = \det[I-C_0^{1/2}A^{*}(AC_0A^{*} + \Gamma)^{-1}AC_0^{1/2}] = \det[I -AC_0A^{*} (AC_0A^{*} + \Gamma)^{-1}]
	\\
	& = \det[\Gamma (AC_0A^{*} + \Gamma)^{-1}],
	\end{align*}
	from which it follows that
	$\log\det(I-S) = \log\det(\Gamma) - \log\det(AC_0A^{*} + \Gamma)$.
	Combining
	$\log\det(I-S)$ and $\trace(S)$ with the first term gives the desired result.
\qed \end{proof}

\subsection{Exact R\'enyi divergences}
\label{section:Renyi}
In this section, we derive the exact formula for the R\'enyi divergences $D_{\Rrm,r}(\nu || \mu)$ between 
two equivalent Gaussian measures $\nu$ and $\mu$ on $\H$. We recall that the R\'enyi divergence between 
$\nu$ and $\mu$ is defined by
\begin{align}
&D_{\Rrm,r}(\nu || \mu) = -\frac{1}{r(1-r)}\log\int_{\H}\left\{\frac{d\nu}{d\mu}(x)\right\}^rd\mu(x).
\end{align}
\begin{theorem}
\label{theorem:Renyi-infinite}
Let $\mu = \Ncal(m_1,Q)$, $\nu = \Ncal(m_2,R)$, with $m_2 - m_1 \in \Im(Q^{1/2})$ and $R = Q^{1/2}(I-S)Q^{1/2}$,
$S \in \Sym(\H) \cap \HS(\H)$. The 
R\'enyi divergence of order $r$, $0 < r < 1$, between $\nu$ and $\mu$ is given by
\begin{align}
D_{\Rrm,r}(\nu || \mu) 
&= \frac{1}{2}||[I-(1-r)S]^{-1/2}Q^{-1/2}(m_2-m_1)||^2 
\nonumber
\\
&+\frac{1}{2r(1-r)}\log\det[(I-S)^{r-1}(I-(1-r)S)].
\end{align}
Furthermore,
\begin{align}
\lim_{r \approach 1^{-}}D_{\Rrm,r}(\nu ||\mu) &= \frac{1}{2}||Q^{-1/2}(m_2-m_1)||^2 - \frac{1}{2}\log\dettwo(I-S) 
= D_{\KL}(\nu ||\mu),
\\
\lim_{r \approach 0}D_{\Rrm,r}(\nu ||\mu) &= \frac{1}{2}||R^{-1/2}(m_1-m_2)||^2 - \frac{1}{2}\log\dettwo[(I-S)^{-1}] = D_{\KL}(\mu ||\nu).
\end{align}
\end{theorem}
\begin{proof}
[\textbf{
	of Theorem \ref{theorem:Renyi-infinite}}]
By Proposition \ref{proposition:L1mu-dnu-dmu}, there exists $p > 1$ such that $I+(p-1)S > 0$. 
Proposition \ref{proposition:L1mu-dnu-dmu} then implies that $\frac{d\nu}{d\mu} \in \L^q(\H,\mu)$ for all $q$ satisfying $0 < q < p$.
By definition of the R\'enyi divergence, we then have for $0 < r < 1$,
\begin{align*}
&D_{\Rrm,r}(\nu || \mu) = -\frac{1}{r(1-r)}\log\int_{\H}\left\{\frac{d\nu}{d\mu}(x)\right\}^rd\mu(x)
\\
& = \frac{1}{2(1-r)}||(I-S)^{-1/2}Q^{-1/2}(m_2-m_1)||^2 
-\frac{1}{r(1-r)}\log\int_{\H}\exp\left[-\frac{r}{2}\sum_{k=1}^{\infty}\Phi_k(x)\right]d\mu(x).
\end{align*}
By Proposition \ref{proposition:L1mu-dnu-dmu}, we have for $0 < r < 1$,
%
\begin{align*}
&\int_{\H}\exp\left[-\frac{r}{2}\sum_{k=1}^{\infty}\Phi_k(x)\right]d\mu(x)
= (\det[(I-S)^{r-1}(I+(r-1)S)])^{-1/2}
\\
& \;\; \times\exp\left(\frac{r^2}{2}||[(I-S)(I+(r-1)S)]^{-1/2}Q^{-1/2}(m_2-m_1)||^2\right).
\end{align*}
Thus it follows that
\begin{align*}
& D_{\Rrm,r}(\nu || \mu) = \frac{1}{2(1-r)}||(I-S)^{-1/2}Q^{-1/2}(m_2-m_1)||^2 
\\
&+\frac{1}{2r(1-r)}\log\det[(I-S)^{r-1}(I+(r-1)S)]
\\
&-\frac{r}{2(1-r)}||[(I-S)(I+(r-1)S)]^{-1/2}Q^{-1/2}(m_2-m_1)||^2.
\end{align*}
Let $c = Q^{-1/2}(m_2-m_1)$, then we have
\begin{align*}
&||(I-S)^{-1/2}c||^2 - r||[(I-S)(I+(r-1)S)]^{-1/2}c||^2 
\\
&= \la (I-S)^{-1}c,c\ra - r\la [(I-S)(I+(r-1)S)]^{-1}c,c\ra
\\
& = \la ([(I-S)^{-1} - r[(I-S)(I-(1-r)S)]^{-1})c, c\ra = (1-r)\la [I-(1-r)S]^{-1}c, c\ra 
\\
& = (1-r)|| [I-(1-r)S]^{-1/2}c||^2.
\end{align*}
Combining this with the previous expression, we obtain
\begin{align*}
D_{\Rrm,r}(\nu ||\mu)
&= \frac{1}{2}||[I-(1-r)S]^{-1/2}Q^{-1/2}(m_2-m_1)||^2 
\\
&+\frac{1}{2r(1-r)}\log\det[(I-S)^{r-1}(I-(1-r)S)].
\end{align*}
This completes the proof of the first part of the theorem.

We now compute
$\lim_{r \approach 1^{-}}D_{\Rrm,r}(\nu ||\mu)$.
Let $\{\alpha_k\}_{k\in \N}$ be the eigenvalues of $S$, then 
\begin{align*}
&\frac{1}{1-r}\log\det[(I-S)^{r-1}(I-(1-r)S)] = \frac{1}{1-r}\sum_{k=1}^{\infty}\log[(1-\alpha_k)^{r-1}(1-(1-r)\alpha_k)]
\\
& = \frac{1}{1-r}\sum_{k=1}^{\infty}[\log(1-(1-r)\alpha_k) - (1-r)\log(1-\alpha_k)].
\end{align*}
By Lemma \ref{lemma:log-inequality-3}, we have $\frac{1}{1-r}[\log(1-(1-r)\alpha_k) - (1-r)\log(1-\alpha_k)] \geq 0$ $\forall k \in \N$. Thus
by Lebesgue's Monotone Convergence Theorem,
\begin{align*}
& \lim_{r \approach 1^{-}}\frac{1}{1-r}\log\det[(I-S)^{r-1}(I-(1-r)S)]
\\
& = \lim_{r \approach 1^{-}}\sum_{k=1}^{\infty}\frac{1}{1-r}[\log(1-(1-r)\alpha_k) - (1-r)\log(1-\alpha_k)]
\\
& = \sum_{k=1}^{\infty}\lim_{r \approach 1^{-}}\frac{1}{1-r}[\log(1-(1-r)\alpha_k) - (1-r)\log(1-\alpha_k)]
\\
& = -\sum_{k=1}^{\infty}[\alpha_k + \log(1-\alpha_k)] \;\;\text{(by Lemma \ref{lemma:log-inequality-3})}
= -\log\dettwo(I-S).
\end{align*}
Combining this limit with the expression for $D_{\Rrm,r}(\nu ||\mu)$ above,
we obtain
\begin{align*}
&\lim_{r \approach 1^{-}}D_{\Rrm,r}(\nu ||\mu) = \frac{1}{2}||Q^{-1/2}(m_2 - m_1)||^2 - \frac{1}{2}\log\dettwo(I-S) = D_{\KL}(\nu ||\mu).
\end{align*}
Also by Lemma \ref{lemma:log-inequality-3} and Lebesgue's Monotone Convergence Theorem,
\begin{align*}
& \lim_{r \approach 0}\frac{1}{r}\log\det[(I-S)^{r-1}(I-(1-r)S)]
\\
 &=\sum_{k=1}^{\infty}\lim_{r \approach 0}\frac{1}{r}[\log(1-(1-r)\alpha_k) - (1-r)\log(1-\alpha_k)] = \sum_{k=1}^{\infty}[\frac{\alpha_k}{1-\alpha_k} + \log(1-\alpha_k)]
 \\
 & = -\log\prod_{k=1}^{\infty}(1-\alpha_k)^{-1}\exp(-\frac{\alpha_k}{1-\alpha_k}) =- \log\det[(I-S)^{-1}\exp(-\frac{S}{I-S})]
 \\
 & = -\log\dettwo[(I-S)^{-1}], \;\;\;\text{since $(I-S)^{-1} = I + \frac{S}{I-S}$}.
\end{align*}
From the proof of Theorem \ref{theorem:limit-Renyi-infinite-1}, we have for any $m \in \Im(Q^{1/2})$,
\begin{align*}
&\lim_{r \approach 0}||[I-(1-r)S]^{-1/2}Q^{-1/2}(m)||
= \lim_{r \approach 0}||[Q^{1/2}(I-(1-r)S)Q^{1/2}]^{-1/2}(m)||
\\
&= ||[Q^{1/2}(I-S)Q^{1/2}]^{-1/2}(m)|| = ||R^{-1/2}(m)||.
\end{align*}
Combining the previous two limits with the expression for $D_{\Rrm,r}(\nu ||\mu)$ above,
we obtain
\begin{align*}
&\lim_{r \approach 0}D_{\Rrm,r}(\nu ||\mu) = \frac{1}{2}||R^{-1/2}(m_1 - m_2)||^2 - \frac{1}{2}\log\dettwo[(I-S)^{-1}]
= D_{\KL}(\mu ||\nu).
\end{align*}
\qed \end{proof}

\subsection{Bhattacharyya and Hellinger distances}
\label{section:hellinger}
We now derive the explicit formulas for the 
Bhattacharyya and Hellinger distances between two equivalent Gaussian measures $\nu$ and $\mu$ on $\H$.
Recall that the Bhattacharyya distance is defined by
\begin{align}
D_{\Brm}(\nu ||\mu) = -\log\int_{\H}\sqrt{\frac{d\nu}{d\mu}(x)}d\mu(x) = \frac{1}{4}D_{\Rrm, 1/2}(\nu ||\mu).
\end{align}
The Hellinger distance $D_{\Hrm}(\nu ||\mu)$ between $\nu$ and $\mu$ is defined by
\begin{align}
D_{\Hrm}^2(\nu ||\mu) &=  \int_{\H}\left(1-\sqrt{\frac{d\nu}{d\mu}(x)}\right)^2 d\mu(x) = 2[1-\exp(-D_{\Brm}(\nu ||\mu))]
\\
& = 2 - 2\int_{\H}\sqrt{\frac{d\nu}{d\mu}(x)}d\mu(x).
\end{align}

\begin{corollary}
\label{corollary:Hellinger}
Let $\mu = \Ncal(m_1,Q)$ and $\nu = \Ncal(m_2,R)$ and $S \in \Sym(\H) \cap \HS(\H)$ be such that
$m_2 - m_1 \in \Im(Q^{1/2})$ and $R = Q^{1/2}(I-S)Q^{1/2}$.
The Bhattacharyya distance $D_{\Brm}(\nu||\mu)$ between $\nu$ and $\mu$ is then given by
\begin{align}
D_{\Brm}(\nu ||\mu) &= \frac{1}{8}||(I-\frac{1}{2}S)^{-1/2}Q^{-1/2}(m_2 - m_1)||^2 
\nonumber
\\
&+ \frac{1}{2}\log\det[(I-S)^{-1/2}(I-\frac{1}{2}S)].
\end{align}
The Hellinger distance $D_{\Hrm}(\nu||\mu)$ between $\nu$ and $\mu$ is  given by
\begin{align}
D_{\Hrm}^2(\nu ||\mu) &=  2\left[1 -  \frac{\exp\left(-\frac{1}{8}||(I-\frac{1}{2}S)^{-1/2}Q^{-1/2}(m_2 - m_1)||^2\right)}{\sqrt{\det[(I-S)^{-1/2}(I-\frac{1}{2}S)]}} \right].
\end{align}
\end{corollary}
\begin{proof}
[\textbf{
of Corollary \ref{corollary:Hellinger}}]
For the Bhattacharyya distance, we use the fact that $D_{\Brm}(\nu||\mu) = \frac{1}{4}D_{\Rrm,1/2}(\nu ||\mu)$ and Theorem 
\ref{theorem:Renyi-infinite} to obtain
\begin{align*}
D_{\Brm}(\nu ||\mu) = \frac{1}{4}D_{\Rrm, 1/2}(\nu ||\mu) &= \frac{1}{8}||(I-\frac{1}{2}S)^{-1/2}Q^{-1/2}(m_2 - m_1)||^2 
\\
&+ \frac{1}{2}\log\det[(I-S)^{-1/2}(I-\frac{1}{2}S)].
\end{align*}
The expression for $D_{\Hrm}(\nu ||\mu)$ then follows from
$D_{\Hrm}^2(\nu ||\mu)  = 2[1-\exp(-D_{\Brm}(\nu ||\mu))]$.
\qed \end{proof}


\begin{proof}
[\textbf{
of Theorems \ref{theorem:limit-KL-gamma-0} and \ref{theorem:limit-Renyi-infinite} and Corollary \ref{corollary:limit-Hellinger-infinite}}]
Theorems \ref{theorem:limit-KL-gamma-0} follows
from Theorem \ref{theorem:limit-KL-gamma-0-1} and Theorem
\ref{theorem:KL-infinite}.
Theorem \ref{theorem:limit-Renyi-infinite} follows from Theorem \ref{theorem:limit-Renyi-infinite-1} and Theorem 
\ref{theorem:Renyi-infinite}.
Corollary \ref{corollary:limit-Hellinger-infinite}
follows from Theorem \ref{theorem:limit-Renyi-infinite-1}
and Corollary \ref{corollary:Hellinger}.
\qed \end{proof}




\section{Miscellaneous technical results}
\label{section:misc-tech}

Let $\Ncal(m,Q)$ denote a Gaussian measure on $\H$ with mean $m$ and covariance operator $Q$.
Let $\{\lambda_k\}_{k=1}^{\infty}$ denote the set of eigenvalues of $Q$, 
with corresponding orthonormal eigenvectors $\{e_k\}_{k=1}^{\infty}$.

\begin{lemma}
	\label{lemma:Gaussian-integral-1}
	For any pair $a,b \in \H$,
	\begin{align}
	\int_{\H}\la x-m, a\ra^2\la x-m, b\ra^2\Ncal(m,Q)(dx) = \la a, Qa\ra\la b,Qb\ra + 2\la a,Qb\ra^2.
	\end{align}
	In particular, for $a = b$,
	$\int_{\H}\la x- m, a\ra^4\Ncal(m,Q)(dx) = 3\la a, Qa\ra^2$.
\end{lemma}
\begin{proof}
	It suffices to prove for 
	$m = 0$.
	We apply
	the following
	(\cite{Handbook:1972}, Formula 7.4.4)
	\begin{align}
	\int_{0}^{\infty}t^{2n}e^{-at^2}dt = \frac{\Gamma(n + \frac{1}{2})}{2a^{n + \frac{1}{2}}}, \;\; \Real(a) > 0.
	\end{align}
	Thus for any $\lambda > 0$,
	\begin{align*}
	\int_{\R}t^2\Ncal(0,\lambda)(dt) &= \frac{1}{\sqrt{2\pi \lambda}}\int_{-\infty}^{\infty}t^2e^{-\frac{t^2}{2\lambda}}dt = \lambda,
	\\
	\int_{\R}t^4\Ncal(0,\lambda)(dt) &= \frac{1}{\sqrt{2\pi \lambda}}\int_{-\infty}^{\infty}t^4e^{-\frac{t^2}{2\lambda}}dt = \frac{1}{\sqrt{2\pi\lambda}}
	\Gamma(2+\frac{1}{2})(2\lambda)^{2+\frac{1}{2}} = 3\lambda^2.
	\end{align*}
	Write $x = \sum_{k=1}^{\infty}x_ke_k$, $a = \sum_{k=1}^{\infty}a_ke_k$. By symmetry, we have
	\begin{align*}
	&\int_{\H}\la x, a\ra^2\la x,b\ra^2\Ncal(0,Q)(dx) = \int_{\H}(\sum_{k=1}^{\infty}a_jx_j)^2(\sum_{k=1}^{\infty}b_kx_k)^2\Ncal(0,Q)(dx)
	\\
	&=  \int_{\H}\left[\sum_{k=1}^{\infty}a_k^2b_k^2x_k^4 + \sum_{j\neq k =1}^{\infty}(a_j^2b_k^2 + 2a_ja_kb_jb_k)x_j^2x_k^2\right]\Ncal(0,Q)(dx)
	\\
	& = \sum_{k=1}^{\infty}a_k^2b_k^2\int_{\R}x_k^4\Ncal(0,\lambda_k)(dx_k) 
	\\
	& + \sum_{j \neq k =1}^{\infty}(a_j^2b_k^2 + 2a_ja_kb_jb_k)
	\left[\int_{\R}x_j^2\Ncal(0,\lambda_j)(dx_j)\right]
	\left[\int_{\R}x_k^2\Ncal(0,\lambda_k)(dx_k)\right]
	\\
	& = 3\sum_{k=1}^{\infty}a_k^2b_k^2\lambda_k^2 
	+ \sum_{j \neq k =1}^{\infty}(a_j^2b_k^2 + 2a_ja_kb_jb_k)\lambda_j\lambda_k
	\\
	& = \sum_{j,k=1}^{\infty}[a_j^2b_k^2 + 2a_ja_kb_jb_k]\lambda_j\lambda_k 
	= (\sum_{j=1}^{\infty}a_j^2\lambda_j)(\sum_{k=1}^{\infty}b_k^2\lambda_k) + 2(\sum_{j=1}^{\infty}a_jb_j\lambda_j)^2
	\\
	& = \la a, Qa\ra\la b,Qb\ra + 2\la a, Qb\ra^2.
	\end{align*}
\qed \end{proof}

\begin{lemma}
	\label{lemma:Gauss-integral-2}
	For any pair $a,b \in \H$,
	\begin{align}
	\int_{\H}\la x-m, a\ra^2\la x-m, b\ra\Ncal(m,Q)(dx) = 0.
	\end{align}
	In particular, for $a = b$,
	$\int_{\H}\la x - m, a\ra^3\Ncal(m,Q)(dx) = 0$.
\end{lemma}
\begin{proof}
	It suffices to prove for $m=0$. Write $x = \sum_{k=1}^{\infty}x_ke_k$, $a = \sum_{k=1}^{\infty}a_ke_k$, then
	\begin{align*}
	&\int_{\H}\la x,a\ra^2\la x,b\ra\Ncal(0,Q)(dx) = \int_{\H}(\sum_{j=1}^{\infty}a_jx_j)^2(\sum_{k=1}^{\infty}b_kx_k)\Ncal(0,Q)(dx)
	= 0,
	\end{align*}
	by symmetry, since each term in the integral contains either $x_j$ or $x_j^3$ $\forall k \in \N$.
\qed \end{proof}

\begin{lemma}
\label{lemma:inequality-log}
In all inequalities below, equality happens if and only if $x = 0$.
\begin{align}
&-[x + \log(1-x)] \geq 0 \;\; \forall x < 1,
\\
&-[x + \log(1-x)] \leq x^2 \;\; \forall x < \frac{1}{2},
\\
& 0 \leq x +(1-x)\log(1-x) \leq x^2 \;\; \forall x < 1.
\end{align}
\end{lemma}

\begin{lemma}
\label{lemma:inequality-log-2}
Let $p > 1$ be fixed. Then
\begin{align}
(p-1)\log(1-x) + \log[1+(p-1)x] &\leq 0,  &-\frac{1}{p-1} < x < 1,
\\
(p-1)\log(1-x) + \log[1+(p-1)x] &\geq -p(p-1)x^2, &-1/2 < x < 1/2.
\end{align}
\end{lemma}

\begin{lemma} 
	\label{lemma:log-inequality-3}
	Let $\alpha < 1$ be fixed. Then
	\begin{align}
	\frac{1}{1-r}[\log(1-(1-r)\alpha) - (1-r)\log(1-\alpha)] &\geq 0, \;\; 0 < r < 1,
	\\
	\lim_{r \approach 1^{-}} \frac{1}{1-r}[\log(1-(1-r)\alpha) - (1-r)\log(1-\alpha)] &= -[\alpha + \log(1-\alpha)].
	\\
	\lim_{r \approach 0}\frac{1}{r}[\log(1-(1-r)\alpha) - (1-r)\log(1-\alpha)] &=\frac{\alpha}{1-\alpha}
+\log(1-\alpha).
	\end{align}
\end{lemma}

\bibliographystyle{plain}
\bibliography{cite_RKHS}

\end{document}

%% file: style.tex
\newcommand{\mapto}{\ensuremath{\rightarrow}}

\newcommand{\approach}{\ensuremath{\rightarrow}}
\newcommand{\imply}{\ensuremath{\Rightarrow}}

\newcommand{\equivalent}{\ensuremath{\Longleftrightarrow}}

\newcommand{\A}{\mathcal{A}}

\newcommand{\X}{\mathcal{X}}

\newcommand{\N}{\mathbb{N}}
\newcommand{\R}{\mathbb{R}}

\newcommand{\Fcal}{\mathcal{F}}

\newcommand{\la}{\langle}
\newcommand{\ra}{\rangle}
\newcommand{\F}{\mathcal{F}}
\renewcommand{\H}{\mathcal{H}}
\renewcommand{\L}{\mathcal{L}}

\def\Im{{\rm Im}}

\def\trace{{\rm tr}}

\def\HS{{\rm HS}}

\def\Sym{{\rm Sym}}

\def\tr{{\rm tr}}

\def\Tr{{\rm Tr}}

\newcommand{\W}{\mathcal{W}}

\def\1{\mathbf{1}}

\newcommand{\detX}{{\rm det_X}}
\newcommand{\trX}{{\rm tr_X}}

\newcommand{\Ncal}{\mathcal{N}}

\newcommand{\Bsc}{\mathscr{B}}
\newcommand{\Csc}{\mathscr{C}}

\newcommand{\Ker}{{\rm Ker}}

\newcommand{\logdet}{{\rm logdet}}
\newcommand{\Real}{{\rm Re}}

\newcommand{\detn}{{\rm det_n}}
\newcommand{\detone}{{\rm det_1}}
\newcommand{\dettwo}{{\rm det_2}}

\newcommand{\Com}{{\rm Com}}

\newcommand{\Bcal}{\mathcal{B}}